\documentclass{article}
\usepackage{graphicx,color} 
\usepackage{amsmath,amssymb,amsthm,enumitem,mathtools,bm}
\usepackage{mathrsfs}
\usepackage[labelfont=bf,textfont=sf]{caption}
\usepackage{geometry}
\usepackage{marvosym}
\usepackage{tikz}
\usepackage{authblk}
\newcommand{\N}{\mathbb{N}}
\newcommand{\Z}{\mathbb{Z}}
\newcommand{\R}{\mathbb{R}}
\newcommand{\Q}{\mathbb{Q}}
\newcommand{\T}{\mathbb{T}}

\newcommand{\editfull}{\mathrm{d}_E}
\newcommand{\eps}{\varepsilon}

\newcommand{\lowdisc}{\mathcal{D}_{\textrm{low}}}
\newcommand{\leb}{\mathcal{L}}

\newcommand{\den}[3]{d_{#1}(#2,#3)}
\newcommand{\modone}[1]{\widehat{#1}} 

\newcommand{\expolimsup}{\overline{\Gamma}}
\newcommand{\expoliminf}{\underline{\Gamma}}
\newcommand{\orbitexpolimsup}{\overline{\Gamma}_{\mathcal{L}}}
\newcommand{\orbitexpoliminf}{\underline{\Gamma}_{\mathcal{L}}}
\newcommand{\ddimorbitexpolimsup}{\overline{\Gamma}_{\mathcal{L}^d}}
\newcommand{\ddimorbitexpoliminf}{\underline{\Gamma}_{\mathcal{L}^d}}

\makeatletter
\newtheorem*{rep@theorem}{\rep@title}
\newcommand{\newreptheorem}[2]{%
	\newenvironment{rep#1}[1]{%
		\def\rep@title{#2 \ref{##1}}%
		\begin{rep@theorem}}%
		{\end{rep@theorem}}}
\makeatother

\newtheorem{introtheorem}{Theorem}

\newtheorem{introcorollary}[introtheorem]{Corollary}
\newtheorem{theorem}{Theorem}[section]
\newreptheorem{theorem}{Theorem}
\newtheorem{proposition}[theorem]{Proposition}
\newtheorem{fact}[theorem]{Fact}
\newtheorem{corollary}[theorem]{Corollary}
\newtheorem{lemma}[theorem]{Lemma}
\theoremstyle{definition}
\newtheorem{definition}[theorem]{Definition}
\newtheorem{remark}[theorem]{Remark}
\newtheorem*{remarknonumber}{Remark}

\newtheorem{notation}[theorem]{Notation}

\newcommand{\warningtriangle}[4]{
	\fill[black!15, draw=black, thick]
	(#1, {#2 + 0.577*#3}) --
	({#1 - #3/2}, {#2 - 0.289*#3}) --
	({#1 + #3/2}, {#2 - 0.289*#3}) -- cycle;
	\node at (#1,#2) {\fontsize{#4}{0pt}\selectfont\Lightning};
}

\title{Edit distance exponents for irrational rotations}
\author[1]{Andrew Best}
\author[1]{Yuval Peres}
\affil[1]{Beijing Institute of Mathematical Sciences and Applications}
\date{\today}

\begin{document}
	\maketitle
	\begin{abstract}
		
		We study quantitative edit-distance asymptotics for symbolic codings of irrational rotations $x \mapsto x+\alpha$ on $\mathbb{T}$ in terms of the irrationality exponent $\mu(\alpha)$.
		For the binary coding determined by an interval $[0,\beta)$, let $d_N(x,y)$ be the edit distance between the codings of the orbits of the points $x$ and $y$ up to time $N-1$. For every $\alpha \notin \Q $, almost every $\beta \in (0,1)$, and almost every $x,y \in \T$, we determine the edit-distance exponents
		\smallskip 
		
		\noindent $ \displaystyle (*) \; \; \limsup_{N\to\infty}\frac{\log d_N(x,y)}{\log N}
		= \frac{\rho(\alpha)}{\rho(\alpha)+1} \; \;
		$
		and  $\; \; \displaystyle\liminf_{N\to\infty}\frac{\log d_N(x,y)}{\log N} = \min\Bigl\{\frac{1}{2},\frac{1}{\rho(\alpha)}\Bigr\}$,
		
		\smallskip
		
		\noindent where  $ \rho(\alpha):=\mu(\alpha) - 1$.  We find a surprising transition   at the golden mean $\varphi= \frac{1+\sqrt{5}}{2}$: If $  \rho(\alpha)  \le \varphi$,   then both  asymptotics in $(*)$  hold for all $\beta$. However, for every finite $\rho>\varphi$, there is an uncountable set of $\alpha$ with $\rho(\alpha)=\rho$ for which both edit-distance exponents are strictly smaller than the corresponding values in $(*)$  for uncountably many~$\beta$.
		
		We also derive consequences for aperiodic circle homeomorphisms and Sturmian sequences. For rotations of $\mathbb{T}^d$ coded by boxes, we prove that for almost every rotation vector, the common edit-distance exponent is $d/(d+1)$. 
	\end{abstract}

	\section{Introduction}
	
	Edit distance has been fundamental in computer science and automata theory since its introduction by Levenshtein in 1965 \cite{levenshtein}. Feldman \cite{feldman} and Katok \cite{katok} introduced it into ergodic theory in the 1970s.
	
	Roughly speaking, an aperiodic ergodic system is \textbf{loosely Kronecker}\footnote{Using terminology of Ratner \cite{ratner84}; Feldman called these systems ``loosely Bernoulli of zero entropy''.} if most orbits, when mapped to a subshift via a generating partition, are relatively close to each other in edit distance. Katok showed loosely Kronecker systems are exactly those which are Kakutani equivalent to an irrational rotation (see \cite{orw} for this attribution and a different proof of this result). Irrational rotations with different Diophantine properties belong to the same Kakutani-equivalence class, but the edit distances between their orbit codings can grow at different rates.
	
	Given an alphabet $\mathcal{A}$, let $\mathcal{A}^+ = \cup_{N\geq 1} \mathcal{A}^N$ be the set of all finite words in the symbols of $\mathcal{A}$. We define the \textbf{edit distance} $\editfull(w,\tilde{w})$ between two words $w,\tilde{w} \in \mathcal{A}^+$ to be half the minimal number of insertions and deletions of letters required to transform $w$ into $\tilde{w}$,  (Alternative definitions differ by a bounded factor.) For example,  $\editfull(0110,0101) = 1$. For two words $w$ and $\tilde{w}$ in $\mathcal A^N$, the length of a longest common subsequence is $N-\editfull(w,\tilde{w})$.
	
	We study the edit distance between codings of   irrational rotation orbits on the circle $\T \coloneq \R/\Z$, represented as   $[0,1)$ with addition modulo 1. For $\alpha \in \R\smallsetminus\Q$ and $\beta \in (0,1)$, let
	\begin{equation} \label{definition of xi} \xi_n(x) \ \coloneq \ \mathbf{1}_{[0,\beta)}(x+n\alpha \bmod 1)  \text{ for all } n \in \Z \text{ and } x \in \T  \,.
	\end{equation}
	For integers $i < j$, write $\xi_{[i,j)}(x)$ for the word $(\xi_\nu(x))_{\nu=i}^{j-1}  \in \{0,1\}^{j-i}$. For $N \in \N$ and   $x,y\in \T$,    write $\den N x y \coloneq \editfull(\xi_{[0,N)}(x),\xi_{[0,N)}(y))$  and $\mathcal{W}_N \coloneq \{ \xi_{[0,N)}(x) : x \in \T \}$. Define
	\[ \mathrm{diam}_E(\mathcal{W}_N) \ \coloneq \ \max\{ \editfull(w,\tilde{w}) : w,\tilde{w} \in \mathcal{W}_N\} \ = \ \max_{x,y \in \T} \den N x y \, .
	\]
	Our goal is to compute the edit-distance exponents of the diameter
	\begin{equation}\label{gamma defn}
		\expolimsup(\alpha,\beta) \ \coloneq \ \limsup_{N\to\infty}  \frac{\log \mathrm{diam}_E(\mathcal{W}_N)}{\log N} \quad \text{ and } \quad  \expoliminf(\alpha,\beta) \ \coloneq \ \liminf_{N\to\infty}  \frac{\log \mathrm{diam}_E(\mathcal{W}_N)}{\log N} \, 
	\end{equation}
	and the edit-distance exponents for typical   orbit pairs: 
	\begin{equation} \label{orbit defn} \orbitexpolimsup (\alpha,\beta) \ := \ \limsup_{N\to\infty} \frac{\log \den N x y}{\log N} \, \text{ a.e. \, and } \; \, \orbitexpoliminf(\alpha,\beta) \ := \ \liminf_{N\to\infty} \frac{\log \den N x y}{\log N} \,  \text{ a.e.}
	\end{equation}
	
	The limsup exponent in \eqref{orbit defn} is constant a.e.\   with respect to Lebesgue measure $\leb^2=\leb \times \leb$ on $\T^2$ since the function $\displaystyle (x,y) \mapsto \limsup_{N\to\infty} \frac{\log \den N x y}{\log N}$ is invariant under the ergodic $\Z^2$-action defined, for each $(n,n') \in \Z^2$, by $(x,y) \mapsto (x+n\alpha \bmod 1, y + n'\alpha \bmod 1)$ (see, e.g., Exercise 8.1.1 in \cite{einsiedlerward}).  
	The same applies to the liminf exponent.
	
	\bigskip
	
	Our main result relates the edit-distance exponents to the \textbf{irrationality exponent} $\mu(\alpha)$ of $\alpha \in \R$, which is the supremum of $\mu \in \R$ for which the inequality
	$0 < |\alpha - p/q| < q^{-\mu}$
	has infinitely many solutions in $p \in \Z$ and $q \in \N$; see, e.g., \cite{bugeaud}. Throughout, we work with 
	\[
	\rho(\alpha) \ \coloneq \ \mu(\alpha)-1 \, . \]
	Classical results of Dirichlet and Khintchine (see \cite[Theorem 32]{khintchine}), respectively, ensure that $\rho(\alpha)\ge 1$ for all $\alpha \in \R\smallsetminus\Q$ and that $\rho(\alpha)=1$ for almost all $\alpha$.
	
	\begin{introtheorem}\label{main thm 1}
		Let $\alpha \in \R\smallsetminus \Q$. Then, there is a full-measure set $G_\alpha \subset (0,1)$ such that for every $\beta \in G_\alpha$,
		\begin{equation}\label{display for thm 1}
			\expolimsup(\alpha,\beta) \ = \ \orbitexpolimsup(\alpha,\beta) \ = \ \frac{\rho(\alpha)}{\rho(\alpha)+1}    \, ,
		\end{equation}
		\begin{equation} \label{display for thm 2}
			\expoliminf(\alpha,\beta) \ = \ \frac{1}{2} \, , \quad \text{ and } \quad \orbitexpoliminf(\alpha,\beta) \ = \ \min\bigl\{\frac{1}{2},\frac{1}{\rho(\alpha)}\bigr\} \, ,
		\end{equation}
		where we understand $\frac{\rho(\alpha)}{\rho(\alpha)+1} = 1$ and $\frac{1}{\rho(\alpha)} = 0$ when $\rho(\alpha) = \infty$.
	\end{introtheorem}
	
	Thus, for $\rho(\alpha)>2$, the liminf exponents of the diameter and typical orbit pairs differ.  
	
	\medskip
	
	Our second main theorem describes for which values of $\rho(\alpha)$ there are exceptional $\beta$ in Theorem~\ref{main thm 1}. There are thresholds at the golden mean $\varphi \coloneq \frac{1+\sqrt{5}}{2}$ and at $2$:
	\begin{itemize}
		\item Up to $\varphi$, there are no exceptional $\beta$.
		\item Between $\varphi$ and $2$, rotations with the same irrationality exponent can behave differently: some have no exceptional $\beta$, while others have uncountably many.
		\item Above $2$, every rotation has exceptional $\beta$.
	\end{itemize}
	\begin{introtheorem}\label{main thm thresholds}
		\begin{enumerate}[label=\textup{(\alph*)}]
			\item Let $\alpha \in \R\smallsetminus \Q$. If $\rho(\alpha) \le \varphi$, then \eqref{display for thm 1} and \eqref{display for thm 2} hold for all $\beta \in (0,1)$.
			\item If $\varphi < \rho \leq 2$, then there are uncountable subsets $A_*(\rho)$ and $A^*(\rho)$ of $\{\alpha : \rho(\alpha)  = \rho\}$ such that
			\begin{itemize}
				\item[\textup{(b1)}] for every $\alpha \in A^*(\rho)$ and $\beta \in (0,1)$, the asymptotics \eqref{display for thm 1} and \eqref{display for thm 2} hold;
				\item[\textup{(b2)}] for every $\alpha \in A_*(\rho)$ there are uncountably many $\beta$ with 
				\begin{equation} \label{goldenexcept1}
					\expolimsup(\alpha,\beta) \ = \ \orbitexpolimsup(\alpha,\beta) \ < \ \frac{\rho}{\rho+1} \quad \text{ and } \quad
					\expoliminf(\alpha,\beta) \ = \ \orbitexpoliminf(\alpha,\beta) \ < \ \frac{1}{2} \, .
				\end{equation}
			\end{itemize}
			\item For every $\alpha$ with $\rho(\alpha) > 2$,  there is a perfect set of $\beta$ with $\expolimsup(\alpha,\beta) \leq 2/3 < \frac{\rho(\alpha)}{\rho(\alpha)+1}$ and a dense $G_\delta$ set of $\beta$ with $\expoliminf(\alpha,\beta) \leq 1/\rho(\alpha) < 1/2$.
		\end{enumerate}
	\end{introtheorem}

	\noindent\textbf{Remarks.} \begin{enumerate}
		\item  For all $\alpha \in \R \smallsetminus \Q$ and $\beta \in (0,1)$, each of the four edit exponents is at most the right-hand side of the respective equation in Theorem~\ref{main thm 1}. See Theorems~\ref{thm: upper bounds on expoliminf and expolimsup} and \ref{thm: orbitexpoliminf upper bound}. 
		\item A more precise version of \eqref{goldenexcept1} is given in Theorem~\ref{main thm golden more second part}.
		\item For $\alpha$ with $\rho(\alpha) = 1$ (which holds for almost every $\alpha$) and every $\beta \in (0,1)$,  all four exponents equal $\frac{1}{2}$. For   $\alpha$ which are badly approximable, Proposition~\ref{prop: badly approximable edit distance} gives the sharper two-sided bound $c_1\sqrt N \le \mathrm{diam}_E(\mathcal{W}_N) \le C_2 \sqrt N$  with constants depending on $\alpha$ and $\beta$.
	\end{enumerate}

	G\"{u}ting \cite{gueting} refined the Jarn\'{i}k--Besicovitch theorem (see \cite{jarnik} and \cite{besicovitch}) and proved that $\dim \{\alpha \in (0,1) : \rho(\alpha) = r\} = \frac{2}{r+1}$ for every $r \geq 1$, where $\dim$ denotes Hausdorff dimension. Consequently, Theorem~\ref{main thm 1} yields a   multifractal description of the $\limsup$ exponent.
	
	\begin{introcorollary}\label{main cor multifractal}
		For every $\gamma \in [1/2,1]$,
		\[
		\dim\{\alpha \in (0,1)\smallsetminus\Q: \, \expolimsup(\alpha,\beta) = \gamma \,  \text{ for almost every } \, \beta \in (0,1)\} \ = \ 2-2\gamma \, .
		\]
	\end{introcorollary}
	
	\noindent (For $\gamma=1$ the set in question is the set of Liouville numbers, of Hausdorff dimension zero; see, e.g., \cite[Theorem 2.4]{oxtoby}.)
	
	\subsection{Rotations of higher-dimensional tori}
	
	For $y \in \R^d$ or $\T^d$, write $\modone{y}$ for the unique element of $y + \Z^d$ in $[0,1)^d$, and consider the rotation $x \mapsto \modone{x+\alpha}$ of $\T^d$. The interval used to code orbits in \eqref{definition of xi} is replaced by a box $\prod_{i=1}^d [0, \beta_i)$; the notation is defined in detail in Section~\ref{sec: multidim}.
	
	\begin{introtheorem}\label{main thm 4}
		Let $d \in \N$. For almost every $\alpha \in \R^d$ and every $\beta\in(0,1)^d$,
		\[
		\expoliminf(\alpha,\beta) \ = \ \expolimsup(\alpha,\beta) \ = \ \ddimorbitexpoliminf(\alpha,\beta) \ = \ \ddimorbitexpolimsup(\alpha,\beta) \ = \ \frac{d}{d+1} \, .
		\]
	\end{introtheorem}
	
	The set of full measure in this theorem contains all rotations $\alpha$ for which the orbits have low discrepancy; see \eqref{eq:dlow} for the precise definition.
	\begin{remarknonumber}
		In Section~\ref{sec: multidim} we also prove that  $\expoliminf(\alpha,\beta) \le \frac{d}{d+1}$  and $\ddimorbitexpolimsup(\alpha,\beta) \geq 1/2$ for all $\beta \in (0,1)^d$ and all \textbf{non-singular} $\alpha \in \R^d$, i.e., $\alpha$ such that the orbit segments $\{\modone{j\alpha}\}_{0 \le j <q}$ are $cq^{-1/d}$-separated for some $c>0$ and infinitely many $q$. (The set of  singular vectors in $\R^d$ has Hausdorff dimension  strictly less than $d$; see Section~\ref{subsec: nonsing} for   references.)
	\end{remarknonumber}

	\subsection{Circle homeomorphisms and Sturmian words}
	
	Given an orientation-preserving aperiodic homeomorphism $f : \T \to \T$, write $\omega(f)$ for its (irrational) rotation number (see Section~\ref{sec: circle maps} for the definition). 
	Such a homeomorphism is conjugate to an irrational circle rotation if and only if it has a dense orbit. We deduce the following analogue of part of Theorem~\ref{main thm 1}. 
	
	Given $x \in \T$, an integer $n$, and $\eta \in (0,1)$, define $\xi_n^f(x) \coloneq \mathbf{1}_{[0,\eta)}(f^nx)$, where $f^n$ denotes iterated composition. For each $N \in \N$, define $\mathcal{W}_N^f \coloneq \{ \xi^f_{[0,N)}(x) : x \in \T \}$. 
	
	\begin{introcorollary}\label{intro corollary for circle maps}
		Let $f : \T \to \T$ be an orientation-preserving aperiodic homeomorphism and let $\nu$ be the unique $f$-invariant Borel probability measure on $\T$. Let $\rho \coloneq \rho(\omega(f))$. If $f$ has a dense orbit, then for $\nu$-almost every $\eta \in (0,1)$,
		\begin{equation} \label{result for circle maps}
			\limsup_{N\to\infty}  \frac{\log \mathrm{diam}_E(\mathcal{W}_N^f)}{\log N} \ = \ \frac{\rho}{\rho+1} \quad \text{ and } \quad
			\liminf_{N\to\infty}  \frac{\log \mathrm{diam}_E(\mathcal{W}_N^f)}{\log N} \ = \ \frac{1}{2}   \, .
		\end{equation}
		Moreover, if $f$ and $f^{-1}$ are diffeomorphisms of class $C^r$ with $r > 1+\rho$, then \eqref{result for circle maps} holds for Lebesgue-almost every $\eta \in \T$.
	\end{introcorollary}
	A similar argument yields analogues of all the preceding results, replacing $\mathcal L$ in $\orbitexpoliminf$ and $\orbitexpolimsup$ by the $f$-invariant measure $\nu$. 
	
	Our results also apply to Sturmian sequences, which are aperiodic sequences with minimal word complexity, as they can be represented using irrational rotations; see	Section~\ref{sec: sturmian}.

	\subsection{Related notions}
	There are several parameters used to analyze systems with zero entropy, among which are word complexity, studied by Ferenczi \cite{ferenczi}, sequence entropy, introduced by Kushnirenko \cite{kushnirenko}, and slow entropy, introduced by Katok and Thouvenot \cite{katokthouvenot}. See Kanigowski--Katok--Wei \cite{kkwsurvey} for a survey. While these notions do not involve edit distance, there is a connection, described below. Ratner in \cite{ratner1981} constructed a Kakutani-equivalence invariant via edit distance and showed that distinct cartesian powers of horocycle flows are inequivalent.
	
	Let $\mathcal A$ be an alphabet of size $a \ge 2$. As observed in \cite[p. 44]{orw}, every loosely Kronecker system must have zero entropy. Next, we note two quantitative versions of this statement that are folklore. Both use the following inequality: For every $u \in \mathcal A^n$ and integer $0 \le r \le n$,
	\begin{equation}\label{eq:ball}
		\#\{v\in\mathcal A^n:d_E(u,v)\le r\} \le \binom nr^2a^r \le \exp\!\left(r\left[2\log\frac{en}{r}+\log a\right]\right) \, .
	\end{equation}
	Indeed, every such $v$ shares a subsequence of length $n-r$ with $u$. Choose its positions in both words and then the remaining $r$ symbols of $v$.
	
	\begin{fact}[Word complexity]
		Let $ \mathcal{W}_n \subset {\mathcal A}^n$ and $D_n \coloneq \max\{d_E(u,v):u,v\in\mathcal{W}_n\}$. If $(**) \; |\mathcal{W}_n| \ge\exp(cn^\gamma)$, where $c>0$ and $0< \gamma \le 1$, then  
		$
		D_n \ \ge \ \frac{cn^\gamma}{2\log(en)+\log a} \, .
		$
		Consequently, if $(**)$ holds for all large $n$, then $\liminf_{n\to\infty}\frac{\log D_n}{\log n}\ge\gamma.$
	\end{fact}
	
	\begin{proof}
		Given $u\in\mathcal{W}_n$,  \eqref{eq:ball} yields
		\[
		cn^\gamma \le \log |\mathcal{W}_n|  \le D_n\left[2\log\frac{en}{D_n}+\log a\right] \le D_n[2\log(en)+\log a] \, . \qedhere
		\]
	\end{proof}
	\medskip
	\begin{fact}[Positive entropy]
		Let $\mu$ be an ergodic shift-invariant probability measure on ${\mathcal A}^{\N}$ with entropy $h=h_\mu(\sigma)>0$. Then for some $\delta>0$ and $(\mu\times\mu)$-almost every $(x,y)$, we have $d_n(x,y)>\delta n$ for  all sufficiently large $n$, where  $d_n(x,y) \coloneq \editfull(x_{[0,n)},y_{[0,n)})$.
	\end{fact}
	
	\begin{proof}
		Choose $\delta\in(0,1/2)$ so that $\delta[2\log(e/\delta)+\log a]<h/4$. By \eqref{eq:ball}, every edit ball of radius $\lfloor\delta n\rfloor$ contains at most $e^{hn/4}$ words. Let $G_n$ consist of the words $v\in\mathcal{A}^n$ whose cylinders $[v] \coloneq \{y\in \mathcal{A}^{\N}\,:\, y_{[0,n)}=v\}$ satisfy $\mu([v])\le e^{-hn/2}$. For every $x$, summing the masses of these cylinders gives
		\[
		\mu\{y \in \mathcal{A}^{\N} \, :\, d_n(x,y)\le \lfloor\delta n\rfloor \, \text{ and } \,  y_{[0,n)}\in G_n\}\le e^{-hn/4}.
		\]
		Integrating in $x$ and applying Borel--Cantelli shows that these events occur only finitely often for almost every pair. The Shannon--McMillan--Breiman theorem gives $y_{[0,n)}\in G_n$ eventually for $\mu$-almost every $y$. Thus $d_n(x,y)> \lfloor \delta n \rfloor$, eventually almost surely.
	\end{proof}

	\section{Organization and Heuristics}
	
	\subsection{Organization}
	
	Section~\ref{sec: upper bounds} proves the universal upper bounds in Theorem~\ref{main thm 1}, while Sections~\ref{sec: lower bounds when mu alpha > 2} and~\ref{sec: orbitexpoliminf lower bound for most beta} prove the lower bounds. The proof of Theorem~\ref{main thm 1} is concluded in Section~\ref{subsec: proof of thm 1}.
	Section~\ref{sec: the case mu alpha equals 2} treats the case $\rho(\alpha)=1$, including the  estimates for badly approximable rotations.
	
	Theorem~\ref{main thm thresholds}(a) is restated as Theorem~\ref{main thm golden less} and proved in Section~\ref{sec: universal lower bounds}. Theorem~\ref{main thm thresholds}(b1) is restated as Theorem~\ref{main thm golden more first part} and proved in Section~\ref{sec: universal lower bounds}. Theorem~\ref{main thm thresholds}(b2) is proved in the following more precise form at the end of Section~\ref{subsec: upper bounds for steady convergent}: 
	
	\begin{theorem}\label{main thm golden more second part}
		Suppose that $\rho >\varphi$ is finite. Then there is an uncountable subset $A_*(\rho)$ of $\{\alpha: \rho(\alpha)=\rho\}$    so that for each $\alpha \in A_*(\rho)$, there are  uncountably many $\beta$ with
		\begin{equation} \label{special-alpha}
			\begin{aligned}
				\expolimsup(\alpha,\beta)& = \orbitexpolimsup(\alpha,\beta)
				= \frac{\rho^2}{2\rho^2-1} \, ,\\
				\expoliminf(\alpha,\beta)& = \orbitexpoliminf(\alpha,\beta)
				= \frac{\rho}{\rho^2+\rho-1} \, .
			\end{aligned}
		\end{equation}
	\end{theorem}
	The set $A_*(\rho)$ above consists of $\alpha \notin \Q$ such that the denominators $(q_m)$ of the convergents in their continued fraction expansion satisfy $q_{m+1}=q_m^{\rho+o(1)}$ as $m \to \infty$. 
	Section~\ref{sec: universal lower bounds} proves that the quantities on the right-hand side of \eqref{special-alpha} are universal lower bounds valid for every $\alpha$ and $\beta$. 
	
	Theorem~\ref{main thm thresholds}(c) is restated as Proposition~\ref{main prop high mu} and proved in Section~\ref{sec: exceptional sets}.

	The applications to circle maps and Sturmian sequences are in Sections~\ref{sec: circle maps} and~\ref{sec: sturmian}. Theorem~\ref{main thm 4} on typical toral rotations and the bounds for non-singular rotations are proved in Section~\ref{sec: multidim}. Section~\ref{sec: questions} collects open problems.
	
	\subsection{Heuristics for upper bounds}\label{subsec: methods}
	
	We obtain our upper bounds on edit distance  from two methods of  matching orbits. \medskip \\
	\noindent\textbf{Approach and follow.} Let $t(\eps)$ be the least positive integer for which the orbit segment $\{n\alpha\bmod1:0\leq n<t(\eps)\}$ is $\eps$-dense. Choose $0\leq\ell<t(\eps)$ so that $\tilde x=x+\ell\alpha\bmod1$ is within $\eps$ of $y$. The first edit shifts the coding of $x$ by $\ell$. Since rotation is an isometry, subsequent mismatches can occur only when an orbit visits one of two intervals of length $\eps$ adjacent to the coding boundaries $0$ and $\beta$; see Figure~\ref{fig: simple danger zones}. Thus the heuristic cost is $t(\eps)+O(N\eps)$. Turning this heuristic into an estimate requires uniform control of visits to these intervals, which is achieved via continued fractions.
	
	\begin{figure}[htbp]
		\centering
		\begin{tikzpicture}[scale=1.15]
			\draw[thick] (0,0) -- (10.0,0);
			\node at (9.99,0) {)};
			\node[below=0.3cm] at (10,0) {$1$};
			\draw[thick] (0,-0.15) -- (0,0.15);
			\node[below=0.3cm] at (0,0) {$0$};
			\warningtriangle{0.5}{0.55}{0.7}{14.4pt}
			\draw[thick] (1,-0.15) -- (1,0.15);
			\node[below=0.37cm] at (1,0) {$\eps$};
			\draw[thick] (5.5,-0.15) -- (5.5,0.15);
			\node[below=0.3cm] at (5.5,0) {$\beta$};
			\warningtriangle{6.0}{0.55}{0.7}{14.4pt}
			\draw[thick] (6.5,-0.15) -- (6.5,0.15);
			\node[below=0.3cm] at (6.5,0) {$\beta+\eps$};
			\filldraw[black] (3.0,0) circle (0.08);
			\node[below=0.39cm] at (3.0,0) {$y$};
			\filldraw[black] (3.5,0) circle (0.08);
			\node[below=0.3cm] at (3.5,0) {$\tilde{x}$};
			\draw[thick] (3.0,0.4) -- (3.0,0.6) -- (4.0,0.6) -- (4.0,0.4);
			\node[above=0.05cm] at (3.5,0.6) {$\eps$};
		\end{tikzpicture}
		\captionsetup{singlelinecheck=off}
		\caption{Choose $0 \leq \ell < t(\eps)$ such that $\tilde{x}\coloneq x+\ell \alpha \bmod 1$ is within $\eps$ of $y$.  A coding mismatch between $\tilde{x}+n\alpha$ and $y + n\alpha$ could only happen when $\tilde{x} + n\alpha \bmod 1$ belongs to either of the danger zones $[0,\eps)$ and $[\beta,\beta+\eps)$. If $\tilde{x}$ is on the other side of $y$, then the danger zones are shifted left by $\eps$.}
		\label{fig: simple danger zones}
	\end{figure}
	
	Let $(p_j/q_j)$ be the convergents of $\alpha$. At the scale $\eps\asymp1/q_j$, the first $q_j$ orbit points are nearly equally spaced (Figure~\ref{fig: how does the orbit of alpha fill the interval} illustrates why these are the useful scales to consider). This gives $\mathrm{diam}_E(\mathcal W_N) \leq q_j+C\frac N{q_j}$ for some constant $C$ and every $j$. 
	If $q_m^2 < N \leq q_{m+1}^2$, then the preceding estimate with $j=m$ and $j=m+1$ yields
	\begin{equation}\label{af}
		\den N x y \ \leq \ \mathrm{diam}_E(\mathcal W_N) \ \leq \ (C+1) \min\left\{\frac N{q_m},q_{m+1}\right\} \, .
	\end{equation}
	Now $\rho=\rho(\alpha)=\mu(\alpha)-1$ satisfies
	\begin{equation}\label{rhodef}
		\rho \ = \ \limsup_{m\to\infty}\frac{\log q_{m+1}}{\log q_m} \, 
	\end{equation}
	(see \eqref{definition of mu}). The worst case in \eqref{af} is when the two terms in the minimum are roughly equal and $q_{m+1} = q_m^{\rho+o(1)}$. Then $q_m = N^{\frac{1+o(1)}{\rho+1}}$, hence
	\begin{equation} \label{upperh}
		\den N x y   \ \le \ N^{\frac{\rho+o(1)}{\rho + 1}} \, ,
	\end{equation}
	which yields the upper bound on the $\limsup$ exponent in Theorem~\ref{main thm 1}.
	
	\bigskip
	
	Is this upper bound tight? When $\rho$ is large and $\beta$ is in the orbit of $\alpha$ (or unusually close to an initial segment of the orbit), another method yields a better upper bound:
	
	\medskip
	\noindent \textbf{Round and synchronize}. 
	We first round $\alpha$ to $p_m/q_m$ and $\beta$ to a multiple of $1/q_m$. Given $x \in \T$, the orbit segment $\{x+k\alpha \bmod 1 : 0 \leq k < K \}$ has the same coding as the (periodic) coding of the orbit segment $\{x+kp_m/q_m \bmod 1 : 0 \leq k < K \}$ until the first $K$ such that the points $x + K\alpha \bmod 1$ and $x + Kp_m/q_m \bmod 1$ are separated by a point in $\frac{1}{q_m} \Z$. Continued fraction convergents satisfy $q_{m-1} p_m/q_m = \frac{(-1)^{m+1}}{q_m} \bmod 1$, so we synchronize the orbits by deleting $q_{m-1}$ consecutive elements from the periodic orbit.    The classical inequality $|\alpha-p_m/q_m|<\frac{1}{q_mq_{m+1}}$ (see \eqref{exact formula for qnalpha - pn} and \eqref{helpful inequality on q_n alpha}) implies that such synchronizations are at least $q_{m+1}$ time steps apart. Clearly the edit distance between a $q_m$-periodic word and each of its cyclic shifts is at most $q_m$. Applying the same process to the orbit of $y \in \T$ we obtain,   via the triangle inequality, that
	\begin{equation} \label{round1}
		\den N x y \ \le \ q_m+Cq_{m-1}\frac{N}{q_{m+1}} \, . 
	\end{equation}
	
	\medskip
	\noindent\textbf{The golden mean as a threshold for $\rho(\alpha)$.} Suppose that $q_{m+1} \approx q_m^{\rho} \approx q_{m-1}^{\rho^2}$ and $N=q_mq_{m+1}$, so that we are in the worst case for \eqref{af}.  Then the upper bound in \eqref{af} is of order $q_{m+1}$ while the bound in \eqref{round1} is of order
	$q_mq_{m-1}$. 
	We see that the latter is better  when $q_{m+1} \approx q_{m-1}^{\rho^2} > q_{m-1}^{1+\rho}$. This holds exactly when $\rho$ exceeds the golden mean $\varphi$ and explains the appearance of $\varphi$ in Theorem~\ref{main thm thresholds}.
	
	For large $\rho$, tightness of \eqref{upperh} depends on the length $\beta$ of the interval used to code the orbits of $x$ and $y$. For typical $\beta$, synchronizations with respect to the left and right boundaries of $[0,\beta)$ clash, and the approach-and-follow method gives the correct exponent. The hard part of Theorem~\ref{main thm 1} consists of showing that the bound \eqref{af}  obtained via this method is sharp for such $\beta$.
	
	Our lower bound arguments are more intricate and are explained in the relevant sections.
	
	\begin{figure}[htbp]
		\centering
		\begin{tikzpicture}[scale=1.15]
			\draw[thick] (0,0) -- (10.0,0);
			\node at (9.99,0) {)};
			\node[below=0.3cm] at (10,0) {$1$};
			\draw[thick] (0, -0.15) -- (0, 0.15);
			\node[above=0.3cm] at (0, 0) {$0$};
			\draw[thick] (0.1, -0.15) -- (0.1, 0.15);
			\node[below=0.25cm] at (0.6, 0) {{\small $q_m\alpha \bmod 1$}};
			\foreach \k in {1,2} {
				\pgfmathsetmacro{\xpos}{10*\k/6}
				\draw[thick] (\xpos, -0.15) -- (\xpos, 0.15);
				\node[below=0.15cm] at (\xpos, 0) {$\frac{\k}{q_m}$};
			}
			\node[below=0.15cm] at (5.85, 0) {$\dots$};
			\foreach \k in {3,4} {
				\pgfmathsetmacro{\xpos}{10*\k/6}
				\draw[thick] (\xpos, -0.15) -- (\xpos, 0.15);
			}
			\pgfmathsetmacro{\xpos}{10*5/6}
			\draw[thick] (\xpos, -0.15) -- (\xpos, 0.15);
			\node[below=0.15cm] at (\xpos, 0) {$\frac{q_m-1}{q_m}$};
		\end{tikzpicture} \medskip \\
		\centering
		\begin{tikzpicture}[scale=1.15]
			\draw[thick] (0,0) -- (10.0,0);
			\node at (9.99,0) {)};
			\node[below=0.3cm] at (10,0) {$1$};
			\draw[thick] (0, -0.15) -- (0, 0.15);
			\node[below=0.3cm] at (0, 0) {$0$};
			\foreach \k in {1,2} {
				\pgfmathsetmacro{\xpos}{10*\k/6}
				\draw[thick] (\xpos, -0.15) -- (\xpos, 0.15);
				\node[below=0.15cm] at (\xpos, 0) {$\frac{\k}{q_m}$};
			}
			\node[below=0.15cm] at (6.1, 0) {$\dots$};
			\foreach \k in {3,4} {
				\pgfmathsetmacro{\xpos}{10*\k/6}
				\draw[thick] (\xpos, -0.15) -- (\xpos, 0.15);
			}
			\pgfmathsetmacro{\xpos}{10*5/6}
			\draw[thick] (\xpos, -0.15) -- (\xpos, 0.15);
			\node[below=0.15cm] at (\xpos, 0) {$\frac{q_m-1}{q_m}$};
			\foreach \k in {0,1,2,3,4,5} {
				\foreach \x in {0.1,0.2,...,0.8} {
					\draw[thick] (10*\k/6+\x,-0.15) -- (10*\k/6+\x,0.15);
				}
			}
		\end{tikzpicture}
		\captionsetup{singlelinecheck=off}
		\caption{Schematic orbit geometry near a convergent. Since $|\alpha-p_m/q_m|<1/(q_mq_{m+1})$, each of the first $q_m$ orbit points is within $q_{m+1}^{-1}$ of the corresponding rational-grid point. In the upper sketch, the additional mark $q_m\alpha\bmod1$ records the small displacement after $q_m$ steps. When $q_{m+1}/q_m$ is large, later orbit points initially form slowly spreading clusters, as indicated in the lower sketch. This illustrates the gaps between useful approximation scales in approach and follow. The figure uses even $m$; for odd $m$, the displacement is in the opposite direction.}
		
		\label{fig: how does the orbit of alpha fill the interval}
	\end{figure}

	\section{Background material}
	
	\subsection{Continued fractions}
	Let $\alpha \in \R\smallsetminus\Q$. The continued fraction of $\alpha$ is defined by the following algorithm: write $t_0 \coloneq \alpha$ and
	define, for $m = 0,1,2,\ldots$,
	\begin{equation*}
		a_m \coloneq \lfloor t_m\rfloor, \quad t_{m+1} \coloneq \frac{1}{\{t_m\}}.
	\end{equation*}
	The continued fraction representation of $\alpha$ is $[a_0;a_1,a_2,\ldots] = \lim_{ m \to \infty} [a_0;a_1,\ldots,a_m]$, where 
	\begin{equation*}
		[a_0;a_1,\ldots,a_m] \ \coloneq \ a_0 +\cfrac{1}{a_1 + \cfrac{\ddots}{\quad\ddots\quad \frac{1}{a_{m-1} + \frac{1}{a_m}}}}
	\end{equation*}
	is the $m$th \textbf{convergent}. The $m$th convergent is expressible as the irreducible fraction $p_m/q_m$, with values defined as follows. Set $(p_0,q_0) \coloneq (a_0,1)$ and $(p_1,q_1) \coloneq (  a_0a_1+1,a_1)$. For $m \in \N$, define
	\begin{equation*}
		p_{m+1} \ \coloneq \ p_{m-1} + a_{m+1}p_{m} \quad \text{ and } \quad q_{m+1} \ \coloneq \ q_{m-1} + a_{m+1}q_{m} \, .
	\end{equation*}
	
	For $x \in \R$, write $\|x\| = \mathrm{dist}_\R(x,\Z) = \min_{k\in \Z} |x - k|$. It is well known that (see, e.g., \cite{khintchine} or \cite{bugeaud}, Theorem D.1)
	\begin{equation} \label{pq identity}
		p_{m-1}q_m - p_mq_{m-1} \ = \ (-1)^m
	\end{equation}
	and
	\begin{equation} \label{exact formula for qnalpha - pn}
		\| q_m\alpha \| \ = \ (-1)^m(q_m\alpha - p_m )\ = \ \frac{1}{t_{m+1}q_m + q_{m-1}} \, . 
	\end{equation}
	Since $1+a_{m+1} > t_{m+1} \ge a_{m+1}$, it follows that
	\begin{equation} \label{helpful inequality on q_n alpha}
		\frac{1}{2q_{m+1}} \ < \ \frac{1}{q_m+q_{m+1}} \ < \ \|q_m\alpha\| \ < \ \frac{1}{q_{m+1}}  \, .
	\end{equation}

	\begin{definition}
		The \textbf{irrationality exponent} $\mu(\alpha)$ of $\alpha \in \R$ is the supremum of $\mu \in \R$ for which the inequality
		\[
		0 \ < \ \left|\alpha - \frac{p}{q}\right| \ < \ \frac{1}{q^\mu}
		\]
		has infinitely many solutions in $p \in \Z$ and $q \in \N$.
	\end{definition}
	\noindent If $p_m/q_m$, $m \in \N$, are the convergents of $\alpha$, then
	\begin{equation} \label{definition of mu}
		\mu(\alpha) \ = \ 1 + \limsup_{m\to\infty} \frac{\log q_{m+1}}{\log q_{m}} \, ;
	\end{equation}
	this is Exercise E.1 in \cite{bugeaud} and a proof can be found in \cite{sondow}. 
	
	\subsection{Information on orbit gaps}
	Suppose $\alpha \in \R\smallsetminus \Q$ and $\beta \in (0,1)$ are given. For all $k \in \N$, define
	\begin{equation}
		\begin{aligned} \label{definition of lambda} 
			\Lambda(k) &   \ \coloneq  \  \{-n\alpha \bmod 1  : n \in [0,k) \cap \Z \} \, , \\ \Lambda^*(k)   & \  \coloneq \ \Lambda(k) \cup \{\beta -n\alpha \bmod 1 : n \in [0,k)\cap \Z\} \, .
		\end{aligned}
	\end{equation}
	
	The following material is classical; see, e.g., \cite[V.7-8]{herman}. \begin{lemma}\label{lem: ialpha almost uniform}
		Let $\alpha \in \R\smallsetminus\Q$ with convergents $(p_m/q_m)$. Fix $m \in \N$.
		\begin{enumerate}[label=\textup{(\alph*)}]
			\item For all $j \in [0,q_m) \cap \Z$, there exists $z \in \Lambda(q_m)$ such that $|z - j/q_m| < 1/q_{m+1}$.
			\item Suppose $k \in \N$ satisfies $q_{m} \leq k \leq q_{m+1}$. Then $[0,1)  \smallsetminus \{n\alpha \bmod 1 : 0 \leq n < k \}$ is a disjoint union of $k$ open intervals, each of length at most $\frac{1}{q_{m}} + \frac{1}{q_{m+1}}$.  
			\item For all $x \in \R$, there exists $\ell \in [0,q_m) \cap \Z$ with $\|x +\ell\alpha \| \leq \frac{1}{2q_m} + \frac{1}{2q_{m+1}} < \frac{1}{q_m}$.

		\end{enumerate}
	\end{lemma}
	\begin{proof}
		For $k \in [0,q_m)\cap\Z$, multiply \eqref{helpful inequality on q_n alpha} by $k/q_m$ and use \eqref{exact formula for qnalpha - pn} to obtain
		\begin{equation}\label{kalpha approximation}
			0 \ \leq \ (-1)^m \bigl(k\alpha - \frac{kp_m}{q_m} \bigr) \ < \ \frac{1}{q_{m+1}} \, .
		\end{equation}
		Parts (a) and (b) follow from \eqref{kalpha approximation} on observing that the map $n \mapsto np_m \bmod q_m$ bijects $[0,q_m)\cap\Z$ to itself. Part (c) follows from (b).
	\end{proof}

	\begin{lemma}\label{lem: description of three gaps by size and count}
		Fix $\alpha \in \R \smallsetminus\Q$ with convergents $(p_m/q_m)$.
		\begin{enumerate}[label=\textup{(\alph*)}]
			\item For every $m \geq 2$, the set $[0,1) \smallsetminus \{n\alpha \bmod 1 : 0 \leq n < q_m \}$ is a disjoint union of $q_m - q_{m-1}$ intervals of length $\|q_{m-1}\alpha\|$ and $q_{m-1}$ intervals of length $\|q_{m-1}\alpha\| + \|q_{m}\alpha\|$.
			\item For every $m \geq 2$, the set $[0,1) \smallsetminus \{n\alpha \bmod 1 : 0 \leq n \leq q_m \}$ is a disjoint union of $q_m - q_{m-1} + 1$ intervals of length $\|q_{m-1}\alpha\|>\frac{1}{2q_m}$, one interval of length $\|q_m\alpha\|$, and $q_{m-1}-1$ intervals of length $\|q_{m-1}\alpha\| + \|q_{m}\alpha\|$.
		\end{enumerate}
	\end{lemma}
	\begin{proof} Each part follows from Theorems 2.2 and 3.3 in \cite{van ravenstein}. One may also consult \cite{slater}.
	\end{proof}

	\subsection{Edit distance}
	We record three basic facts here and below prove new results about edit distance as needed. First, edit distance is subadditive, i.e., for all $w_1,w_2,\tilde{w}_1,\tilde{w}_2 \in \mathcal{A}^+$,
	\begin{equation}\label{edit distance is subadditive}
		\editfull (w_1w_2,\tilde{w}_1\tilde{w}_2) \ \leq \ \editfull(w_1,\tilde{w}_1) + \editfull(w_2,\tilde{w}_2) \, .
	\end{equation}
	
	Second, given a bi-infinite word $w = (w_n)_{n=-\infty}^{\infty} \in \mathcal{A}^\Z$ and integers $i < j$, write $w{[i,j)}$ for $(w_n)_{n=i}^{j-1}$. Then, for two bi-infinite words $w,\tilde{w} \in \mathcal{A}^{\Z}$, the function $N \mapsto \editfull(w{[0,N)},\tilde{w}{[0,N)})$ is nondecreasing.
	
	Third, the edit distance between finite $w, \tilde{w} \in \mathcal{A}^N$ is at most their Hamming distance:\begin{equation*}
		\editfull(w,\tilde{w}) \ \leq \ \#\{n \in [0,N) : w_{n} \neq \tilde{w}_n \} \, .
	\end{equation*}
	Combining this with the triangle inequality, we see that, for every $x,y \in \T$ and $0 \leq \ell < N$,
	\begin{equation}\label{hamming bound}
		\den N x y \ \leq \ \ell + \#\{n \in [0,N) : \xi_n(x+\ell\alpha) \neq \xi_n(y) \} \, .
	\end{equation}

	\section{Universal upper bounds via approach-and-follow} \label{sec: upper bounds}
	Let $\alpha \in \R\smallsetminus \Q$ and $\beta \in (0,1)$. Given $N \in \N$ and $x,y \in \T$, to obtain an upper bound on
	\begin{equation*}
		\den N x y \ = \ \editfull (\xi_{[0,N)}(x),\xi_{[0,N)}(y)) \, ,
	\end{equation*}
	first find a point $x + \ell \alpha $ that is close to $y$ modulo 1, then observe that $\xi_{\ell + n}(x) = \xi_n(y)$ for many $n$. Lemma~\ref{lem: ialpha almost uniform}(c) makes the first part precise, while the following lemma handles the second part. These are combined in Lemma~\ref{lem: q + N/q bound}.
	\begin{lemma}\label{lem: frequent following}
		Let $\alpha \in \R\smallsetminus\Q$ with convergents $(p_m/q_m)$ and let $\beta \in (0,1)$. Let $x,y \in \R$ and suppose $\|x - y\| \leq \eps$. Then, for all $m \in \N$,
		\begin{equation*}
			\#\bigl\{ n \in [0,q_m]  : \xi_n(x) \neq \xi_n(y) \} \ \leq \ 8\eps q_m + 3 \, .
		\end{equation*}
	\end{lemma}
	\begin{proof}
		Let $S \coloneq [-\eps,\eps]\cup [\beta-\eps,\beta+\eps] \bmod 1$ and observe $x \in [0,1) \smallsetminus S$ implies $\xi_0(x) = \xi_0(y)$. Therefore, it suffices to show that
		\begin{equation} \label{eqn in lem: frequent following}
			\#\{ n \in [0,q_m]\cap\Z : x+n\alpha \bmod 1 \in S \} \ \leq \ 8\eps q_m + 3 \, .
		\end{equation}
		
		For $m \ge 2$, Lemma~\ref{lem: description of three gaps by size and count}(b) and \eqref{helpful inequality on q_n alpha} show that all but one of the gaps between the points $x+n\alpha \bmod1$ with $0 \le n \le q_m$ have length greater than $1/(2q_m)$. For $m=1$, the points $x+n\alpha \bmod 1$ with $0 \le n \le q_1$, have $q_1$ consecutive gaps of length $\{\alpha\}>1/(q_1+1) \ge 1/(2q_1)$ and one remaining gap. An interval of length $2\eps$ contains at most $4\eps q_m$ pairwise disjoint arcs of length greater than $\frac{1}{2q_m}$, hence at most $4\eps q_m + 1$ endpoints of such arcs. Therefore, since the short arc contributes at most one more endpoint to $S$, we deduce \eqref{eqn in lem: frequent following}.
	\end{proof}
	
	\begin{lemma}\label{lem: q + N/q bound}
		Let $\alpha \in \R\smallsetminus\Q$ with convergents $(p_m/q_m)$ and let $\beta \in (0,1)$. For all $N \in \N$ and $m \in \N$,
		\begin{equation*}
			\mathrm{diam}_E(\mathcal{W}_N) \ \leq \ q_m + 11 \frac{N}{q_m} + 11 \, .
		\end{equation*}
	\end{lemma}
	\begin{proof}
		Fix $x,y \in \T$ and put $\eps \coloneq 1/q_m$. By Lemma~\ref{lem: ialpha almost uniform}(c), there exists $\ell \in [0,q_m) \cap \Z$ such that $\tilde{x} \coloneq x + \ell\alpha$ satisfies $\|\tilde{x}  - y \| \leq \eps$. Let $k \coloneq q_m +1$. For each $i \in [0, \lceil N/k\rceil - 1 ] \cap \Z$, Lemma~\ref{lem: frequent following} applied to $\tilde{x} + ik\alpha$ and $y + ik\alpha$ yields
		\begin{equation*} \#\bigl\{ n \in [0,q_m] : \xi_{ik + n}(\tilde{x}) \neq \xi_{ik+n}(y) \} \ \leq \ 11 \, .
		\end{equation*}
		Therefore, since $\left\lceil \frac{N}{k}\right\rceil \leq \frac{N}{q_m} + 1$,
		\begin{equation*} 
			\#\{n \in [0, N) : \xi_{n}(\tilde{x}) \neq \xi_n(y) \} \ \leq \ 11 \frac{N}{q_m} + 11\, .
		\end{equation*}
		The result then follows by \eqref{hamming bound}. 
	\end{proof}
	
	The following upper bounds follow from Lemma~\ref{lem: q + N/q bound}. See also Figure~\ref{fig: approach-and-follow log-log}.
	
	\begin{figure}
		\begin{tikzpicture}[x=0.45cm,y=0.31cm]
			\draw (0,0) -- (17,0) node[right,font=\scriptsize] {$\log N$};
			\draw (0,0) -- (0,9.6) node[above,font=\scriptsize] {$\log\mathrm{diam}_E(\mathcal{W}_N)$};
			\draw[gray,dashed] (0,0) -- (16.4,8.2) node[below right,font=\small] {slope $\tfrac12$};
			\draw[dashed] (0,0) -- (13,8.67) node[above left,font=\small] {slope $\tfrac{\rho}{\rho+1}$};
			\draw[very thick] (2,1) -- (3,2) -- (4,2) -- (6,4) -- (8,4) -- (12,8) -- (16,8);
			\fill (3,2) circle (0.18);
			\fill (6,4) circle (0.18);
			\fill (12,8) circle (0.18);
			\fill[gray] (4,2) circle (0.18);
			\fill[gray] (8,4) circle (0.18);
			\fill[gray] (16,8) circle (0.18);
			\node[above left] at (6,4) {$N=q_mq_{m+1}$};
			\node[gray,below right] at (8,4) {$N=q_{m+1}^2$};
		\end{tikzpicture}
		\captionsetup{singlelinecheck=off}
		\caption{$\!\!\!$ Log--log profile of the approach-and-follow upper bound for $\mathrm{diam}_E(\mathcal W_N)$. For $N\in(q_m^2,q_{m+1}^2]$, the bound obtained from the two adjacent denominators is of order $\min\{N/q_m,q_{m+1}\}$, and its branches meet at $N=q_mq_{m+1}$. Along a subsequence for which $\log q_{m+1}/\log q_m\to\rho=\rho(\alpha)$, the logarithmic exponent of this upper bound at the meeting points tends to $\rho/(\rho+1)$. }
		\label{fig: approach-and-follow log-log}
	\end{figure}
	\begin{theorem}\label{thm: upper bounds on expoliminf and expolimsup} For $\alpha \in \R \smallsetminus \Q$ and $\beta \in (0,1)$,
		\begin{align} \expoliminf (\alpha,\beta) \  & \leq \ 1/2  \label{expoliminf upper bound of 1/2} \\ \text{and } \quad \expolimsup (\alpha,\beta) \ & \leq \ \frac{\rho(\alpha)}{\rho(\alpha)+1} \, \label{expolimsup upper bound of mu-1/mu}.
		\end{align}
	\end{theorem}
	\begin{proof} Suppose $\alpha$ has convergents $(p_m/q_m)$. Let $N \coloneq q_m^2$ with $q_m \geq 11$. By Lemma~\ref{lem: q + N/q bound}, $\mathrm{diam}_E(\mathcal{W}_N) \leq 13\sqrt{N}$, whence \eqref{expoliminf upper bound of 1/2}.
		
		Regarding \eqref{expolimsup upper bound of mu-1/mu}, the inequality $\expolimsup(\alpha,\beta) \leq 1$ is trivial, so assume $\rho(\alpha) < \infty$. To prove \eqref{expolimsup upper bound of mu-1/mu}, it suffices to show that, for every $\rho_+ > \rho(\alpha)$,
		\begin{equation}\label{sufficient statement for expolimsup upper bound of mu-1/mu}
			\expolimsup(\alpha,\beta) \ \leq \ \frac{\rho_+}{1+\rho_+} \, .
		\end{equation}
		
		\noindent By \eqref{definition of mu}, there exists $m_0 \geq 1$ such that $q_{m+1} \leq q_{m}^{\rho_+}$ for all $m \geq m_0$. Fix $N \in \N$ such that $\sqrt{N} > q_{m_0}$, and let $m \in \N$ be such that $q_{m} < \sqrt{N} \leq q_{m+1}$. By Lemma~\ref{lem: q + N/q bound},
		\begin{equation*}
			\mathrm{diam}_E(\mathcal{W}_N) \ \leq \ \min \left\{ q_{m} + 11 \frac{N}{q_{m}}, q_{m+1} + 11 \frac{N}{q_{m+1}} \right\} + 11 \, .
		\end{equation*}
		Now, since $q_{m} < \sqrt{N} < N/{q_{m}}$ and $N/q_{m+1} \leq \sqrt{N} \leq q_{m+1}$, we deduce
		\begin{equation}\label{af2}
			\mathrm{diam}_E(\mathcal{W}_N) \ \leq \ 23 \min \left\{\frac{N}{q_{m}},q_{m+1}\right\}.
		\end{equation}
		Either $N \leq q_{m}q_{m+1}$ or not. In the first case, $q_{m+1} \leq q_{m}^{\rho_+}$ implies $q_{m} \geq N^{\frac{1}{1+\rho_+}}$, so
		\begin{equation*}
			\frac{N}{q_{m}} \ \leq \ N^{\frac{\rho_+}{1+\rho_+}} \, .
		\end{equation*}
		In the second case, $q_{m} \geq q_{m+1}^{1/\rho_+}$ implies $N > q_{m+1}^{1+1/\rho_+}$, so
		\begin{equation*}
			q_{m+1} \ \leq \ N^{\frac{\rho_+}{1+\rho_+}} \, .
		\end{equation*}
		In either case, we conclude
		\[
		\frac{\log \mathrm{diam}_E(\mathcal{W}_N)}{\log N} \ \leq \ \frac{\rho_+}{1+\rho_+} + \frac{\log 23}{\log N} \, .
		\]	
		The previous argument holds for all sufficiently large $N$, hence \eqref{sufficient statement for expolimsup upper bound of mu-1/mu} follows.
	\end{proof}

	\subsection{Upper bound on the liminf exponent for typical orbits}
	\begin{lemma}\label{lem: useful upper bound for orbitexpoliminf} Fix $\alpha \in \R \smallsetminus \Q$ with convergents $(p_m/q_m)$ and let $\beta \in (0,1)$. Suppose $q_{m+1} \geq 16q_m$ for some $m \in \N$ and $N := \lfloor q_{m+1}/8\rfloor$. Then there is a set $A_m \subset \T^2$ such that $\leb^2(A_m) \geq \frac {9}{64}$ and $\den N x y \leq q_m$ for all $(x,y) \in A_m$.
	\end{lemma}
	
	\begin{proof}
		Write $\eta \coloneq q_m\alpha-p_m$ and define $b \coloneq \lfloor q_m\beta\rfloor$ and $r \coloneq q_m\beta - b \in [0,1)$. By \eqref{helpful inequality on q_n alpha}, we have $h \coloneq (N-1)|\eta| < 1/8$. Using the convention that $(c,d)$ is empty if $c \geq d$, let $J_1 \coloneq (0,r)$ and $J_2 \coloneq (r,1)$, then define
		\[
		S_1 \coloneq \begin{cases} (0,r-h) & \text{ if } \eta > 0 \, , \\ (h,r) & \text{ if } \eta < 0 \, , \end{cases} \qquad S_2 \coloneq \begin{cases} (r,1-h) & \text{ if } \eta > 0 \, , \\ (r+h,1) & \text{ if } \eta < 0 \, . \end{cases}
		\]
		
		Fix $i\in\{1,2\}$. For $z \in [0,1)$, let $a \coloneq \lfloor q_m z \rfloor$ and $u \coloneq q_mz \bmod 1$, assume $u \in S_i$, and, for each integer $0\le n<N$, put $j_z(n) \coloneq a+np_m \bmod q_m \in \{0,\ldots, q_m-1\}$. Then $q_m(z+n\alpha) = a + np_m + u + n\eta$ and $u + n\eta \in [0,1)$ for each $n \in [0,N-1]$. Using these facts and $q_m\beta = b+r$, we derive, for each such $n$,
		\begin{equation}\label{periodic code characterization}
			\xi_n(z) = 1
			\quad\Longleftrightarrow\quad
			j_z(n) < b \ \text{ or } \ \bigl(j_z(n) = b \text{ and } u + n\eta < r \bigr).
		\end{equation}
		Moreover, if $u \in S_1$, then $u + n\eta < r$ for all $n \in [0,N-1]$, and if $u \in S_2$, then $u + n\eta > r$ for all $n \in [0,N-1]$.
		
		Suppose $x,y \in [0,1)$ are such that $q_mx \bmod 1$ and $q_my \bmod 1$ lie in the same $S_i$, and let $\ell \in [0,q_m)$ be such that $\lfloor q_mx \rfloor +\ell p_m\equiv \lfloor q_my \rfloor \bmod q_m$. By \eqref{periodic code characterization} and the comment below it, $\xi_{n+\ell}(x)=\xi_n(y)$ for $n \in [0,N-\ell)$, hence $\den N x y \leq \ell < q_m$.
		
		Finally, define $E_i\coloneq\{z\in\T:q_mz\bmod1\in S_i\}$, and put $A_m\coloneq(E_1\times E_1)\cup(E_2\times E_2)$. This set has measure $\sum_i \leb(S_i)^2 \geq \bigl(\frac{1}{2}-h\bigr)^2 > \frac{9}{64}$ since $\leb(J_i) \geq 1/2$ for at least one $i$.
	\end{proof}

	\begin{theorem}\label{thm: orbitexpoliminf upper bound}  Let $\alpha \in \R \smallsetminus \Q$ satisfy $\rho(\alpha) > 1$, and let $\beta \in (0,1)$. Then
		\[
		\orbitexpoliminf(\alpha,\beta)
		\ \leq \ \min\left\{\frac 12, \frac{1}{\rho(\alpha)} \right\}
		\, ,
		\]
		taking $\frac{1}{\rho(\alpha)} = 0$ if $\rho(\alpha) = \infty$.
	\end{theorem}
	\begin{proof}
		First, we have $\orbitexpoliminf(\alpha,\beta) \leq \expoliminf(\alpha,\beta) \leq 1/2$ by \eqref{expoliminf upper bound of 1/2}. 
		
		\noindent Then, suppose $\alpha$ has convergents $(p_m/q_m)$. By \eqref{definition of mu}, choose $m_j \to \infty$ with
		\begin{equation}\label{a growth rate}
			\lim_{j\to\infty} \frac{\log q_{m_j+1}}{\log q_{m_j}} \ = \  \rho(\alpha)\, . 
		\end{equation}
		Then $q_{m_j+1}/q_{m_j} \to \infty$, so Lemma~\ref{lem: useful upper bound for orbitexpoliminf} applies for large $j$, hence Fatou's lemma gives $\leb^2(\limsup_j A_{m_j}) \geq \frac {9}{64}$.
		For every $(x,y)$ in this limsup set, Lemma~\ref{lem: useful upper bound for orbitexpoliminf} and \eqref{a growth rate} imply
		\begin{equation*}
			\liminf_{N\to\infty} \frac{\log \den N x y}{\log N} \ \leq \ \frac {1}{\rho(\alpha)} \, , \quad \text{ hence } \ \orbitexpoliminf(\alpha,\beta) \ \leq \ \frac{1}{\rho(\alpha)} \, . \qedhere 
		\end{equation*}
	\end{proof}

	\section{Useful lemmas for lower bounds}
	
	Define $R_\alpha : \T \to \T$ by $R_\alpha(x) \coloneq x + \alpha \bmod 1$. Lemma~\ref{lem: same code iff same interval} characterizes when (long enough) orbit codings of two points agree.
	
	Given $\alpha$ and $\beta$, consider the partition of $[0,1) \smallsetminus \Lambda^*(k)$ by open intervals with endpoints in $\Lambda^*(k)$, where $\Lambda^*(k)$ is as defined in \eqref{definition of lambda}. Given $x \in [0,1) \smallsetminus \Lambda^*(k)$, write $I_k(x)$ for the open interval in this partition that contains $x$. 
	
	\begin{lemma}\label{lem: same code iff same interval} Let $\alpha \in \R\smallsetminus\Q$ and $\beta \in (0,1)$. Then there exists $k_0 = k_0(\alpha,\beta) > 0$ such that, for all $k \geq k_0$ and all $x,y \in [0,1) \smallsetminus \Lambda^*(k)$, if $\xi_{[0,k)}(x) = \xi_{[0,k)}(y)$, then $I_k(x) = I_k(y)$.
	\end{lemma}
	\begin{remark}
		Fix $k \ge 1$. If $I_k(x) = I_k(y)$, then for all $j \in [0,k)$,  the points $x+j\alpha \bmod 1$ and $y+j\alpha \bmod 1$ are in the  same arc of $\T\smallsetminus\{0,\beta\}$, so   $\xi_0(x+j\alpha \bmod 1)=\xi_0(y+j\alpha \bmod 1)$, whence  
		$\xi_{[0,k)}(x) = \xi_{[0,k)}(y)$.  
	\end{remark}
	\begin{proof}
		First assume $\beta \leq 1/2$.
		
		By density of $(n\alpha \bmod 1)_{n \ge 1}$, we can choose $k_0=k_0(\alpha,\beta)$ so that $\T = \cup_{m = 0}^{k_0-1} R_{\alpha}^{-m}(0,\beta)$. Now suppose $k \geq k_0$ and $x,y \in [0,1) \smallsetminus \Lambda^*(k)$ satisfy $\xi_{[0,k)}(x) = \xi_{[0,k)}(y)$. There exists an integer $m \in [0,k_0)$ such that $x +m\alpha \bmod 1 \in (0,\beta)$, so $y + m\alpha \bmod 1 \in (0,\beta)$ as well, hence $\|x - y\| < \beta \leq 1/2$ and there is a unique shortest closed arc $J_0 \subset \T$ with endpoints $x$ and $y$. Fix $n \in [0,k)\cap\Z$ and let $J_n \subset \T$ denote the shortest closed arc with endpoints $x +n\alpha \bmod 1$ and $y + n\alpha \bmod 1$. By assumption, $x + n\alpha \bmod 1$ and $y + n\alpha \bmod 1$ both belong either to $(0,\beta)$ or to $(\beta,1)$. In the first case, $J_n \subset (0,\beta)$, hence $J_n \cap \{0,\beta\} = \varnothing$, hence $J_0 \cap \{-n\alpha \bmod 1, -n\alpha + \beta \bmod 1\} = \varnothing$. In the second case, $J_n\subset(\beta,1)$, hence the same conclusion follows. Therefore, $J_0$ does not intersect $\Lambda^*(k)$, so $I_k(x) = I_k(y)$.
		
		Now assume $\beta \in (1/2,1)$. We write $\tilde{\beta} \coloneq 1-\beta$ and  $\tilde{\xi}_n(x)  \coloneq \mathbf{1}_{[0,\tilde{\beta})}(x+n\alpha  \bmod 1)$. Suppose that $k \ge  k_0(\alpha,\beta)\coloneq k_0(\alpha, \tilde{\beta})$. Write $\tilde{\Lambda}^*(k)\coloneq \Lambda(k) \cup (\Lambda(k)+\tilde{\beta} \bmod 1)$. For $x \in [0,1) \smallsetminus \tilde{\Lambda}^*(k)$, denote by $\tilde{I}_k(x)$ the open interval containing $x \bmod 1$ in the partition of $[0,1) \smallsetminus \tilde{\Lambda}^*(k)$ by open intervals with endpoints in $\tilde{\Lambda}^*(k)$. Suppose $x,y \in [0,1) \smallsetminus \Lambda^*(k)$ satisfy $\xi_{[0,k)}(x) = \xi_{[0,k)}(y)$. Then $x-\beta \bmod 1 ,y-\beta \bmod 1 \notin \tilde{\Lambda}^*(k)$ and 
		\begin{equation}\label{tilde equality} 
			\tilde{\xi}_{[0,k)}(x-\beta) \ = \ \tilde{\xi}_{[0,k)}(y-\beta) \, .
		\end{equation}
		Now, $\tilde{\beta} \leq 1/2$ and \eqref{tilde equality} imply $\tilde{I}_k(x-\beta ) = \tilde{I}_k(y -\beta)$, which is equivalent to $I_k(x) = I_k(y)$.
	\end{proof}
	
	The combinatorial observation behind our lower bounds is that making few deletions leaves most disjoint intervals unchanged by the matching. We formalize this concept using the notion of clean intervals:
	
	\begin{notation}\label{notation block for clean and shift}
		Let $\mathcal A$ be an alphabet, $N \in \N$, and $w,\widetilde w\in\mathcal A^N$. Put $\Delta\coloneq\editfull(w,\widetilde w)$ and choose a longest common subsequence, represented by strictly increasing maps
		\[
		\pi,\tau:\{1,\ldots,N-\Delta\}\longrightarrow\{0,\ldots,N-1\} \, , \qquad w_{\pi(i)}=\widetilde w_{\tau(i)} \, .
		\]
		For each matched pair, define its \textbf{shift} by $s(i) \coloneq \tau(i)-\pi(i)$. A nonempty interval $I$ of consecutive indices in $\{0,\ldots,N-1\}$ is \textbf{clean}, relative to this matching, if $I^\dagger\coloneq\pi^{-1}(I)$ satisfies $\pi(I^\dagger)=I$ and $\tau(I^\dagger)$ is an interval. In this case $s(i)$ is constant on $I^\dagger$; its common value, denoted by $s_I$, is the shift of $I$. Thus, for integers $n\geq0$ and $k\geq1$ with $n+k\leq N$, a clean interval $I=[n,n+k)\cap\Z$ gives
		\[
		w{[n,n+k)}=\widetilde w{[n+s_I,n+s_I+k)} \, .
		\]
	\end{notation}
	
	See Figure~\ref{fig: clean intervals} for an example.
	
	\begin{center}
		\begin{tikzpicture}[x=.82cm,y=.62cm,font=\small]
			\fill[blue!10] (-.35,.62) rectangle (2.35,1.38);
			\fill[orange!15] (2.65,.62) rectangle (5.35,1.38);
			\fill[red!9] (5.65,.62) rectangle (9.35,1.38);
			\fill[blue!10] (.65,-1.38) rectangle (3.35,-.62);
			\fill[red!9] (5.65,-1.38) rectangle (10.35,-.62);
			\foreach \a/\b in {0/1,1/2,2/3,3/4,4/5,6/6,7/7,8/9,9/10,10/11}
			\draw[black!45] (\a,.65)--(\b,-.65);
			\foreach \a/\letter in {0/a,1/b,2/c,3/d,4/e,5/f,6/g,7/h,8/i,9/j,10/k,11/\ell}
			\node at (\a,1) {$\letter$};
			\foreach \a/\letter in {0/Y,1/a,2/b,3/c,4/d,5/e,6/g,7/h,8/X,9/i,10/j,11/k}
			\node at (\a,-1) {$\letter$};
			\draw[red!75!black,thick] (5,1) ellipse (.24 and .35);
			\draw[red!75!black,thick] (11,1) ellipse (.24 and .35);
			\draw[red!75!black,thick] (0,-1) ellipse (.24 and .35);
			\draw[red!75!black,thick] (8,-1) ellipse (.24 and .35);
			\node[anchor=east] at (-1.3,1) {$w$};
			\node[anchor=east] at (-1.3,-1) {$\widetilde w$};
			\draw[blue!65!black] (-.3,1.5)--(-.3,1.7)--(2.3,1.7)--(2.3,1.5);
			\draw[orange!70!black] (2.7,1.5)--(2.7,1.7)--(5.3,1.7)--(5.3,1.5);
			\draw[red!65!black] (5.7,1.5)--(5.7,1.7)--(9.3,1.7)--(9.3,1.5);
			\node at (1,2.1) {$I_1,\ s_{I_1}=1$};
			\node at (4,2.1) {$I_2$};
			\node at (7.5,2.1) {$I_3$};
		\end{tikzpicture}
		\captionof{figure}{Clean intervals for the displayed longest common subsequence. The circled letters are deleted, and all others are matched. The intervals $I_1,I_2,I_3$ consist respectively of the letters $abc$, $def$, and $ghij$ in $w$. The interval $I_1$ is clean with shift $s_{I_1}=1$. The interval $I_2$ is not clean because $f$ is deleted; $I_3$ is not clean because its matched indices have a gap at the deleted $X$.}
		\label{fig: clean intervals}
	\end{center}
	
	\begin{lemma}\label{lem: lcs basics}
		Assume the setup of Notation~\ref{notation block for clean and shift}. Then
		\begin{equation}\label{eqn: lcs basics}
			|s(i)| \ \leq \ \Delta\quad\text{for every matched pair},
			\qquad
			\sum_{i=1}^{N-\Delta-1}|s(i+1)-s(i)| \ \leq \ 2\Delta \, .
		\end{equation}
		Moreover, for every family $\mathcal I$ of pairwise disjoint nonempty intervals in $\{0,\ldots,N-1\}$, at most $2\Delta$ intervals in $\mathcal I$ are not clean.
	\end{lemma}
	\begin{proof}
		The first inequality in \eqref{eqn: lcs basics} follows because at most $\Delta$ coordinates are deleted from either $w$ or $\tilde{w}$. If $a_i \coloneq \pi(i+1)-\pi(i)-1$ and $b_i \coloneq \tau(i+1)-\tau(i)-1$, then $s(i+1)-s(i)=b_i-a_i$, while $\sum_i a_i \le \Delta$ and $\sum_i b_i \le \Delta$. This proves the variation bound in \eqref{eqn: lcs basics}.  Let $\mathcal I_1$ be the collection of non-clean intervals $I \in \mathcal I$ such that $I^\dagger:=\pi^{-1}(I)$ satisfies $\pi(I^\dagger) = I$, and let $\mathcal I_2$ consist of the remaining non-clean intervals in $\mathcal I$.

		Clearly $|\mathcal I_2 |\le \Delta$.   Thus, it suffices to show that every interval $I \in \mathcal I_1$ can be mapped to a distinct element of $D_\tau := ([0,N)\cap\Z) \smallsetminus \operatorname{im}\tau$, since $|D_\tau|=\Delta$. Each $I\in \mathcal I_1$ is not clean, so $\tau(I^\dagger)$ is not an interval. Thus there are consecutive $i,i+1\in I^\dagger$ for which $\tau(i+1)\ge\tau(i)+2$. Choose the least such $i$, and map $I$ to $\tau(i)+1 \in D_\tau$. Monotonicity of $\pi$ and $\tau$ implies that this mapping from $\mathcal I_1$ to $D_\tau$ is injective, concluding the proof. 
	\end{proof}
	
	The next lemma turns this deterministic observation into a probability bound. It applies to the codings of independent initial points in any invertible measure-preserving system.
	\begin{lemma}\label{lem: small edit probability}
		Let $\mathcal A$ be a finite alphabet, and let $W,\widetilde W$ be independent random elements of $\mathcal A^{\Z}$ with the same law. Suppose this law is invariant under the coordinate shift. For integers $N,k,\ell\geq1$, put
		\[
		\phi_k \ \coloneq \ \mathbb P\bigl(W{[0,k)}=\widetilde W{[0,k)}\bigr) \, , \qquad M\coloneq\lfloor N/k\rfloor \, .
		\]
		If $M \ge 4\ell$, then
		\begin{equation}\label{eq: small edit probability}
			\mathbb P\bigl(\editfull(W{[0,N)},\widetilde W{[0,N)})  \leq  \ell\bigr) \ \leq \ 2(2\ell+1)\phi_k \, .
		\end{equation}
	\end{lemma}
	\begin{proof}
		Count all equalities between one of the $M$ disjoint length-$k$ blocks of $W{[0,N)}$ and a block of $\widetilde W$ shifted by at most $\ell$:
		\[
		S \ \coloneq \ \sum_{j=0}^{M-1}\sum_{s=-\ell}^{\ell}
		\mathbf 1_{\{W{[jk,jk+k)}=\widetilde W{[jk+s,jk+s+k)}\}}.
		\]
		On the event in \eqref{eq: small edit probability}, Lemma~\ref{lem: lcs basics} ensures there are at least $M-2\ell$ clean blocks, each contributing to $S$, so $S\geq M-2\ell$. Independence and invariance under coordinate shifts imply that $\mathbb E S = M(2\ell+1)\phi_k$. Markov's inequality proves the lemma since $\frac{M}{M-2\ell} \le 2$.
	\end{proof}

	The following corollary turns decay of the probability of having equal codings into lower bounds for edit-distance exponents.
	\begin{corollary}\label{utility cor for edit exponent lower bound}
		Let $W,\widetilde W$ be independent copies of a bi-infinite stationary process on a finite alphabet, and put
		\[
		\phi_k \ \coloneq \ \mathbb P\bigl(W{[0,k)}=\widetilde W{[0,k)}\bigr) \, , \qquad
		D_N \ \coloneq \ \editfull(W{[0,N)},\widetilde W{[0,N)}) \, .
		\]
		Let $s>0$. Then:
		\begin{enumerate}[label=\textup{(\alph*)}]
			\item If $\phi_k\leq k^{-s+o(1)}$ as $k\to\infty$, then almost surely
			\[
			\liminf_{N\to\infty}\frac{\log(D_N)}{\log N} \ \geq \ \frac{s}{1+s} \, .
			\]
			\item If $\phi_{k_j}\leq k_j^{-s+o(1)}$ along a sequence $k_j\to\infty$, then almost surely
			\[
			\limsup_{N\to\infty}\frac{\log(D_N)}{\log N} \ \geq \ \frac{s}{1+s} \, .
			\]
		\end{enumerate}
	\end{corollary}
	\begin{proof}
		For (a), fix $0<\eps<s/(1+s)$ and put
		\[
		k\coloneq\lfloor N^{1/(1+s)}\rfloor \, , \qquad
		\ell\coloneq\lfloor N^{s/(1+s)-\eps}\rfloor \, .
		\]
		For all sufficiently large $N$, we have $\lfloor N/k\rfloor\geq4\ell$, so Lemma~\ref{lem: small edit probability} gives
		\[
		\mathbb P(D_N\leq\ell) \ \leq \ 2(2\ell+1)\phi_k \ \leq \ N^{-\eps+o(1)} \, .
		\]
		These bounds are summable in $N=2^j$. By Borel--Cantelli, almost surely $D_{2^j}>2^{j(s/(1+s)-\eps)}$ for all sufficiently large $j$. For $2^j \leq N < 2^{j+1}$, the monotonicity mentioned below \eqref{edit distance is subadditive} gives the same lower bound for $D_N$. Dividing logarithms yields a liminf at least $s/(1+s)-\eps$, then let $\eps\downarrow0$.
		
		For (b), fix $0<\eta<s$ and put
		\[
		N_j\coloneq\lceil k_j^{1+s}\rceil \, , \qquad
		\ell_j\coloneq\lfloor k_j^{s-\eta}\rfloor \, .
		\]
		For all sufficiently large $j$, we have $\lfloor N_j/k_j\rfloor\geq4\ell_j$, so Lemma~\ref{lem: small edit probability} gives
		\[
		\mathbb P(D_{N_j}\leq\ell_j) \ \leq \ k_j^{-\eta+o(1)} \ \longrightarrow \ 0 \, .
		\]
		Consequently,
		\[
		\mathbb P\bigl(\limsup_j\{D_{N_j}>\ell_j\}\bigr) \ \geq\ \limsup_j\mathbb P(D_{N_j}>\ell_j) \ = \ 1 \, .
		\]
		On this set, the limsup exponent is at least $(s-\eta)/(1+s)$, so let $\eta \downarrow 0$.
	\end{proof}
	
	For circle rotations, equality of long enough codings places both initial points in the same short interval. Combining this fact with Lemma~\ref{lem: small edit probability} yields the following estimate.
	\begin{lemma}\label{lem: utility lemma on diameter}
		Let $\alpha\in\R\smallsetminus\Q$ have convergents $(p_m/q_m)$, and let $\beta\in(0,1)$. Let $k_0$ be as in Lemma~\ref{lem: same code iff same interval}, fix $m\in\N$, and put $k\coloneq q_m+1$. Suppose $k>k_0$. Then:
		\begin{enumerate}[label=\textup{(\alph*)}]
			\item If $\ell,N\in\N$ satisfy $4k\ell\leq N$, then for independent uniform $x,y\in\T$,
			\[
			\mathbb P\bigl(\den Nxy\leq\ell\bigr) \ \leq \ \frac{12\ell}{q_m} \, .
			\]
			\item If $N\in\N$ satisfies $q_m^2\leq N$, then
			\[
			\mathrm{diam}_E(\mathcal W_N) \ \geq \ q_m/13 \, .
			\]
		\end{enumerate}
	\end{lemma}
	\begin{proof}
		By Lemma~\ref{lem: same code iff same interval}, $\xi_{[0,k)}(x) = \xi_{[0,k)}(y)$ is equivalent to $I_k(x)=I_k(y)$ outside the finite set $\Lambda^*(k)$. Lemma~\ref{lem: ialpha almost uniform}(b) bounds the length of every $I_k(\cdot )$ by $2/q_m$. Thus the probability $\phi_k$ in Lemma~\ref{lem: small edit probability}, applied to the two orbit codings, satisfies $\phi_k \leq 2/q_m$. Since $\lfloor N/k\rfloor \geq 4\ell$, that lemma gives
		\[
		\mathbb P\bigl(\den Nxy \leq \ell\bigr) \ \leq \ 2(2\ell+1)\phi_k \ \leq \ \frac{12\ell}{q_m} \, ,
		\]
		proving (a). For (b), take $\ell\coloneq\lfloor q_m/13\rfloor$.
	\end{proof}
	
	\section{The case $\rho(\alpha) = 1$}\label{sec: the case mu alpha equals 2}
	We first give a simple lower bound for the liminf exponent for typical orbits.
	
	\begin{theorem}\label{thm: orbitexpoliminf lower bound when mu=2} Let $\alpha \in \R\smallsetminus\Q$ with $\rho(\alpha) = 1$ and $\beta \in (0,1)$. Then
		\begin{equation*}
			\orbitexpoliminf(\alpha,\beta) \ \geq \ \frac{1}{2} \, . 
		\end{equation*}
	\end{theorem}
	\begin{proof}
		Let $(p_m/q_m)$ be the convergents of $\alpha$. Let $x,y\in\T$ be independent and uniform, and put $\phi_k\coloneq\mathbb P(\xi_{[0,k)}(x)=\xi_{[0,k)}(y))$. For all sufficiently large $k$, choose $m$ with $q_m\leq k<q_{m+1}$. Lemmas~\ref{lem: same code iff same interval} and~\ref{lem: ialpha almost uniform}(b) give $\phi_k\leq2/q_m$. Since $\rho(\alpha)=1$, \eqref{definition of mu} implies $q_m=k^{1+o(1)}$, hence $\phi_k\leq k^{-1+o(1)}$. The two orbit codings are independent stationary processes, so Corollary~\ref{utility cor for edit exponent lower bound}(a), with $s=1$, proves the theorem.
	\end{proof}
	
	\begin{proof}[Proof of Theorem~\ref{main thm 1} when $\rho(\alpha) = 1$]
		
		Let $\beta \in (0,1)$. By Theorem~\ref{thm: upper bounds on expoliminf and expolimsup}, $\orbitexpolimsup (\alpha,\beta) \leq \expolimsup(\alpha,\beta) \leq~\tfrac 12$. By Theorem~\ref{thm: orbitexpoliminf lower bound when mu=2}, we have $1/2 \leq \orbitexpoliminf(\alpha,\beta) \leq \expoliminf(\alpha,\beta)$.
	\end{proof}	
	
	For badly approximable $\alpha$, the edit-distance diameter can be estimated more precisely. 
	\begin{proposition}\label{prop: badly approximable edit distance}
		Let $\alpha \in \R\smallsetminus\Q$ with convergents $(p_m/q_m)$ and partial quotients $(a_m)$ and let $\beta \in (0,1)$. Suppose that $\alpha$ is \textbf{badly approximable}, i.e., $\Upsilon(\alpha)\coloneq\max_{m \ge 1} a_m <\infty$. Then there exist constants $c_1,C_2 > 0$ depending only on $\alpha,\beta$ such that, for all $N \in \N$,
		\begin{equation*}
			c_1\sqrt{N} \ \leq \ \mathrm{diam}_E(\mathcal{W}_N) \ \leq \ C_2\sqrt{N} \, . 
		\end{equation*}
	\end{proposition}
	
	\begin{proof}
		Let $k_0$ be as in Lemma~\ref{lem: same code iff same interval}.	
		For all $m \in \N$, we have $q_{m+1} \leq 2a_{m+1}q_m \leq 2\Upsilon(\alpha)q_m$.
		
		Given $N>q_1^2$, find $m \in \N$ such that $q_{m}^2 < N \leq q_{m+1}^2$. We may assume $N$ is large enough that $q_m + 1 > k_0$. By Lemma~\ref{lem: utility lemma on diameter}(b),
		\begin{equation*}
			\mathrm{diam}_E(\mathcal{W}_N) \ \geq \ \frac{q_m}{13}  \ \geq \  \frac{q_{m+1}}{26 \Upsilon(\alpha)}\ \geq \ \frac{\sqrt{N}}{26 \Upsilon(\alpha)}  \, . 
		\end{equation*}
		By Lemma~\ref{lem: q + N/q bound}, 
		$\mathrm{diam}_E(\mathcal{W}_N) \leq q_m + 11 \frac{N}{q_m} + 11 \leq 34\Upsilon(\alpha) \sqrt{N} .$
		The result follows.
	\end{proof}

	\section{Lower bounds on diameter exponents when $\rho(\alpha) > 1$} \label{sec: lower bounds when mu alpha > 2}
	
	We first record an easy general lower bound.
	\begin{proposition}\label{prop: orbitexpolimsup at least 1/2}
		Let $\alpha \in \R\smallsetminus \Q$ with convergents $(p_m/q_m)$ and let $\beta \in (0,1)$. Then $\orbitexpolimsup(\alpha,\beta) \geq 1/2$, hence $\expolimsup(\alpha,\beta) \geq 1/2$. 
	\end{proposition}
	\begin{proof}
		Let $x,y\in\T$ be independent and uniform, so that $(\xi_j(x))_{j \ge 0}$ and $(\xi_j(y))_{j \ge 0}$ are independent stationary processes. For all sufficiently large $m$, Lemmas~\ref{lem: same code iff same interval} and~\ref{lem: ialpha almost uniform}(b) give
		\[
		\mathbb P \bigl(\xi_{[0,q_m+1)}(x)=\xi_{[0,q_m+1)}(y)\bigr) \ \leq \ \frac{2}{q_m} \ = \ (q_m+1)^{-1+o(1)} \, .
		\]
		Corollary~\ref{utility cor for edit exponent lower bound}(b), with $s=1$, therefore gives $\orbitexpolimsup(\alpha,\beta) \geq 1/2$.
	\end{proof}
	
	\medskip

	Next, we need a lemma that counts how often a finite orbit of a point in $\T$ enters $[0,\beta)$, as long as that point is not too close to the boundary.
	
	\begin{lemma}\label{lem: phase count identity}
		Let $\alpha \in \R\smallsetminus\Q$ have convergents $(p_m/q_m)$, let $\beta\in(0,1)$, and fix $m\in\N$. Put $q\coloneq q_m$, $Q\coloneq q_{m+1}$, and $r\coloneq q\beta \bmod 1$. For $z\in\R$, define
		\[
		F(z) \ \coloneq \ \sum_{j=0}^{q-1}\mathbf{1}_{[0,\beta)}(z+j\alpha\bmod1) \, .
		\]
		If $\min \{ \|qz\|,\|qz-r\| \}  > q/Q$, then $F(z) = \lfloor q\beta\rfloor+\mathbf{1}_{[0,r)}(qz \bmod 1)$.
	\end{lemma}
	\begin{proof}
		Since $p_m$ and $q$ are coprime, the grid
		\[
		G \ \coloneq \ \{z+jp_m/q\bmod1:0\leq j<q\} \ =\ \{z+k/q \bmod 1:0\leq k<q\}
		\]
		satisfies $|G\cap[0,\beta)| = \lfloor q\beta\rfloor+\mathbf{1}_{[0,r)}(qz\bmod1)$. Recall that $\|j\alpha - jp_m/q\| < 1/Q$ for such $j$ by \eqref{helpful inequality on q_n alpha}. Moreover, $\mathrm{dist}_{\T}(0,G) = \|qz\|/q$ and $\mathrm{dist}_{\T}(\beta,G) = \|qz-r\|/q$. The hypothesis then yields $F(z) = |G \cap [0,\beta)|$.
	\end{proof}
	
	Now, we record a lower bound on how often a finite arithmetic progression on $\T$ enters an arbitrary union of arcs.
	
	\begin{lemma}\label{lem: finite circle progression count}
		Let $M,k \in \N$, let $\delta \in (0,1)$ satisfy $(M-1)\delta < 1$, and let $\sigma \in \{-1,1\}$. Given $v_0 \in \T$, put $v_i \coloneq v_0+\sigma i \delta \bmod 1$ for $i \in [0,M) \cap \Z$. Let $I$ be the arc traversed from $v_0$ to $v_{M-1}$ in the direction $\sigma$. Let $A \subset \T$ be a union of at most $k$ arcs. Then
		\begin{equation*}
			\#\{i \in [0,M) \cap \Z : v_i \in A \} \ \geq \ \frac{\leb(A)}{\delta}-\bigl(\frac{\leb(\T \smallsetminus I)}{\delta} + k + 1 \bigr) \, .
		\end{equation*}
	\end{lemma}
	
	\begin{proof}
		After combining overlapping arcs, $A$ has at most $k$ connected components,
		so $A \cap I$ is a union of at most $k+1$ arcs. If their lengths are $\ell_1,\ldots,\ell_s$, where $s \leq k+1$, then each contains at least $\ell_j/\delta-1$ of the points $v_i$, so
		\begin{equation*}
			\#\{ i \in [0,M)\cap\Z : v_i\in A\} \ \geq  \ \sum_{j=1}^s\left(\frac{\ell_j}{\delta}-1\right) \ \geq \ \frac{\leb(A\cap I)}{\delta}-(k+1) \, ,
		\end{equation*}
		and the result follows since $\leb(A\cap I)\geq\leb(A)-\leb(\T \smallsetminus I)$.
	\end{proof}
	
	The following key lemma says that if a pair $(x,y)$ has small edit distance in the coding determined by $[0,\beta)$, then either $\beta$ or $x-y$ is abnormally close to the grid $q_m^{-1}\Z$. 
	\begin{lemma}\label{lem: blocks} Let $\alpha \in \R \smallsetminus \Q$ with convergents $(p_m/q_m)$, and let $\beta \in (0,1)$. Put $q \coloneq q_m$ and $Q \coloneq q_{m+1}$ for some fixed $m \in \N$. Let $L,N \in \Z$ with $L \geq 0$ and $N \geq Q$. For all $x,y \in \T$, if $\den N x y\leq L$, then
		\begin{equation*}
			a \ \coloneq \ \min\{\|q(x-y)\|,\|q\beta\|\} \ \leq \ 10\left(\frac{Lq}{N}+\frac LQ+\frac qQ\right).
		\end{equation*}
	\end{lemma}
	\begin{remark}
		Although we will not need this, we note Lemma~\ref{lem: blocks} implies, when $a \geq 20q/Q$,
		\begin{equation*}
			\den Nxy \ \geq \ \frac{a}{40}\min\left\{\frac Nq,\,Q\right\} \, .            
		\end{equation*}
		This inequality should be compared with the approach-and-follow bound from \eqref{af2}, namely, $\den Nxy \leq C\min\{N/q_m,q_{m+1}\}$; observe that $a$ is of order 1 for typical $\beta$ and $(x,y)$.  
	\end{remark}
	\begin{proof}[Proof of Lemma \ref{lem: blocks}]
		Use $r$ and $F$ from Lemma~\ref{lem: phase count identity}, and put $J \coloneq [0,r) \subset \T$.
		
		Fix $x,y \in \T$, assume $\Delta \coloneq \den N x y \leq L$, and choose a longest common subsequence between $\xi_{[0,N)}(x)$ and $\xi_{[0,N)}(y)$ as in Notation~\ref{notation block for clean and shift}. If $[n,n+q)$ is clean for some $n \in \Z$ and has shift $s$, then $|s| \leq \Delta \leq L$ and $\xi_{[n,n+q)}(x) = \xi_{[n+s,n+s+q)}(y)$, so one has the identities
		\begin{equation}\label{a phase shift}
			F(x+n\alpha) \ = \ F(y+(n+s)\alpha) \quad \text{ and } \quad u_{n+s}(y) \ \equiv \ u_n(x)+t+sq\alpha \bmod 1 \, ,
		\end{equation}
		where $t \coloneq q(y-x) \bmod 1$ and $u_i(z) \coloneq q(z+i\alpha) \bmod 1$ for $z \in \T$ and $i \in \Z$. Then $a = \min\{\|t\|,\|r\|\}$. Writing $\triangle$ for the symmetric difference and $J-t$ for subtraction in $\T$, we have $\leb(J \mathbin{\triangle} (J-t)) =  2a$. Define the set (see Figure~\ref{fig: example of set A})
		\begin{equation} \label{definition of A}
			A \ \coloneq \ \left\{u\in J \mathbin{\triangle} (J-t ) : \mathrm{dist}_\T(u,\{0,r,-t,r-t \}) > \frac{q+L}{Q} \right\} \, ,
		\end{equation}
		which has measure
		\begin{equation}\label{measure of A}
			\leb(A) \ \geq \ 2a - 8\frac{q+L}{Q} \, .
		\end{equation}
		
		\begin{figure} 
			\begin{tikzpicture}[x=9.5cm,y=.62cm,font=\small]
				\foreach \yy in {0,1,2,3}{\draw[black!35] (0,\yy)--(1,\yy);}
				\fill[blue!50!black] (0,2.88) rectangle (.45,3.12);
				\fill[blue!50!black] (0,1.88) rectangle (.15,2.12);
				\fill[blue!50!black] (.70,1.88) rectangle (1,2.12);
				\fill[red!35] (.15,.88) rectangle (.45,1.12);
				\fill[red!35] (.70,.88) rectangle (1,1.12);
				\fill[black!18] (.15,-.12) rectangle (.45,.12);
				\fill[black!18] (.70,-.12) rectangle (1,.12);
				\fill[red!65!black] (.185,-.12) rectangle (.415,.12);
				\fill[red!65!black] (.735,-.12) rectangle (.965,.12);
				\foreach \xx in {0,.15,.45,.70,1}{\draw[densely dotted,black!40] (\xx,-.18)--(\xx,3.3);}
				\node[anchor=east] at (-.025,3) {$J$};
				\node[anchor=east] at (-.025,2) {$J-t$};
				\node[anchor=east] at (-.025,1) {$J \mathbin{\triangle} (J-t)$};
				\node[anchor=east] at (-.025,0) {$A$};
				\node[above] at (0,3.3) {$0$};
				\node[above] at (.15,3.3) {$r-t \bmod 1$};
				\node[above] at (.45,3.3) {$r$};
				\node[above] at (.70,3.3) {$-t\bmod1$};
				\node[above] at (1,3.3) {$1$};
				\draw[<->] (.15,-.42)--(.185,-.42);
				\node[below] at (.1675,-.42) {$\eps$};
			\end{tikzpicture}
			\captionsetup{singlelinecheck=off}
			\caption{Example of \eqref{definition of A} with $r=0.45$, $t=0.30$, and $\eps=0.035 = \frac{q+L}{Q}$.}
			\label{fig: example of set A}
		\end{figure}
		
		\noindent \textbf{Claim 1.} For all $n \in [0,N-q] \cap \Z$, if $u_n(x) \in A$, then $[n,n+q)$ is not clean. 
		
		Indeed, if $[n,n+q)$ were clean with shift $s$, then \eqref{helpful inequality on q_n alpha} and $|s| \leq L$ would imply the inequality $|s(q\alpha-p_m)| \leq L/Q$. With $\mathrm{dist}_\T(u_n(x)+t,\{0,r\}) > (q+L)/Q$ and \eqref{a phase shift}, obtain
		\begin{equation} \label{a consequence} 
			\mathbf{1}_J(u_n(x) + t \bmod 1) \ = \ \mathbf{1}_J(u_{n+s}(y)) \quad \text{ and } \quad 
			\mathrm{dist}_\T(u_{n+s}(y),\{0,r\}) \ > \ q/Q \, .
		\end{equation}
		Since also $\mathrm{dist}_\T(u_n(x),\{0,r\}) > q/Q$, it would follow by Lemma~\ref{lem: phase count identity} and \eqref{a phase shift} that
		\begin{equation*}
			\mathbf{1}_J(u_n(x)) \ = \ \mathbf{1}_J(u_{n+s}(y)) \ = \ \mathbf{1}_J(u_n(x) + t \bmod 1) \, ,
		\end{equation*}
		the second equality holding by \eqref{a consequence}, but this contradicts $u_n(x) \in J \mathbin{\triangle} (J-t)$. \medskip
		
		Now, put $B \coloneq \lfloor\frac NQ \rfloor$ and $R_0 \coloneq \lfloor\frac Qq \rfloor$. We lower bound the number of bad rounds of length $q$ in a block of length $Q$:
		
		\noindent \textbf{Claim 2.} For each $h \in [0,B)\cap \Z$,
		\begin{equation*}
			\#\{ i \in [0,R_0) \cap \Z : u_{hQ+iq}(x) \in A  \} \ \geq \ R_0 \leb(A) - 6 \, .
		\end{equation*}
		For each $i \in [0,R_0) \cap \Z$, we have
		\begin{equation} \label{round phase relation}
			u_{hQ+iq}(x) \ \equiv \ u_{hQ}(x) + iq^2\alpha \bmod 1 \, .
		\end{equation}
		Let $d \coloneq q|q\alpha - p_m|$, which by \eqref{helpful inequality on q_n alpha} satisfies $\frac {q}{Q+q} < d < \frac qQ$, hence
		\begin{equation}\label{how long is the block arc}
			(R_0-1)d \ > \ \left(\frac Qq - 2\right) \frac{q}{Q+q}
			\ = \ \frac{Q-2q}{Q+q} \, .
		\end{equation}
		By \eqref{round phase relation}, there is an arc $I \subset \T$ of length $(R_0-1)d$ from $u_{hQ}(x)$ to $u_{hQ + (R_0-1)q}(x)$ containing all $u_{hQ + iq}(x)$ for $i \in [0,R_0) \cap \Z$. By \eqref{how long is the block arc}, the arc $\T \smallsetminus I$ has length less than $3q/(Q+q) < 3d$. The symmetric difference $J \mathbin{\triangle}(J-t)$ is a union of at most two arcs, and deleting the neighborhoods of $0,r,-t,r-t$ leaves at most two arcs. Apply Lemma~\ref{lem: finite circle progression count} with $M = R_0$, $\delta = d$, and $k = 2$. Since $1/d > Q/q \geq R_0$ and $\leb(\T\smallsetminus I)/d < 3$, its error term is less than $3+2+1=6$. This proves the claim. \medskip
		
		It follows from Claim 2 that the set $R_\textrm{bad}$ of starting times $n=hQ+iq$ with $h\in[0,B)\cap\Z$, $i\in[0,R_0)\cap\Z$, and $u_n(x)\in A$ has cardinality at least
		\begin{equation}\label{how many bad rounds?}
			B(R_0\leb(A)-6)\ \geq\ B\left(2R_0a-8R_0\frac{q+L}{Q}-6\right)
		\end{equation}
		by \eqref{measure of A}. For each $n \in R_\textrm{bad}$, the interval $[n,n+q)$ is not clean by Claim 1, and the collection $\{ [n,n+q) : n \in R_\textrm{bad} \}$ is pairwise disjoint and therefore, by Lemma~\ref{lem: lcs basics}, has cardinality at most $2\Delta \leq 2L$. Combining this fact with \eqref{how many bad rounds?}, we deduce
		\[
		\min\{\|q(x-y)\|,\|q\beta\|\} \ = \ a \ \leq \ \frac{L}{BR_0} + 4\frac{q+L}{Q} + \frac{3}{R_0} \, .
		\]
		Because $N \geq Q$ and $Q \geq q$, we have
		$B \geq N/(2Q)$ and $R_0 \geq Q/(2q)$. It follows that $a \leq 4Lq/N+4L/Q+10q/Q$, which proves the lemma.
	\end{proof}
	
	We immediately derive the following corollary, which will be used in this section.
	
	\begin{corollary}\label{cor: blocks}
		Let $\alpha\in\R\smallsetminus\Q$ have convergents $(p_m/q_m)$ and let $\beta\in(0,1)$. Put $q\coloneq q_m$ and $Q\coloneq q_{m+1}$ for some $m\in\N$. If $N\geq q$ is an integer, $L\geq 0$, and $\mathrm{diam}_E(\mathcal W_N)\leq L$, then
		\begin{equation*}
			\|q\beta\| \ \leq \ 20\left(\frac{Lq}{N}+\frac{q+L}{Q}\right) \, .
		\end{equation*}
	\end{corollary}
	\begin{proof}
		Put $D\coloneq\mathrm{diam}_E(\mathcal W_N)$ and $k\coloneq\lceil Q/N\rceil$. By \eqref{edit distance is subadditive}, for all $x,y\in\T$,
		\[
		\den {kN}xy \ \leq \ \sum_{j=0}^{k-1} \den N{x+jN\alpha}{y+jN\alpha} \ \leq \ kD \, .
		\]
		Choose $x=0$ and $y=1/(2q)$, so that $\|q(x-y)\|=1/2\geq\|q\beta\|$. Since $kN\geq Q$ and $kD$ is a nonnegative integer, Lemma~\ref{lem: blocks}, applied with $L = kD$, gives
		\[
		\|q\beta\| \ \leq \ 10\left(\frac{Dq}{N}+\frac{kD}{Q}+\frac qQ\right).
		\]
		Using $k\leq Q/N+1$ and $D\leq L$, we obtain
		\[
		\|q\beta\| \ \leq \ 10\left(\frac{D(q+1)}N+\frac{q+D}{Q}\right) \ \leq \ 20\left(\frac{Lq}{N}+\frac{q+L}{Q}\right) \, . \qedhere
		\]
	\end{proof}
	
	\subsection{The exceptional set}
	
	Suppose $\alpha \in \R \smallsetminus \Q$ has convergents $(p_m/q_m)$. Given $\eps > 0$ and $m \in \N$, define
	\[
	B^{\eps}_m(\alpha) \ \coloneq \ \{ \beta \in (0,1) : \mathrm{dist}_\T(\beta,\Lambda(q_m)) \ \leq \ q_m^{-1-\eps}\} \, . 
	\]
	This set has Lebesgue measure at most $2q_m^{-\eps}$. Since $\sum_{m\in\N} 2q_m^{-\eps} < \infty$, the Borel--Cantelli lemma implies the set
	\[ B^\eps(\alpha) \ \coloneq \ \bigcap_{M \in \N} \bigcup_{m > M} B^{\eps}_m(\alpha) \ = \ \limsup_{m\to\infty} B_m^{\eps}(\alpha)\]
	has measure zero.
	
	If $\eps_1 < \eps_2$, then $B_m^{\eps_1}(\alpha) \supseteq B_m^{\eps_2}(\alpha)$; hence
	\begin{equation}\label{B alpha definition}
		B(\alpha) \ \coloneq \ \bigcup_{\eps > 0} B^{\eps}(\alpha) \ = \ \bigcup_{0 < \eps \in \Q} B^{\eps}(\alpha)
	\end{equation}
	has measure zero.
	
	Lemma~\ref{lem: ialpha almost uniform}(a) shows the following: for $q\coloneq q_m$ and $Q\coloneq q_{m+1}$,
	\begin{equation}\label{grid comparison for beta}
		q\,\mathrm{dist}_{\T}(\beta,\Lambda(q))\ \leq\ \|q\beta\|+q/Q \, .
	\end{equation}
	Indeed, choose a point of $q^{-1}\Z\bmod1$ nearest to $\beta$, and then a point of $\Lambda(q)$ at distance less than $1/Q$ from the chosen point.
	
	\subsection{Lower bound on the liminf diameter exponent when $\rho(\alpha)>1$}
	
	\begin{theorem}\label{thm: expoliminf lower bound of 1/2}
		Let $\alpha\in\R\smallsetminus\Q$ with $\rho(\alpha)>1$. For every $\beta\in(0,1)\smallsetminus B(\alpha)$,
		\[
		\expoliminf(\alpha,\beta) \ \geq \ 1/2\, .
		\]
	\end{theorem}
	\begin{proof}
		Fix $\beta\notin B(\alpha)$ and $\eta>0$. Write $D_N\coloneq\mathrm{diam}_E(\mathcal W_N)$ and choose consecutive convergent denominators $q=q_m\le\sqrt N<q_{m+1}=Q$. Corollary~\ref{cor: blocks}, with $L=D_N$, yields
		\[
		\sqrt N\,\|q\beta\| \le 20\left(\frac{D_Nq}{\sqrt N}+\frac{\sqrt N(q+D_N)}{Q}\right) \le 40D_N+20q \, .
		\]
		For all large $N$, Lemma~\ref{lem: utility lemma on diameter}(b) gives $q \le 13D_N$, since $q^2\le N$ and also $q+1>k_0(\alpha,\beta)$. Together with \eqref{grid comparison for beta}, these inequalities yield
		\[
		\sqrt N\,q\,\mathrm{dist}_{\T}(\beta,\Lambda(q)) \le \sqrt N\,\|q\beta\|+\frac{\sqrt N\,q}{Q} \le 40D_N+21q \le 313D_N \, .
		\]
		Since $\beta\notin B(\alpha)$, eventually $q\,\mathrm{dist}_{\T}(\beta,\Lambda(q))>q^{-\eta}$.
		It follows that
		\begin{equation}\label{diamlowernew}
			D_N>\frac{\sqrt N\,q^{-\eta}}{313} \ge \frac{N^{(1-\eta)/2}}{313}.
		\end{equation}
		Taking logarithms, dividing by $\log N$, and then letting $N\to\infty$ and $\eta\downarrow 0$ proves the theorem.
	\end{proof}
	
	\section{Lower bounds on exponents for typical orbits}\label{sec: orbitexpoliminf lower bound for most beta}
	
	\subsection{Lower bound on the limsup exponent}
	
	\begin{theorem}\label{thm: orbitexpolimsup lower bound mu-1/mu}
		Let $\alpha\in\R\smallsetminus\Q$ with $\rho(\alpha)>1$. For every $\beta\in(0,1)\smallsetminus B(\alpha)$,
		\[
		\orbitexpolimsup(\alpha,\beta)\ \geq\ \frac{\rho(\alpha)}{\rho(\alpha)+1}\, ,
		\]
		hence $\expolimsup(\alpha,\beta)\geq\frac{\rho(\alpha)}{\rho(\alpha)+1}$, with the right-hand side interpreted as $1$ when $\rho(\alpha)=\infty$.
	\end{theorem}
	\begin{proof}
		Write $(p_m/q_m)$ for the convergents of $\alpha$, and fix an infinite set $\mathcal M\subset\N$ along which $\log q_{m+1}/\log q_m\to\rho(\alpha)$. Fix $\eps \in (0,\frac{\rho(\alpha)}{\rho(\alpha)+1})$. For $m\in\mathcal M$, write $q\coloneq q_m$ and $Q\coloneq q_{m+1}$, and take
		\[
		N \ \coloneq \ qQ \, , \qquad \ell \ \coloneq \ \lfloor QN^{-\eps}\rfloor \, , \qquad
		H \ \coloneq \ 10\bigl(2N^{-\eps}+q/Q\bigr) \, .
		\]
		Then $\ell\to\infty$, and Lemma~\ref{lem: blocks} shows that $\den Nxy \leq \ell$ implies
		\begin{equation}\label{application of lem: blocks} 
			\min\{\|q(x-y)\|,\|q\beta\|\}\ \leq \ H\, .
		\end{equation}
		Choose $s>1$ so that $Q\geq q^s$ for all sufficiently large $m\in\mathcal M$, and put \[ \delta \  \coloneq \ \min\{(1+s)\eps,s-1\} > 0 \, .\]
		Then $H \leq 30q^{-\delta}$. 
		Let $\eta \in (0,\delta)$. Since $\beta \notin B(\alpha)$, \eqref{grid comparison for beta} gives $\|q\beta\|>q^{-\eta}-q/Q>H$ for all sufficiently large $m\in\mathcal M$. Consequently, for uniform independent $x,y \in \T$, \eqref{application of lem: blocks} implies
		\[
		\mathbb P\bigl(\den Nxy \leq \ell \bigr)\ \leq\ \mathbb P\bigl(\|q(x-y)\|\leq H\bigr)\ \leq\ 2H\ \leq\ 60q^{-\delta}\, ,
		\]
		where the second inequality holds since the variable $q(x-y)\bmod1$ is uniform on $\T$. The series of these bounds over $m\in\mathcal M$ converges. Borel--Cantelli therefore gives, for almost every pair $(x,y)$, the eventual inequality
		\[
		\den {qQ}xy\ >\ \lfloor Q(qQ)^{-\eps}\rfloor\, ,
		\qquad m\in\mathcal M\, .
		\]
		It follows that
		\[
		\limsup_{m\to\infty} \frac{\log\den {qQ}xy}{\log(qQ)} \ \geq \ \frac{\rho(\alpha)}{\rho(\alpha)+1}-\eps \, .
		\]
		Taking $\eps \downarrow 0$ proves the theorem.
	\end{proof}
	
	\begin{remark}
		The subsequence $N=q_mq_{m+1}$, $m\in\mathcal M$, in the preceding proof depends only on $\alpha$.
	\end{remark}
	
	\subsection{Lower bound on the liminf exponent}
	
	\begin{theorem}\label{thm: orbitexpoliminf lower bound}
		Let $\alpha \in \R \smallsetminus \Q$ satisfy $\rho(\alpha) > 1$. For all $\beta \in (0,1)\smallsetminus B(\alpha)$,
		\[
		\orbitexpoliminf(\alpha,\beta)
		\ \geq \ \min\left\{\frac 12, \frac{1}{\rho(\alpha)} \right\}
		\, ,
		\]
		taking $\frac{1}{\rho(\alpha)} = 0$ if $\rho(\alpha) = \infty$.
	\end{theorem}
	\begin{proof}
		The desired inequality is trivial when $\rho(\alpha) = \infty$, so assume $\rho(\alpha) \in (1,\infty)$, write $(p_m/q_m)$ for the convergents of $\alpha$, and let $\beta \in (0,1)\smallsetminus B(\alpha)$.
		
		Put $\theta \coloneq \min \{ \frac 12, \frac{1}{\rho(\alpha)} \}$, and fix $\eps \in (0,\theta)$. For a dyadic integer $N$, let $\ell \coloneq \lfloor N^{\theta-\eps}\rfloor$ and choose $m$ so that $q_m \leq N^\theta < q_{m+1}$. Write $q \coloneq q_m$ and $Q \coloneq q_{m+1}$.
		
		Since $\theta\rho(\alpha)\leq1$, we have $(\theta-\eps/2)\rho(\alpha)<1$. Choose $\gamma > 0$ with
		\begin{equation} \label{gamma choice}
			(\theta-\eps/2)(\rho(\alpha)+\gamma) \ < \ 1 \, .
		\end{equation}
		By \eqref{definition of mu}, if $m$ is sufficiently large, then
		\begin{equation} \label{bound on Q for orbitexpoliminf}
			Q \ \leq \ q^{\rho(\alpha)+\gamma} \, .
		\end{equation}
		Set
		\begin{equation}\label{definition of eta*}
			\eta_* \ \coloneq \ \frac{\eps}{2(\theta-\eps/2)} \, .
		\end{equation}
		Fix $\eta \in (0,\eta_*)$. Since $\beta \notin B(\alpha)$, for sufficiently large $m$, we have $q\,\mathrm{dist}_{\T}(\beta,\Lambda(q)) > q^{-\eta}$, so
		\begin{equation}\label{condition on qbeta}
			q\,\mathrm{dist}_{\T}(\beta,\Lambda(q)) \ > \ q^{-\eta} \ > \ 31q^{-\eta_*} \, .
		\end{equation}
		
		\noindent As $N \to \infty$, we have $m \to \infty$. Therefore, by taking $N$ large enough, we may assume that $N^{\eps} > 8$ and \eqref{bound on Q for orbitexpoliminf} hold, as well as \eqref{condition on qbeta} and $q+1 > k_0(\alpha,\beta)$, where $k_0$ is as in Lemma~\ref{lem: same code iff same interval}.
		
		First, we show by cases that, for uniform and independent $x,y \in \T$,
		\begin{equation}\label{claim for summability}
			\mathbb{P} \bigl( \den N x y \leq \ell \bigr) \ \leq \ 60N^{-\eps/2} \, .
		\end{equation}
		
		\medskip
		\noindent\textbf{Case 1:} $q \geq N^{\theta-\eps/2}$. Since $q \leq N^\theta$ and $2\theta \leq 1$,
		\[
		8q\ell \ \leq \ 8N^{2\theta-\eps} \ \leq \ 8N^{1-\eps} \ < \ N \, ,
		\]
		hence $4(q+1)\ell < N$. By Lemma~\ref{lem: utility lemma on diameter}(a), applied with $k=q+1$,
		\begin{equation*}
			\mathbb{P} \bigl ( \den N x y \leq \ell \bigr) \ \leq \ \frac{12\ell}{q} \ \leq \ 12N^{-\eps/2} \, ,
		\end{equation*}
		where the second inequality holds by the assumption $q \geq N^{\theta-\eps/2}$.
		
		\medskip
		\noindent\textbf{Case 2:} $q < N^{\theta-\eps/2}$. Then $\frac{Q}{q} > \frac{N^\theta}{q} > N^{\eps/2}$. By \eqref{bound on Q for orbitexpoliminf} and \eqref{gamma choice},
		\[
		Q \ \leq \ N^{(\theta-\eps/2)(\rho(\alpha)+\gamma)} \ < \ N \, ,
		\]
		so Lemma~\ref{lem: blocks} applies. Define the quantity
		\[
		H_N \ \coloneq \ 10\left(\frac{\ell q}{N} + \frac \ell Q + \frac qQ \right) \, .
		\]
		Using the three inequalities $q < N^{\theta-\eps/2}$, $Q > N^\theta$, and $\frac{Q}{q} > N^{\eps/2}$, one verifies the first inequality in
		\begin{equation} \label{inequality on H_n}
			H_N \ \leq \ 10\left(N^{2\theta-1-3\eps/2} + N^{-\eps} + N^{-\eps/2}\right) \ \leq \ 30N^{-\eps/2} \, ,
		\end{equation}
		the second inequality holding since $2\theta \leq 1$. Now, the inequality $q<N^{\theta-\eps/2}$ implies $N^{-\eps/2}<q^{-\eta_*}$ (see \eqref{definition of eta*}), hence $q/Q<q^{-\eta_*}$. By \eqref{grid comparison for beta} and \eqref{condition on qbeta},
		\begin{equation*}
			\|q\beta\| \ \geq \ q\,\mathrm{dist}_{\T}(\beta,\Lambda(q))-\frac qQ \ > \ q^{-\eta}-q^{-\eta_*} \ > \ 30q^{-\eta_*} \, .
		\end{equation*}
		It follows that $H_N < \|q\beta\|$ by \eqref{inequality on H_n}. Lemma~\ref{lem: blocks} then implies
		\[
		\mathbb{P} \bigl( \den N x y \leq \ell \bigr)
		\ \leq \ \mathbb{P} \bigl( \|q(x-y)\| \leq H_N \bigr) \ \leq \ 2H_N \ \leq \ 60N^{-\eps/2} \, .
		\]
		We have shown \eqref{claim for summability}.
		
		\medskip \noindent Now, we use a Borel--Cantelli argument to finish. For $j \in \N$, define the event
		\[
		A_j \ \coloneq \ \bigl\{\exists N \in [2^j,2^{j+1})\cap\N : \den N x y \leq \bigl\lfloor 2^{j(\theta-\eps)} \bigr\rfloor\bigr\} \ =  \ \bigl\{ \den {2^j} x y \ \leq\ \bigl\lfloor 2^{j(\theta-\eps)} \bigr\rfloor\bigr\} \, ,
		\]
		where the equality follows from the monotonicity of edit distance recorded below \eqref{edit distance is subadditive}. For all sufficiently large $j$, \eqref{claim for summability} gives
		\[
		\mathbb P(A_j) \ \leq \ 60 \cdot 2^{-j\eps/2} \, .
		\]
		Thus $\sum_{j\in\N} \mathbb P(A_j) < \infty$. Borel--Cantelli and the integer-valuedness of edit distance imply that, for almost every $(x,y) \in \T^2$, for all sufficiently large $j$ and every $N \in [2^j,2^{j+1})\cap\N$,
		\[
		\den N x y \ > \ 2^{j(\theta-\eps)} \, .
		\]
		For such $N$,
		\[
		\frac{\log \den N x y}{\log N} \ > \ \frac{\log 2^{j(\theta-\eps)}}{\log 2^{j+1}} \ = \ \frac{j}{j+1}(\theta-\eps) \, .
		\]
		It follows that
		\[ \orbitexpoliminf(\alpha,\beta) \ \geq \ \theta - \eps \, . \]
		Letting $\eps \downarrow 0$ proves the result.
	\end{proof}
	
	\subsection{Proof of Theorem~\ref{main thm 1}} \label{subsec: proof of thm 1}
	
	Suppose $\alpha \in \R$ satisfies $\rho(\alpha)>1$. Take $G_\alpha\coloneq (0,1)\smallsetminus B(\alpha) $, where $B(\alpha)$ is as defined in \eqref{B alpha definition}, and let $\beta\in G_\alpha$. Theorem~\ref{thm: upper bounds on expoliminf and expolimsup} gives $\orbitexpolimsup(\alpha,\beta)\leq\expolimsup(\alpha,\beta)\leq\rho(\alpha)/(\rho(\alpha)+1)$, and Theorem~\ref{thm: orbitexpolimsup lower bound mu-1/mu} gives the corresponding lower bound. Next, Theorem~\ref{thm: upper bounds on expoliminf and expolimsup} gives $\expoliminf(\alpha,\beta)\leq1/2$, while Theorem~\ref{thm: expoliminf lower bound of 1/2} gives $\expoliminf(\alpha,\beta)\geq1/2$. Finally, Theorem~\ref{thm: orbitexpoliminf upper bound} gives $\orbitexpoliminf(\alpha,\beta)\leq\min\{1/2,1/\rho(\alpha)\}$, and Theorem~\ref{thm: orbitexpoliminf lower bound} gives the corresponding lower bound.
	
	The case $\rho(\alpha)=1$ was done in Section~\ref{sec: the case mu alpha equals 2}, where the conclusions hold for every $\beta\in(0,1)$.

	\section{Lower bounds that hold for all $\beta$}\label{sec: universal lower bounds}
	
	\begin{lemma}\label{lem: application of circle progression count}
		Let $\alpha \in \R\smallsetminus\Q$ have convergents $(p_m/q_m)$. Fix $m\in\N$, and write $q \coloneq q_m$, $Q \coloneq q_{m+1}$, and $R \coloneq \lfloor Q/q\rfloor$. For every arc $J\subset\T$ with $\leb(J)\leq 1/(4q)$, one has
		\[
		\#\left\{ i \in [0,R)\cap\Z : \{(iq+n)\alpha \bmod 1 : n \in [0,q) \cap \Z \} \cap J \neq \varnothing \right\} \ \geq \ Q \leb(J)-4 \, .
		\]
	\end{lemma}
	\begin{proof}
		The estimate is trivial if $\leb(J) \leq 2/Q$, so assume otherwise.
		By \eqref{kalpha approximation} at index $m+1$,  for each $0\le n<Q$, 
		we have $\|( n \alpha -   np_{m+1}/Q)\|   < \ 1/Q$.
		
		Since $p_{m+1}$ and $Q$ are coprime, the 
		sequence $(\frac{np_{m+1}}{Q}  \bmod 1)_{n=0}^{Q-1}$ runs through the entire $Q$-grid $(\frac{j}{Q}  )_{j=0}^{Q-1}$.
		Let $J'$ be obtained by removing length $1/Q$ from each end of $J$.
		Every $Q$-grid point in $J'$ has its corresponding orbit point in $J$.
		An arc of length $\leb(J')$ contains at least $Q\leb(J')-1$ grid points,
		so
		\[
		\#\{0\leq n<Q:n\alpha\bmod1\in J\}\geq Q\leb(J)-3.
		\]
		When $q>1$, best approximation gives
		$\|h\alpha\|\geq\|q_{m-1}\alpha\|>1/(2q)$ for $0<|h|<q$.
		Since $\leb(J)\le1/(4q)$, two indices in a $q$-block cannot both give points in $J$. Thus each $q$-block visits $J$ at most once (also trivially when $q=1$), as does the remainder $[Rq,Q)\cap\Z$. The number of   full $q$-blocks that visit $J$ is consequently at least $Q\leb(J)-4$, as required.
	\end{proof}
	
	In the next lemma, given a matching between the two orbit segments, we partition blocks of length $Q$ into blocks of length $q$ and assign a sign to every clean $q$-block according to the relative displacement of the matched orbit segments. In each $Q$-block, either both signs occur, forcing a large change in the matching shift, or one sign is absent, in which case visits to a short boundary arc force many $q$-blocks to be non-clean. Lemma~\ref{lem: lcs basics} then yields the stated lower bound on edit distance. Define $(t)_+ = \max\{t,0\}$ for $t \in \R$.
	
	\begin{lemma}\label{lem: clean intervals, edit lower bound}
		Let $\alpha\in\R\smallsetminus\Q$ have convergents $(p_m/q_m)$, and let $\beta\in(0,1)$. There exists $m_0=m_0(\alpha,\beta)$ such that the following holds for every $m\ge m_0$. Write $q\coloneq q_m$ and $Q\coloneq q_{m+1}$, and suppose $H>0$ satisfies
		\begin{equation}\label{extra hypothesis}
			0<|h|<Q/5\quad\text{and}\quad\|h\alpha\|<4/q \quad\Longrightarrow\quad |h|\ge H/16 \qquad(h\in\Z).
		\end{equation}
		Let $B\in\N$, $N\coloneq BQ$, $L\ge0$, and $0<u<1/(4q)$. If $x,y\in\T\smallsetminus\Lambda^*(N)$ satisfy $\den Nxy\le L$ and
		\begin{equation}\label{separation hypothesis}
			\|x+j\alpha-y\|>u\qquad\text{for every }j\in[-L,L]\cap\Z \, ,
		\end{equation}
		then
		\begin{equation}\label{eq: clean intervals, edit lower bound}
			\den Nxy \ge \frac 1{64} \min\{Q,B(uQ-8)_+,BH\}.
		\end{equation}
		In particular, if $\sigma>1$ and
		\begin{equation}\label{eventual growth condition}
			q_{k+1}\le q_k^\sigma
		\end{equation}
		for all sufficiently large $k$, then one may take $H=q^{1/\sigma}$ for all sufficiently large $m$.
	\end{lemma}
	
	\begin{proof}
		Choose $m_0$ so that $q>\max\{k_0(\alpha,\beta),4,2/\|\beta\|\}$, where $k_0$ is as in Lemma~\ref{lem: same code iff same interval}. Put $\Delta \coloneq\den Nxy$. If $\Delta\ge Q/10$, then the conclusion is immediate, so assume $\Delta<Q/10$. Fix a longest common subsequence as in Notation~\ref{notation block for clean and shift}. For integers $0\le b<B$ and $0\le r<\lfloor Q/q\rfloor$, put
		\[
		I(b,r) \coloneq [bQ+rq,bQ+(r+1)q)\cap\Z \, .
		\]
		
		\medskip
		\noindent\textbf{Claim 1.}
		Every clean interval $I=I(b,r)$ of shift $s_I$ has a unique $\delta_I \in (-2/q,2/q)$ congruent to $x-y-s_I\alpha$ modulo $1$, and $|\delta_I|>u$.
		
		\smallskip
		
		Indeed, the matched $q$-words have starting points outside $\Lambda^*(q)$, because $x,y \notin \Lambda^*(N)$. By Lemmas~\ref{lem: same code iff same interval} and \ref{lem: ialpha almost uniform}(b), those starting points lie in the same partition interval, of length at most $1/q+1/Q<2/q$. This gives $\delta_I$, uniquely since $q>4$. Also $|s_I| \leq \Delta \leq L$ by Lemma~\ref{lem: lcs basics}, so \eqref{separation hypothesis} gives $|\delta_I|>u$.
		
		\medskip
		\noindent\textbf{Claim 2.}
		Call a clean interval $I=I(b,r)$ positive or negative according to the sign of $\delta_I$. A positive interval contains no $n$ with $x+n\alpha\bmod1\in J_+\coloneq(0,u/2)$; a negative interval contains no $n$ with $x+n\alpha\bmod1\in J_-\coloneq(1-u/2,1)$.
		
		\medskip 
		
		Indeed, if $\delta_I>0$ and $z=x+n\alpha\bmod1\in J_+$, then $z\in(0,\beta)$ whereas its matched point $1+z-\delta_I$ lies in $(\beta,1)$, since $u<\delta_I<2/q<\|\beta\|$. If $\delta_I<0$ and $z\in J_-$, then $z\in(\beta,1)$ whereas its matched point $z+|\delta_I|-1$ lies in $(0,\beta)$. Both cases contradict equality of the matched symbols.

		\medskip
		\noindent\textbf{Claim 3.}
		For each $b$, if both signs occur among the clean intervals $I(b,r)$ for $0 \le r <\lfloor Q/q \rfloor$, then  two of these clean intervals  have respective shifts $s$ and $s'$ differing by at least $H/16$. If a sign is absent, at least $M \coloneq (uQ/2-4)_+$ of the intervals $I(b,r)$ are not clean.
		
		\medskip
		
		If both signs occur, then take clean intervals of opposite signs, with shifts $s,s'$. Claim 1 gives
		\[
		0<|s-s'| \le 2\Delta<Q/5, \qquad \|(s-s')\alpha\|<4/q,
		\]
		so \eqref{extra hypothesis} gives $|s-s'|\ge H/16$. If no negative clean interval occurs, then apply Lemma~\ref{lem: application of circle progression count} to the arc $J_+-x-bQ\alpha$, of length $u/2$. At least $M$ of the intervals $I(b,r)$ visit $J_+$ under $n\mapsto x+n\alpha\bmod1$; none can be clean by Claim 2. If no positive clean interval occurs, then we apply the same lemma to $J_--x-bQ\alpha$; there are at least $M$  intervals that visit it, and they cannot be clean either.
		
		\medskip
		\noindent\textbf{Adding the bounds.}
		Let $S$ count the values of $b$ for which both signs occur, and write $s(j)$ for the shift of the $j$'th matched pair. For each such $b$, choose one matched pair from each of two clean intervals of opposite signs, with matching indices $a_b<c_b$.
		Then
		\[
		\sum_{j=a_b}^{c_b-1}|s(j+1)-s(j)| \ge |s(c_b)-s(a_b)|\ge H/16.
		\]
		The ranges $[a_b,c_b)$ are disjoint, since the chosen pairs have their first coordinates in distinct blocks $[bQ,(b+1)Q)$. Lemma~\ref{lem: lcs basics} therefore gives
		\[
		SH/16\le2\Delta,\qquad (B-S)M\le2\Delta.
		\]
		Adding these inequalities yields
		\[
		4\Delta\ge B\min\{H/16,M\} \ge\frac B{16}\min\{H,(uQ-8)_+\},
		\]
		which proves \eqref{eq: clean intervals, edit lower bound}.
		
		\medskip
		\noindent\textbf{The choice $H=q^{1/\sigma}$.}
		Assume \eqref{eventual growth condition} holds eventually. For  $q=q_m$, choose $n$ such that $q_n\le q/8<q_{n+1}$; for large $m$, we have $q_{n+1}\le q_n^\sigma$. If $0<|h|<q_n$, then best approximation and \eqref{helpful inequality on q_n alpha} give
		\[
		\|h\alpha\|\ge\|q_{n-1}\alpha\|>\frac1{2q_n}\ge\frac4q.
		\]
		Thus, for $h\ne0$, the inequality $\|h\alpha\|<4/q$ forces
		\[
		|h|\ge q_n\ge q_{n+1}^{1/\sigma} >(q/8)^{1/\sigma}\ge q^{1/\sigma}/16,
		\]
		which verifies \eqref{extra hypothesis} with the stated choice of $H$.
	\end{proof}

	The next proposition upgrades the previous lemma to a more useful form.
	
	\begin{proposition} \label{prop: edit distance lower bound utility} Let $\alpha\in\R\smallsetminus\Q$ have convergents $(p_m/q_m)$, let $\beta\in(0,1)$, and suppose that $q_{k+1}\le q_k^\sigma$ for some $\sigma > 1$ and all sufficiently large $k$. 
		\begin{enumerate}[label=\textup{(\alph*)}]
			\item There is an absolute constant $c > 0$ and $m_1=m_1(\alpha,\beta,\sigma)$ such that the following holds. Let $m\geq m_1$, write $q \coloneq q_m$ and $Q \coloneq q_{m+1}$, suppose $Q > 80q$, let $B \in \N$, and put $N \coloneq BQ$. Then, for uniform and independent $x,y \in \T$,
			\begin{equation} \label{eqedit3}
				\mathbb{P}\bigl ( \den N x y \geq c \min\bigl\{Q,\frac{BQ}{q},Bq^{1/\sigma} \bigr\} \bigr) \ \geq \ \frac{1}{5} \, .
			\end{equation}
			\item Suppose in addition that $Q\leq 2q^2$ and
			\begin{equation}\label{extra hypothesis 2}
				0<|h|<Q/5 \quad\text{and}\quad \|h\alpha\|<4/q \quad\Longrightarrow\quad |h|\geq q/16 \qquad(h\in\Z).
			\end{equation}
			Then
			\begin{equation}\label{eqedit3 under extra hypothesis 2}
				\mathbb{P}\left(
				\den Nxy \geq \frac{1}{640}\min\{Q,BQ/q\} \right) \ \geq \ \frac{1}{5} \, .
			\end{equation}
		\end{enumerate}
	\end{proposition}
	\begin{proof} Set $c\coloneq1/640$, and choose $m_1$ so that Lemma~\ref{lem: clean intervals, edit lower bound} applies whenever $m\ge m_1$. Let $E$ be the complement of the closed intervals of length $4/(5q)$ centered at the points of the rational grid $\{r/q:0\leq r<q\}$. Then $\leb(E)=1/5$. 
		For $x,y\notin\Lambda^*(N)$, put $\Delta \coloneq \den Nxy$.\\
		\noindent \textbf{(a)} For every integer $j$ with $|j|<Q/5$, the estimate
		\[
		\left\|j\alpha-jp_m/q\right\|<1/(5q)
		\]
		shows that $t\in E$ implies $\|j\alpha-t\|>1/(5q)$.  Suppose $y-x\in E$. If $\Delta \geq Q/10$, the asserted lower bound holds. Otherwise, apply Lemma~\ref{lem: clean intervals, edit lower bound} with $L=\Delta$ and $u=1/(5q)$. Since $Q>80q$,
		\[
		\Delta \geq \frac1{64} \min\{Q,B(Q/(5q)-8),Bq^{1/\sigma}\} \geq c\min\{Q,BQ/q,Bq^{1/\sigma}\} \, .
		\]
		Since $y-x$ is uniform in $\T$, we obtain $\mathbb P(y-x \in E) = 1/5$ and  \eqref{eqedit3} follows. 
		
		\medskip
		\noindent\textbf{(b)} Suppose $y-x\in E$. If $\Delta\geq Q/10$, the asserted lower bound holds. Otherwise, the choice of $E$ gives
		\[
		\|x+s\alpha-y\|>1/(5q) \qquad\text{for every }s\in[-\Delta,\Delta]\cap\Z \, .
		\]
		Apply Lemma~\ref{lem: clean intervals, edit lower bound} with $L=\Delta$, $u=1/(5q)$, and $H=q$; the required hypothesis \eqref{extra hypothesis} follows from \eqref{extra hypothesis 2}.
		Since $Q>80q$ and $Q\leq2q^2$, we have
		\[
		\frac{Q}{5q}-8\geq\frac{Q}{10q} \; \text{ and } \; Bq\geq\frac{BQ}{2q}.
		\]
		Consequently,
		\[
		\Delta \geq \frac {1}{64} \min\{Q,B(Q/(5q)-8),Bq\} \geq \frac{1}{640}\min\{Q,BQ/q\} \, .
		\]
		Finally,  $\mathbb P(y-x\in E)=1/5$ and 
		\eqref{eqedit3 under extra hypothesis 2} follows.
	\end{proof}
	
	The next lemma says that, at every sufficiently large scale $N$, two random orbit codings are unlikely to have edit distance below $N^{\theta_\sigma-\varepsilon}$, where $\theta_\sigma$ is the key exponent defined below.
	
	\begin{lemma}\label{lem: small edit distance unlikely}
		Let $\alpha\in\R\smallsetminus\Q$ have convergents $(p_m/q_m)$, let $\beta\in(0,1)$, and suppose that $q_{k+1}\le q_k^\sigma$ for some $\sigma>1$ and all sufficiently large $k$. Let $x,y\in\T$ be uniform independent. 
		\begin{enumerate}[label=\textup{(\alph*)}]
			\item Put
			\[
			\theta \ = \ \theta_\sigma \ \coloneq \ \min\left\{\frac12,
			\frac{\sigma}{\sigma^2+\sigma-1}\right\}.
			\]
			For every $\eps\in(0,\theta)$, for all sufficiently large $N$,
			\begin{equation*}
				\mathbb{P} \bigl( \den N x y \leq N^{\theta-\eps} \bigr) \ \leq \ 500N^{-\eps} \, .
			\end{equation*}
			\item Suppose in addition that $q_{m+1} \leq 2q_m^2$ for all sufficiently large $m$ and, for every such $m$ with $q_{m+1}>2q_m$,
			\begin{equation}\label{extra hypothesis 2 at large jumps}
				0<|h|<q_{m+1}/5 \ \text{ and } \ \|h\alpha\|<4/q_m
				\quad\Longrightarrow\quad |h|\geq q_m/16
				\qquad (h\in\Z) \, .
			\end{equation}
			Then, for every $\eps\in(0,1/2)$, for all sufficiently large $N$,
			\begin{equation*}
				\mathbb{P}\bigl(\den Nxy\leq N^{1/2-\eps}\bigr)
				\ \leq \ 500N^{-\eps} \, .
			\end{equation*}
		\end{enumerate}
	\end{lemma}
	
	\begin{proof} \textbf{(a)} Write  $\ell \coloneq \lfloor N^{\theta-\eps}\rfloor$. Choose consecutive denominators
		\[
		q \coloneq q_m \le N^\theta < Q  \coloneq q_{m+1} \; \text{ and write } \; B \coloneq \lfloor N/Q\rfloor.
		\]
		For sufficiently large $N$, we have $m\ge m_0$ from Lemma~\ref{lem: clean intervals, edit lower bound} and $Q+1>k_0$ from Lemma~\ref{lem: same code iff same interval}. Since $\theta(\sigma+1-1/\sigma) \le 1$, for all sufficiently large $N$,
		\begin{equation*}
			Bq^{1/\sigma} \ge \frac{N}{2q^{\sigma-1/\sigma}} \ge \frac 12 N^{1-\theta(\sigma-1/\sigma)} \ge \frac 12 N^\theta \, .
		\end{equation*}
		
		If $4(Q+1)\ell\le N$, Lemma~\ref{lem: utility lemma on diameter}(a), applied at index $m+1$, gives
		$$\mathbb P(\den Nxy\le\ell) \ \le \ 12\ell/Q \ \le \ 12N^{-\eps} \,.$$
		Otherwise $Q>N/(8\ell)$, so
		\begin{equation}\label{hard case asymptotics}
			\frac qQ<8N^{-\eps},\qquad \frac{q\ell}{N}\le N^{-\eps},\qquad \frac\ell Q<8N^{-2\eps} \, .
		\end{equation}
		Set
		\[
		u \coloneq \frac{9}{Q}+\frac{128\ell}{N} \, ;
		\]
		then \eqref{hard case asymptotics} gives $qu = 9q/Q+128q\ell/N \le 200N^{-\eps}$, so $u<1/(4q)$ for large $N$. Outside the null set of pairs with a coordinate in $\Lambda^*(BQ)$, the event $\den N x y \le \ell$ forces
		\begin{equation}\label{small edit forces orbit proximity}
			\|x+s\alpha-y\| \le  u \quad \text{for some integer } |s| \le \ell \, .
		\end{equation}
		Indeed, monotonicity gives $\den{BQ}xy\le\den Nxy\le\ell$. If every distance in \eqref{small edit forces orbit proximity} were greater than $u$, the hypothesis \eqref{separation hypothesis} of Lemma~\ref{lem: clean intervals, edit lower bound} would hold for this prefix with $L=\ell$, so
		\[
		64\ell\ge\min\{Q,B(uQ-8),Bq^{1/\sigma}\} \, .
		\]
		The first and third terms exceed $64\ell$ for large $N$, whereas $$B(uQ-8) = B+128BQ\ell/N > 64\ell \,,$$ a contradiction.
		
		Since $|\alpha-p_m/q|<1/(qQ)$, condition \eqref{small edit forces orbit proximity} places $y-x$ within $u+\ell/(qQ)$ of the grid $\{r/q:0\le r<q\}$ on $\T$. This difference is uniform, so a union bound and \eqref{hard case asymptotics} give
		\[
		\mathbb P(\den Nxy\le\ell) \ \le \  2q\left(u+\frac{\ell}{qQ}\right) = \frac{18q}{Q}+\frac{256q\ell}{N}+\frac{2\ell}{Q} \le 416 N^{-\eps} \, .
		\]
		\medskip 
		
		\noindent \textbf{(b)} Fix $0<\eps<1/2$, put $\ell \coloneq \lfloor N^{1/2-\eps}\rfloor$,  choose consecutive denominators
		\[
		q=q_m\le\sqrt N/2<Q=q_{m+1}  \; \text{ and write } \;  B \coloneq \lfloor N/Q\rfloor \, .
		\]
		For all sufficiently large $N$, we again have $m\ge m_0$ and $Q+1>k_0$, and the additional growth hypothesis gives
		\[
		Q\le2q^2\le N/2,\qquad BQ\ge N/2,\qquad Q>\sqrt N/2,\qquad Bq\ge\frac{N}{4q}\ge\sqrt N/2 \, .
		\]
		If $4(Q+1)\ell\le N$, Lemma~\ref{lem: utility lemma on diameter}(a) gives
		\[
		\mathbb P(\den Nxy\le\ell) \ \le \ \frac{12\ell}{Q} \ \le \ 24N^{-\eps} \, .
		\]
		It remains to show the case $4(Q+1)\ell>N$. Then $Q>N/(8\ell)$, and hence
		\begin{equation}\label{hard case asymptotics under extra hypothesis 2}
			\frac qQ < 4N^{-\eps},\qquad \frac{q\ell}{N}\le\tfrac12N^{-\eps},\qquad \frac\ell Q<8N^{-2\eps} \, .
		\end{equation}
		In particular, $Q>2q$ for large $N$, so the additional hypothesis \eqref{extra hypothesis 2 at large jumps} holds at the chosen index $m$.
		
		Put $u \coloneq 9/Q+128\ell/N$. By \eqref{hard case asymptotics under extra hypothesis 2}, $qu<100N^{-\eps}$, so $0<u<1/(4q)$ for large $N$. Suppose that $\den Nxy\le\ell$ and  $x,y \notin \Lambda^*(BQ)$.  Next, we will verify the following
		
		\smallskip
		\noindent{\bf Claim $(*)$:}  
		$\|x+s\alpha-y\|\le u$ for some integer $|s| \le \ell$.
		
		Indeed, if all these distances were greater than $u$, we could apply \eqref{eq: clean intervals, edit lower bound} from Lemma~\ref{lem: clean intervals, edit lower bound} to the prefix of length $BQ$, with $L=\ell$ and $H=q$.
		
		Its hypotheses hold: monotonicity gives $\den{BQ}xy\le\ell$, and \eqref{extra hypothesis 2 at large jumps} gives precisely the required bound $|h|\ge H/16=q/16$. The conclusion would be
		\[
		64\ell\ge\min\{Q,B(uQ-8),Bq\}.
		\]
		But $Q$ and $Bq$ are at least $\sqrt N/2$, and therefore exceed $64\ell$ for large $N$. The middle term satisfies
		\[
		B(uQ-8)=B+\frac{128BQ\ell}{N}\ge B+64\ell>64\ell.
		\]
		This contradiction proves the claim $(*)$.
		
		\smallskip
		For the integer $s$ just obtained, $\|s\alpha-sp_m/q\|\le\ell/(qQ)$. Thus $y-x$ is within $u+\ell/(qQ)$ of one of the $q$ points $r/q$, $0\le r<q$. Since $y-x$ is uniform on $\T$, a union bound followed by \eqref{hard case asymptotics under extra hypothesis 2} yields
		\begin{align*}
			\mathbb P(\den Nxy\le\ell) & \le 2q\left(u+\frac{\ell}{qQ}\right)
			= \frac{18q}{Q}+\frac{256q\ell}{N}+\frac{2\ell}{Q}\\
			&\le200N^{-\eps}+16N^{-2\eps} \le216N^{-\eps}.
		\end{align*}
		Both cases give the stated bound $500N^{-\eps}$.
	\end{proof}
	
	\subsection{Universal lower bound on limsup exponent for typical orbits}
	
	\begin{theorem}\label{thm: orbitexpolimsup universal lower bound}
		Let $\alpha \in \R \smallsetminus \Q$ satisfy $\rho \coloneq \rho(\alpha) > 1$. For every $\beta \in (0,1)$,
		\begin{equation}\label{orbitexpolimsup universal lower bound eqn}
			\orbitexpolimsup(\alpha,\beta) \ \geq \ \min\left\{\frac\rho{\rho+1},
			\frac{\rho^2}{2\rho^2-1}\right\} \, ,
		\end{equation}
		where the right-hand side equals $1/2$ when $\rho = \infty$.
	\end{theorem}
	\begin{remark}
		The right-hand side of \eqref{orbitexpolimsup universal lower bound eqn} is $\rho/(\rho+1)$ if and only if $\rho(\alpha) \leq \varphi$.
	\end{remark}
	\begin{proof}
		For $\rho=\infty$, use Proposition~\ref{prop: orbitexpolimsup at least 1/2}. Otherwise fix $\sigma > \rho$ and by \eqref{definition of mu} choose a subsequence of consecutive denominators $q=q_m$ and $Q=q_{m+1}$ with
		$\log Q/\log q \longrightarrow \rho$. Put
		\[
		K \ \coloneq \ \min\{Q/q,q^{1/\sigma}\} \, , \qquad B\coloneq\lceil Q/K\rceil \, , \qquad N\coloneq BQ \, .
		\]
		Let $x,y \in \T$ be uniform and independent. Write $A_m$ for the event $\{\den Nxy \geq cQ\}$ for $m$ in the subsequence. For all sufficiently large $m$, Proposition~\ref{prop: edit distance lower bound utility}(a) gives
		\[
		\mathbb P(A_m) \ \geq \ 1/5 \, ,
		\]
		since $BK \geq Q$. We deduce by Fatou's lemma that $\leb^2(\limsup_m A_m) \geq 1/5$. Since
		\[
		\frac{\log N}{\log q} \ \longrightarrow \
		2\rho-\min\{\rho-1,1/\sigma\} \ = \ \max\{\rho+1,2\rho-1/\sigma\} \, ,
		\]
		we conclude that
		\[
		\orbitexpolimsup(\alpha,\beta) \ \geq \ \frac\rho{\max\{\rho+1,2\rho-1/\sigma\}} \, .
		\]
		Letting $\sigma \downarrow \rho$ proves the theorem.
	\end{proof}

	\subsection{Universal lower bound on liminf exponent for typical orbits}

	\begin{theorem}\label{thm: orbitexpoliminf universal lower bound}
		Let $\alpha\in\R\smallsetminus\Q$ and $\beta\in(0,1)$. Write $\rho = \rho(\alpha)$. Then
		\begin{equation}\label{orbitexpoliminf universal lower bound eqn}
			\orbitexpoliminf(\alpha,\beta) \ \geq \ \min\left\{\frac{1}{2},\frac{\rho}{\rho^2+\rho-1}\right\},
		\end{equation}
		where the right-hand side equals $0$ when $\rho=\infty$.
	\end{theorem}
	\begin{remark}
		The right-hand side of \eqref{orbitexpoliminf universal lower bound eqn} is $1/2$ if and only if $\rho(\alpha) \leq \varphi$.
	\end{remark}
	
	\begin{proof}
		The asserted inequality is immediate when $\rho=\infty$, so assume $\rho<\infty$. Fix $\sigma>\rho$. Then \eqref{eventual growth condition} holds for all sufficiently large indices. Put
		\[
		\theta_\sigma \ \coloneq \ \min\left\{\frac{1}{2},\frac{\sigma}{\sigma^2+\sigma-1}\right\},
		\]
		and fix $\eps\in(0,\theta_\sigma)$.
		
		Let $x,y\in\T$ be uniform and independent. For $j\in\N$, define the event
		\[
		A_j \ \coloneq \ \left\{ \exists N\in[2^j,2^{j+1}) : \den N x y \leq \left\lfloor 2^{j(\theta_\sigma-\eps)}\right\rfloor \right\}
		\ = \
		\left\{ \den {2^j} x y \leq \left\lfloor 2^{j(\theta_\sigma-\eps)}\right\rfloor \right\},
		\]
		where the equality follows from the monotonicity of edit distance recorded below \eqref{edit distance is subadditive}. For all sufficiently large $j$, Lemma~\ref{lem: small edit distance unlikely}(a) gives $\mathbb{P}(A_j) \leq 500\cdot 2^{-j\eps}$, so Borel--Cantelli implies that, for almost every $(x,y) \in \T^2$, for all sufficiently large $j$ and every $N \in [2^j,2^{j+1})$,
		\[
		\den N x y \ > \ 2^{j(\theta_\sigma-\eps)}.
		\]
		For such $N$,
		\[
		\frac{\log \den N x y}{\log N} \ > \ \frac{\log 2^{j(\theta_\sigma-\eps)}}{\log 2^{j+1}} \ = \ \frac{j}{j+1}(\theta_\sigma-\eps) \, .
		\]
		It follows that $\orbitexpoliminf(\alpha,\beta) \geq \theta_\sigma-\eps$. Letting $\eps \downarrow 0$ and then $\sigma \downarrow \rho$ gives the required
		\[
		\orbitexpoliminf(\alpha,\beta) \ \geq \ \min\left\{\frac{1}{2},\frac{\rho}{\rho^2+\rho-1}\right\} \, . \qedhere
		\]
	\end{proof}
	
	The following theorem restates Theorem~\ref{main thm thresholds}(a).
	\begin{theorem}\label{main thm golden less}
		Let $\alpha\in\R\smallsetminus\Q$ satisfy $\rho(\alpha)\leq\varphi$.
		Then \eqref{display for thm 1} and \eqref{display for thm 2} hold for
		every $\beta\in(0,1)$.
	\end{theorem}
	\begin{proof}
		The case $\rho(\alpha)=1$ was shown in Section~\ref{sec: the case mu alpha equals 2}, so suppose $\rho(\alpha) \in (1, \varphi]$. Theorem~\ref{thm: upper bounds on expoliminf and expolimsup} shows $\orbitexpolimsup(\alpha,\beta)  \leq \expolimsup(\alpha,\beta) \leq  \frac{\rho(\alpha)}{\rho(\alpha)+1}$, and Theorem~\ref{thm: orbitexpolimsup universal lower bound} gives the corresponding lower bound. Theorem~\ref{thm: upper bounds on expoliminf and expolimsup} shows $\orbitexpoliminf(\alpha,\beta) \leq \expoliminf(\alpha,\beta) \leq \frac{1}{2}$ and Theorem~\ref{thm: orbitexpoliminf universal lower bound} gives the lower bound.
	\end{proof}
	
	The following theorem restates Theorem~\ref{main thm thresholds}(b1).
	\begin{theorem}\label{main thm golden more first part}
		Suppose $\rho \in (\varphi,2]$. There are uncountably many $\alpha \in \R$ with $\rho(\alpha)=\rho$ so that, for every $\beta \in (0,1)$,\begin{equation}\label{display for thm 1, later}
			\expolimsup(\alpha,\beta) \ = \ \orbitexpolimsup(\alpha,\beta) \ = \ \frac{\rho}{\rho+1}    \, ,
		\end{equation}
		\begin{equation} \label{display for thm 2, later}
			\expoliminf(\alpha,\beta) \ = \ \orbitexpoliminf(\alpha,\beta) \ = \ \frac{1}{2} \, .
		\end{equation}
	\end{theorem}
	\begin{proof}
		Fix $\rho\in(\varphi,2]$, and let $\alpha=[0;a_1,a_2,\ldots]$ satisfy
		\[
		\left|a_{m+1}-q_m^{\rho-1}\right|<2 \quad\text{if $m$ is a power of $2$} \, \text{ and } \, a_{m+1}=1\quad\text{otherwise}.
		\]
		For fixed $j$, we abbreviate $q=q_{2^j}$ and $Q=Q_j\coloneq q_{2^j+1}$.
		The denominator recurrence gives $Q \leq 2q^2$ for all sufficiently large $j$, along with
		\begin{equation}\label{mtgmfp1}
			\lim_{j\to\infty} \frac{\log Q_j}{\log q_{2^j}} = \rho \,  \quad \text{ and } \quad \lim_{j\to\infty} \frac{q_{2^{j-1}+1}}{q_{2^j}} = 0 \, .
		\end{equation}
		For $m$ not of the form $2^j$, we have $q_{m+1}<2q_m$; hence \eqref{definition of mu} gives $\rho(\alpha)=\rho$, and $q_{m+1} \leq 2q_m^2$ for all sufficiently large $m$.
		At every sufficiently large index $m=2^j$, there are at least two admissible choices for $a_{m+1}$. Thus this construction gives uncountably many $\alpha$.
		
		\noindent \textbf{Proof of } \eqref{display for thm 1, later}. Theorem~\ref{thm: upper bounds on expoliminf and expolimsup} shows $\orbitexpolimsup(\alpha,\beta) \leq \expolimsup(\alpha,\beta)  \leq \frac{\rho}{\rho+1}$, so \eqref{display for thm 1, later} will follow from $\orbitexpolimsup(\alpha,\beta) \geq \frac{\rho}{\rho+1}$.
		
		We verify the following condition: for all sufficiently large $j$ and all $h\in\Z$,
		\begin{equation}\label{mtgmfp2}
			0<|h|<Q/5 \ \text{ and } \ \|h\alpha\|<4/q \quad\Longrightarrow\quad |h|>q/16 \, .
		\end{equation}
		Indeed, choose $n$ with $q_n\leq|h|<q_{n+1}$. Best approximation and \eqref{helpful inequality on q_n alpha} give $q_{n+1}>q/8$, so \eqref{mtgmfp1} gives $2^{j-1}<n\leq2^j$ for all large $j$. If $n=2^j$, then $|h|\geq q$; otherwise $a_{n+1}=1$, and hence $|h|\geq q_n>q_{n+1}/2>q/16$.
		
		Fix $\beta\in(0,1)$ and choose $\sigma>\rho$. Then \eqref{eventual growth condition} holds for all sufficiently large $k$. Let $x,y\in\T$ be uniform and independent. Applying Proposition~\ref{prop: edit distance lower bound utility}(b) with $B = q$,
		\begin{equation*}
			\mathbb P\left(\den{qQ}{x}{y}\geq\frac{Q}{640} \right) \ \geq \ \frac{1}{5} \, .
		\end{equation*}
		
		Now, let $A_j\coloneq\{(x,y)\in\T^2:\den{N_j}{x}{y}\geq Q_j/640\}$, where $N_j \coloneq q_{2^j}Q_j=qQ$.
		Then
		\[
		\leb^2(\limsup_j A_j) \ \geq \ \limsup_j \leb^2(A_j) \ \geq \ \frac{1}{5} \, \; \, \text{ and  } \quad  
		\frac{\log(Q_j/640)}{\log(N_j)} \longrightarrow \frac{\rho}{\rho+1} \, .
		\]
		Thus the upper exponent is at least $\rho/(\rho+1)$ with positive probability, so $\orbitexpolimsup(\alpha,\beta)  \geq \frac{\rho}{\rho+1}$.
		
		\noindent \textbf{Proof of } \eqref{display for thm 2, later}. Theorem~\ref{thm: upper bounds on expoliminf and expolimsup} shows $\orbitexpoliminf(\alpha,\beta) \leq \expoliminf(\alpha,\beta)  \leq 1/2$, so \eqref{display for thm 2, later} will follow from $\orbitexpoliminf(\alpha,\beta) \geq 1/2$. Fix $\eps\in(0,1/2)$. Since $q_{m+1}<2q_m$ when $m$ is not a power of $2$, the bound $q_{m+1}\leq 2q_m^2$ and \eqref{mtgmfp2} verify the hypotheses of Lemma~\ref{lem: small edit distance unlikely}(b). Therefore, for all sufficiently large $N$,
		\[
		\mathbb P\bigl(\den Nxy\leq N^{1/2-\eps}\bigr) \ \leq \ 500N^{-\eps} \, .
		\]
		Dyadic Borel--Cantelli and monotonicity, as in the proof of Theorem~\ref{thm: orbitexpoliminf universal lower bound}, imply that \\ $\orbitexpoliminf(\alpha,\beta) \geq 1/2-\eps$; letting $\eps \downarrow 0$ gives $1/2$.
	\end{proof}
	
	\section{Upper bounds via round-and-synchronize}\label{sec: exceptional sets}
	We develop in Lemma~\ref{lem: general bound on diameter in terms of previous q_m} an alternative method for bounding edit distance from above by ``rounding and synchronizing''. It is used to prove Theorem~\ref{main thm golden more second part} and Proposition~\ref{main prop high mu}. The proof of Lemma~\ref{lem: general bound on diameter in terms of previous q_m} first approximates $\alpha$ by $p_m/q_m$ and $\beta$ by a multiple of $1/q_m$ (the ``rounding''), then estimates the edit distance blockwise, correcting accumulated drift by deleting $q_{m-1}$ symbols per block (the ``synchronizing''); the resulting words are close in edit distance because they are shifts of a common $q_m$-periodic word.
	
	\begin{lemma}\label{lem: general bound on diameter in terms of previous q_m} 
		Let $\alpha \in \R\smallsetminus\Q$ with convergents $(p_m/q_m)$ and let $\beta \in (0,1)$. Then there exists an absolute constant $C > 0$ with the following property: If $b, m \in \N$ satisfy
		\begin{equation}\label{assumption on beta distance to lambda qm} 
			\mathrm{dist}_\T(\beta,\Lambda(q_m)\smallsetminus \{0\}) \ \leq \  \frac{b}{q_{m+1}} \, ,
		\end{equation}
		then, for all $N \in \N$,
		\begin{equation*}
			\mathrm{diam}_E(\mathcal{W}_{N}) \ \leq \ q_{m}+  C(q_{m-1}+b) \Bigl\lceil\frac{N}{q_{m+1}}\Bigr\rceil \, .
		\end{equation*}
	\end{lemma}
	\begin{remark}
		One may take $C = 16$; no attempt has been made to optimize it.
	\end{remark}
	\begin{proof}
		Lemma~\ref{lem: ialpha almost uniform}(a) implies there exists $j_0 \in (0,q_m) \cap \Z$ such that
		\begin{equation} \label{approximating beta} \Bigl | \beta - \frac{j_0}{q_m} \Bigr | \ \leq \ \frac{b+1}{q_{m+1}} \ \leq \ \frac{2b}{q_{m+1}} \,.
		\end{equation}
		
		Given $x \in [0,1)$ and $n \in \Z$, let
		\[
		x_n \ \coloneq \ x + n\alpha  \quad \text{ and } \quad \tilde{x}_n \ \coloneq \ \frac{\lfloor{q_m x}\rfloor}{q_m} + n\frac{p_m}{q_m} \, .
		\]
		
		Define $W(x) \coloneq (\mathbf{1}_{[0,j_0/q_m)}(x_{n} \bmod 1))_{n=0}^{N-1}$ and $P(x) \coloneq (\mathbf{1}_{[0,j_0/q_m)}(\tilde{x}_{n} \bmod 1))_{n=0}^{N-1}$. Observe that, for $x,y \in [0,1)$, the words $P(x)$ and $P(y)$ are subwords of the same $q_m$-periodic infinite word, hence
		\begin{equation} \label{periodic difference} \max_{x,y\in [0,1)} \editfull (P(x),P(y)) \ \leq \ q_m \, .
		\end{equation}
		
		Let $N^* \coloneq \lceil N/q_{m+1} \rceil$. 
		To prove the lemma, we apply \eqref{periodic difference} and the following two claims:  
		\begin{align}
			\textbf{Claim 1.} \text{ For all } x \in [0,1)\, ,& \quad \editfull(\xi_{[0,N)}(x),W(x)) \ \leq \ 4bN^* \, . \label{claim 1} \\
			\textbf{Claim 2.} \text{ For all } x \in [0,1)\, , & \quad \editfull ( W(x), P(x)) \ \leq \ q_{m-1}N^* \, . \label{claim 2}
		\end{align}
		
		\medskip
		\noindent To show \eqref{claim 1}, let $x \in [0,1)$. By Lemma~\ref{lem: description of three gaps by size and count}(a), the set $[0,1) \smallsetminus\{x_n \bmod 1 : n \in [0,q_{m+1}) \}$ is a disjoint union of arcs, each with length greater than $\frac{1}{2q_{m+1}}$. 
		Thus,  $x_n \bmod 1$ can belong to a given interval of length $2b/q_{m+1}$ for at most $4b$ values of $n \in [0,q_{m+1})$. Therefore, in view of \eqref{approximating beta}, for each $i \in [0,N/q_{m+1} ] \cap \Z$, 
		\begin{equation*}
			\xi_{iq_{m+1}+n}(x) \ = \ \mathbf{1}_{[0,j_0/q_m)}(x_{iq_{m+1}+n} \bmod 1) 
		\end{equation*}
		for all but at most $4b$ values of $n \in [0,q_{m+1})$, yielding \eqref{claim 1}. \medskip
		
		Next, we will prove \eqref{claim 2}. First, assume $\alpha > p_m/q_m$ (so $m$ is even). 
		
		For each $j \in \N \cup \{0\}$, set $n_j \coloneq \min\{ n \geq 0 : x_{n} - \tilde{x}_{n} \geq j/q_m \}$. Since $m$ is even, the function $n \mapsto x_n - \tilde{x}_n$ is increasing, and at each step it increases by $\alpha - p_m/q_m < \frac{1}{q_mq_{m+1}}$ by \eqref{helpful inequality on q_n alpha}, so $n_{j+1}-n_j \geq q_{m+1}$ for all $j \geq 1$.
		
		Let $i \in \N \cup \{0\}$ be such that $N-1 \in [n_{i},n_{i+1})$. Then $i \leq N^*$.
		
		For $n \in [n_j,n_{j+1})$, set $x_n^* \coloneq \tilde{x}_n + \frac{j}{q_m}$. Then $x_n - x^*_n \in [0,\frac{1}{q_m})$, so the word
		\[
		W^*(x) \ \coloneq \ (\mathbf{1}_{[0,j_0/q_m)}(x^*_{n} \bmod 1))_{n=0}^{N-1}
		\]
		satisfies $W^*(x) = W(x)$. We now exhibit a common subsequence of $W^*(x)$ and $P(x)$ with length at least $N - q_{m-1}N^*$. If $i = 0$, then $W^*(x) = P(x)$, so suppose $i > 0$. Observe that, for $n \in [n_j,n_{j+1})$, since $m$ is even, the identity $p_{m-1}q_m - p_mq_{m-1} = (-1)^m$ (see \eqref{pq identity}) implies
		\begin{equation}\label{xstar}
			x_n^* \ = \ \tilde{x}_{n-jq_{m-1}} + jp_{m-1} \ \equiv \ \tilde{x}_{n-jq_{m-1}} \bmod 1 \, .
		\end{equation}
		
		Therefore, writing $\zeta{[c,d)} \coloneq (\mathbf{1}_{[0,j_0/q_m)}(\tilde{x}_n \bmod 1))_{n=c}^{d-1}$ for integers $c < d$, we see that
		\[
		W^*(x) \ = \ B_0\dots B_{i-1}B_i \, ,
		\]
		where $B_i \coloneq \zeta[n_i-iq_{m-1},N-iq_{m-1})$ and $B_j \coloneq \zeta[n_j-jq_{m-1},n_{j+1}-jq_{m-1})$ for $j < i$.
		If $N-n_i > q_{m-1}$, the common subsequence $\zeta{[0,N-iq_{m-1})}$ is found by taking $B_0$, followed by all but the first $q_{m-1}$ symbols of each $B_j$ with $j > 0$. Since $i \leq N^*$, the observation $N - iq_{m-1} \geq N - q_{m-1}N^*$ shows \eqref{claim 2}. If instead $N - n_i \leq q_{m-1}$, then we take the same parts of each block as above, except that we instead discard all of $B_i$. This yields the common subsequence $\zeta[0,n_i - (i-1)q_{m-1})$, which has length at least $N - iq_{m-1} \geq N - q_{m-1}N^*$.
		
		This completes the proof of \eqref{claim 2} under the assumption $\alpha > p_m/q_m$. When $\alpha < p_m/q_m$, follow the same argument with the definitions $n_j \coloneq \min \{n \geq 0 : \tilde{x}_n - x_n > \frac{j-1}{q_m} \}$ and $x_n^* \coloneq \tilde{x}_n - \frac{j}{q_m}$ for $n \in [n_j,n_{j+1})$. Then \eqref{xstar} again will hold and hence \eqref{claim 2}.  
	\end{proof}
	
	\subsection{Upper bounds for $\alpha$ with steady convergent growth}\label{subsec: upper bounds for steady convergent} 
	We are interested in $\alpha$ with finite $\rho(\alpha) > \varphi $ such that the convergents $(p_m/q_m)$ satisfy the ``steady convergent growth'' condition 
	\begin{equation}\label{steady convergent growth}
		q_{m+1} \ = \ q_m^{\rho(\alpha)+o(1)} \qquad \text{ as } \ m \to \infty \,.
	\end{equation}
	
	Fix $\rho > \varphi$ and any  initial string of partial quotients $a_1,a_2,\ldots, a_{m_0}$. Then $q_{m_0}$ is determined. For every sufficiently large $m$, choose $a_{m+1} \in \N$ such that
	\[
	\left\lceil q_m^{\rho-1}\right\rceil \ \leq \ a_{m+1} \ \leq \ \left\lfloor 2q_m^{\rho-1} \right\rfloor .
	\]
	There are at least two choices for $a_{m+1}$ if $m$ is large enough. By the convergent recurrence $q_{m+1} = a_{m+1}q_m + q_{m-1}$, the inequality $q_m^\rho \leq q_{m+1} \leq 3q_m^\rho$ holds for all large $m$, hence every irrational with partial quotients chosen by this procedure satisfies the steady convergent growth condition with $\rho(\alpha) = \rho$ (see \eqref{definition of mu}). The infinitely many independent choices of $a_{m+1}$ produce an uncountable set of such $\alpha$.
	
	Given $\alpha$, define
	\begin{equation} \label{definition of B upper star}
		B^*(\alpha) \ \coloneq \ \bigcup_{M\in\N} \bigcap_{m \geq M} \left\{ \beta \in (0,1) : \mathrm{dist}_\T(\beta,\Lambda(q_m)) \ \leq \ \frac{2}{q_{m+1}} \right\}. 
	\end{equation}
	The next proposition shows that $B^*(\alpha)$ is topologically large.
	\begin{proposition}\label{prop: dense perfect}
		Let $\alpha \in \R\smallsetminus\Q$ with convergents $(p_m/q_m)$ and assume $\rho(\alpha) > 1$. The set $B^*(\alpha)$ defined in \eqref{definition of B upper star} is dense and contains a nonempty perfect set.
	\end{proposition}
	\begin{proof}
		Density follows from the fact that, for each $M \in \N$,
		\[ \Lambda(q_{M}) \smallsetminus \{0\} \subset\bigcap_{m \geq M} \left\{ \beta \in (0,1) : \mathrm{dist}_\T(\beta,\Lambda(q_m)) \ \leq \ \frac{2}{q_{m+1}} \right\}  \, .
		\]
		
		Since $\rho(\alpha) > 1$, by \eqref{definition of mu}, there exists an increasing sequence $A = (m(k))_{k=1}^\infty \subset \N$ such that $q_{m(k)+2} \geq 7 q_{m(k)+1}$ for all $k \ge 1$.
		
		For $x \in [0,1)$ and $\delta \in [0,1]$, we write $B[x, \delta] \coloneq \{y \in [0,1): \| x - y\| \leq \delta \}$.
		
		Choose $n_0$ so large that $x_0 \coloneq -\alpha \bmod 1$ belongs to $\Lambda(q_{n_0})\smallsetminus\{0\}$ and $B[x_0,2/q_{n_0+1}]$ does not contain $0$, and put $S_{n_0}:=\{x_0\}$. Suppose, for some $n \geq n_0$, there exists nonempty $S_{n} \subseteq \Lambda(q_{n})$ such that the sets $B[x,2/q_{n+1}]$ for $x \in S_{n}$ are pairwise disjoint. We claim there exists $S_{n+1} \subseteq \Lambda(q_{n+1})$ such that the sets $B[x,2/q_{n+2}]$ for $x \in S_{n+1}$ are pairwise disjoint,
		\begin{equation*}
			|S_{n+1}| \ = \ \begin{cases} |S_n| & \text{ if } n \not\in A, \\
				2|S_n| & \text{ otherwise,} \end{cases} 
		\end{equation*}
		and $\cup_{x \in S_{n+1}} B[x,2/q_{n+2}] \subseteq \cup_{x \in S_n} B[x,2/q_{n+1}]$.
		
		Indeed, if $n \not\in A$, then $S_{n+1} \coloneq S_n$ satisfies the desired properties. Otherwise, $n = m(k)$ for some $k \in \N$. Given $x \in S_{m(k)}$, choose $y(x) \in \Lambda(q_{m(k)+1})$ so that
		\begin{equation}\label{choosing a second point}
			\frac{4}{q_{m(k)+2}} \ < \ \| x - y(x) \| \ < \ \frac{2}{q_{m(k)+1}} - \frac{2}{q_{m(k)+2}} \, ;
		\end{equation}
		this is possible because $\{z \in\T : \frac{4}{q_{m(k)+2}} < \|x - z\| < \frac{2}{q_{m(k)+1}} - \frac{2}{q_{m(k)+2}} \}$ is a union of two intervals, each of length
		\begin{equation*}
			\frac{2}{q_{m(k)+1}} - \frac{6}{q_{m(k)+2}} \ \geq \ \frac{1}{q_{m(k)+1}} + \frac{1}{q_{m(k)+2}} \ > \ \| q_{m(k)}\alpha\| + \|q_{m(k)+1}\alpha\|,   
		\end{equation*}
		and thus each interval contains a point of $\Lambda(q_{m(k)+1})$ in view of Lemma~\ref{lem: description of three gaps by size and count}(a).
		
		Moreover, \eqref{choosing a second point}  implies that $B[x,2/q_{m(k)+2}]$ and $B[y(x),2/q_{m(k)+2}]$ are disjoint and contained in $B[x,2/q_{m(k)+1}]$. Thus take $S_{n+1} \coloneq \cup_{x \in S_n} \{x,y(x)\}$. By construction, the sets $B[x,2/q_{n+2}]$ for $x \in S_{n+1}$ are pairwise disjoint and $|S_{n+1}| = 2|S_n|$. This completes the proof of the claim. Finally, we observe that
		\[
		\bigcap_{n \geq n_0} \bigcup_{x \in S_n} B[x,2/q_{n+1}]
		\]
		is a nonempty perfect set contained in $B^*(\alpha)$.
	\end{proof}
	
	The following two propositions give, for $\alpha$ with steady convergent growth, matching upper bounds to the lower bounds in Theorems~\ref{thm: orbitexpolimsup universal lower bound} and \ref{thm: orbitexpoliminf universal lower bound}. 
	
	\begin{proposition}\label{prop: steady convergent expolimsup upper}  Let $\alpha \in \R \smallsetminus \Q$ have convergents $(p_m/q_m)$ and suppose that they satisfy  $q_{m+1} = q_m^{\rho + o(1)}$ as $m \to \infty$ for some finite $\rho > \varphi$. Then, for all $\beta \in B^*(\alpha)$,
		\begin{equation*}
			\expolimsup(\alpha,\beta) \ \leq \ \frac{\rho^2}{2\rho^2-1} \, .
		\end{equation*}
	\end{proposition}
	\begin{remark}
		The inequality $\frac{\rho^2}{2\rho^2-1} < \frac{\rho}{\rho+1} = \frac{\rho(\alpha)}{\rho(\alpha)+1}$ holds if and only if $\rho > \varphi$. Thus, this proposition improves the upper bound \eqref{expolimsup upper bound of mu-1/mu} in Theorem~\ref{thm: upper bounds on expoliminf and expolimsup} for uncountably many $\alpha$ with $\rho(\alpha) > \varphi$. 
	\end{remark}
	\begin{proof}[Proof of Proposition~\ref{prop: steady convergent expolimsup upper}]
		Given $\beta \in B^*(\alpha)$, let $m_0 \in \N$ such that, for all $m \geq m_0$, \eqref{assumption on beta distance to lambda qm} holds for $b = 2$. For $m \geq m_0$, by Lemma~\ref{lem: general bound on diameter in terms of previous q_m}, choose an absolute constant $C'$ such that 
		\begin{equation*}
			\mathrm{diam}_E(\mathcal{W}_{N}) \ \leq \ C'\max\{q_{m},q_{m-1}\frac{N}{q_{m+1}}\} \, .
		\end{equation*}
		For $N>1$, restrict to $m$ with $q_m\leq N$ and put $\theta_m\coloneq\log q_m/\log N \in [0,1]$, so that $q_{m-1} = N^{{\theta_m}/\rho + o(1)}$ and $q_{m+1} = N^{\rho{\theta_m}+o(1)}$. Define $f(\theta) \coloneq \max\{\theta,(1/\rho-\rho)\theta+1\}$ for $\theta \geq 0$. Then
		\begin{equation*}
			\frac{\log \mathrm{diam}_E(\mathcal{W}_{N})}{\log N} \ \leq \ f(\theta_m) + o(1) \, \text{ as } m \to \infty\, ,
		\end{equation*}
		
		To choose a suitable denominator, fix $\lambda \in (0,1/\rho)$, to be optimized below, and let $m=m(N)$ be the least index with $q_m\ge N^\lambda$. Then $q_{m-1}<N^\lambda$ and the assumed growth of the denominators gives
		$
		N^\lambda\le q_m=q_{m-1}^{\rho+o(1)} \le N^{\rho\lambda+o(1)}.
		$
		Thus $q_m\le N$ for sufficiently large $N$, while
		\[
		\lambda\le\theta_m=\frac{\log q_m}{\log N} \le \rho\lambda+o(1) \, .
		\]
		The preceding estimate and continuity of $f$ therefore imply
		\begin{equation} \label{fbound}
			\limsup_{N\to\infty}
			\frac{\log\mathrm{diam}_E(\mathcal W_N)}{\log N}
			\le\max_{\theta\in[\lambda,\rho\lambda]}f(\theta) 
			=\max\{\rho\lambda,\,1-(\rho-1/\rho)\lambda\} \,.
		\end{equation}

		The minimum over $\lambda$ of this upper bound   occurs when the two terms agree:
		\[
		\rho\lambda=1-(\rho-1/\rho)\lambda, \qquad\text{that is,} \qquad \lambda=\frac{\rho}{2\rho^2-1} \in (0,1/\rho)\, .
		\]
		Substituting this value of $\lambda$  in \eqref{fbound} gives
		\[
		\limsup_{N\to\infty} \frac{\log\mathrm{diam}_E(\mathcal W_N)}{\log N} \le \frac{\rho^2}{2\rho^2-1} \, . \qedhere
		\]
	\end{proof}

	\begin{proposition}\label{prop: steady convergent expoliminf upper}
		Let $\alpha \in \R \smallsetminus \Q$ have convergents $(p_m/q_m)$ and suppose that they satisfy
		$q_{m+1}=q_m^{\rho+o(1)}$ as $m \to \infty$ for some finite $\rho > \varphi$. Then, for all	$\beta \in B^*(\alpha)$,
		\begin{equation*}
			\expoliminf(\alpha,\beta) \ \le \ \frac{\rho}{\rho^2+\rho-1} \,.
		\end{equation*}
	\end{proposition}
	
	\begin{proof}
		Fix $\beta \in B^*(\alpha)$, and let $C$ be the constant in Lemma~\ref{lem: general bound on diameter in terms of previous q_m}. By the definition of $B^*(\alpha)$, for all large $m$ the hypothesis of that lemma holds with $b=2$. Hence, for every $N$,
		\[
		\mathrm{diam}_E(\mathcal W_N) \ \le \ q_m + C(q_{m-1}+2) \left\lceil \frac{N}{q_{m+1}}\right\rceil \, .
		\]
		Take
		\[
		N_m \ \coloneq \ \left\lfloor\frac{q_mq_{m+1}}{q_{m-1}}\right\rfloor \, .
		\]
		Since $q_{m+1} = q_m^{\rho+o(1)}$ and also $q_m=q_{m-1}^{\rho+o(1)}$, we have $q_{m-1} = q_m^{1/\rho+o(1)}$ and
		\[
		\log N_m \ = \ \bigl(\rho+1-1/\rho+o(1)\bigr)\log q_m \, .
		\]
		Put $C'\coloneq 1+6C$. Since $q_{m-1}+2\leq3q_{m-1}$ and $\lceil N_m/q_{m+1}\rceil\leq2q_m/q_{m-1}$, substitution gives $\mathrm{diam}_E(\mathcal W_{N_m})\leq C'q_m$, and therefore
		\[
		\expoliminf(\alpha,\beta) \ \le \ \frac{1}{\rho+1-1/\rho} \ = \ \frac{\rho}{\rho^2+\rho-1} \, . \qedhere
		\]
	\end{proof}
	
	Recall that the following theorem is a precise form of Theorem~\ref{main thm thresholds}(b2).
	
	\begin{reptheorem}{main thm golden more second part}
		Suppose that $\rho >\varphi$ is finite. Then there is an uncountable subset $A_*(\rho)$ of $\{\alpha: \rho(\alpha)=\rho\}$    so that for each $\alpha \in A_*(\rho)$, there are  uncountably many $\beta$ with
		\begin{equation*}
			\begin{aligned}
				\expolimsup(\alpha,\beta)& = \orbitexpolimsup(\alpha,\beta)
				= \frac{\rho^2}{2\rho^2-1} \, ,\\
				\expoliminf(\alpha,\beta)& = \orbitexpoliminf(\alpha,\beta) = \frac{\rho}{\rho^2+\rho-1} \, .
			\end{aligned}
		\end{equation*}
	\end{reptheorem}
	\begin{proof}
		Let $\alpha$ satisfy \eqref{steady convergent growth}, and let $\beta \in B^*(\alpha)$. The lower bounds follow from Theorem~\ref{thm: orbitexpolimsup universal lower bound} and Theorem~\ref{thm: orbitexpoliminf universal lower bound}, while Propositions~\ref{prop: steady convergent expolimsup upper} and \ref{prop: steady convergent expoliminf upper} give the matching upper bounds. Note $B^*(\alpha)$ is uncountable by Proposition~\ref{prop: dense perfect}.
	\end{proof}

	\subsection{Upper bounds on diameter exponents when $\rho(\alpha) > 2$}
	
	The following proposition shows, when $\rho(\alpha) > 2$, that the upper bound \eqref{expoliminf upper bound of 1/2} in Theorem~\ref{thm: upper bounds on expoliminf and expolimsup} may be improved and that a dense $G_\delta$ set of exceptions is required in the statement of Theorem~\ref{thm: expoliminf lower bound of 1/2}.
	
	Given an infinite set $A \subset \N$, define
	\[
	B_A(\alpha) \ \coloneq  \ \bigcap_{M\in \N} \bigcup_{\substack{m \in A \\ m > M}} \left\{\beta \in (0,1) : \mathrm{dist}_\T(\beta,\Lambda(q_m)) \ < \ \frac{2}{q_{m+1}} \right\} \, .
	\]
	(The choice of the numerator 2 is arbitrary.) This definition expresses $B_A(\alpha)$ as a countable intersection of dense open sets in $(0,1)$, so it is a dense $G_\delta$ set.
	\begin{proposition}\label{prop: expoliminf special upper bound}  Let $\alpha \in \R\smallsetminus\Q$ with convergents $(p_m/q_m)$ and assume that $\rho(\alpha) > 2$. Given $1 < \rho < \rho(\alpha)$, let $A(\rho) \coloneq \{ m \in \N : q_{m+1} \geq q_m^{\rho} \}$. Then, for all $\beta \in B_{A(\rho)}(\alpha)$,
		\begin{equation}\label{liminf small}
			\expoliminf(\alpha,\beta) \ \leq \ \frac{1}{\rho} \, .
		\end{equation}
		Consequently, for all $\beta \in  \displaystyle B_*(\alpha) \coloneq \bigcap_{\rho \in (1, \rho(\alpha))} B_{A(\rho)}(\alpha)$, we have $\expoliminf(\alpha,\beta) = 0$ if $\rho(\alpha) = \infty$ and
		\begin{equation*}
			\expoliminf(\alpha,\beta)  \ \leq \ \frac{1}{\rho(\alpha)} \ < \ \frac{1}{2} \quad \text{ otherwise.} 
		\end{equation*}
	\end{proposition}
	\begin{remark}
		Only the fact that $\rho(\alpha) > 1$ is used in the proof. When $\rho(\alpha) \leq 2$, Theorem~\ref{thm: upper bounds on expoliminf and expolimsup} provides a better upper bound. The intersection defining $B_*(\alpha)$ is unchanged when restricted to all rational $\rho \in (1,\rho(\alpha))$, so $B_*(\alpha)$ is a dense $G_\delta$ set. 
		Note that $B_*(\alpha)$ is a subset of $B(\alpha)$, defined in \eqref{B alpha definition}. 
	\end{remark}
	\begin{proof} Let $\rho < \rho(\alpha)$. Since $\expoliminf(\alpha,\beta) \leq 1$, we may assume $\rho > 1$. By \eqref{definition of mu}, $A(\rho)$ is infinite. Let us now fix $\beta \in B_{A(\rho)}(\alpha)$ and show  \eqref{liminf small}.
		
		Let $A \subseteq A(\rho)$ be an infinite set such that $\mathrm{dist}_\T(\beta,\Lambda(q_m) \smallsetminus \{0\}) < \frac{2}{q_{m+1}}$ for all $m \in A$. For each $m \in A$, the condition \eqref{assumption on beta distance to lambda qm} holds for the choice $b = 2$. Therefore, by Lemma~\ref{lem: general bound on diameter in terms of previous q_m}, there exists an absolute constant $C > 0$ such that, for all $m \in A$,
		\begin{equation*}
			\mathrm{diam}_E(\mathcal{W}_{q_{m+1} }) \ \leq \ Cq_{m} \ \leq \ Cq_{m+1}^{1/\rho} \, ,
		\end{equation*}
		where the second inequality holds since $m \in A(\rho)$, hence
		\begin{equation*}
			\frac{	\log \mathrm{diam}_E(\mathcal{W}_{q_{m+1} }) }{\log q_{m+1}} \ \leq \ \frac{1}{\rho} + \frac{\log C}{\log q_{m+1}}  \, . 
		\end{equation*}
		We have shown \eqref{liminf small}. The result follows.
	\end{proof}
	
	The next proposition shows that when $\rho(\alpha) > 2$, the upper bound \eqref{expolimsup upper bound of mu-1/mu} in Theorem~\ref{thm: upper bounds on expoliminf and expolimsup} may be improved and an exceptional set is required in Theorem~\ref{thm: orbitexpolimsup lower bound mu-1/mu}.
	
	\begin{proposition}\label{prop: expolimsup special upper bound}
		Let $\alpha \in \R\smallsetminus\Q$ with convergents $(p_m/q_m)$ and assume $\rho(\alpha) > 2$. Then $\expolimsup(\alpha,\beta) \leq \frac{2}{3}$ for all $\beta \in B^*(\alpha)$, defined in \eqref{definition of B upper star}.
	\end{proposition}
	\begin{remark} The hypothesis $\rho(\alpha) > 2$ is not used in the proof. However, Theorem~\ref{thm: upper bounds on expoliminf and expolimsup} provides a better upper bound when $\rho(\alpha) \leq 2$.
	\end{remark}
	\begin{proof}
		Let $C$ be the constant in Lemma~\ref{lem: general bound on diameter in terms of previous q_m}.
		
		Given $\beta\in B^*(\alpha)$, choose $m_0$ so that \eqref{assumption on beta distance to lambda qm} holds with $b=2$ for every $m\geq m_0$. For $N>q_{m_0}^3$, choose $m\geq m_0$ such that
		\[
		q_m \ < \ N^{1/3} \ \leq \ q_{m+1} \, .
		\]
		If $q_{m+1}\leq N^{2/3}$, Lemma~\ref{lem: q + N/q bound} gives
		\[
		\mathrm{diam}_E(\mathcal W_N) \ \leq \ q_{m+1}+\frac{11N}{q_{m+1}}+11 \ \leq \ 12N^{2/3}+11 \, .
		\]
		If $q_{m+1}>N^{2/3}$, Lemma~\ref{lem: general bound on diameter in terms of previous q_m} and $q_{m-1}+2\leq3q_m$ instead give
		\[
		\begin{split}
			\mathrm{diam}_E(\mathcal W_N)
			& \ \leq \ q_m+C(q_{m-1}+2)\left(\frac{N}{q_{m+1}}+1\right)\\
			& \ \leq \ N^{1/3}+3C N^{1/3}(N^{1/3}+1)
			\leq (1+6C)N^{2/3} \, .
		\end{split}
		\]
		Thus, with the absolute constant $C' \coloneq \max\{23,1+6C\}$, both cases yield
		\begin{equation*}
			\frac{\log\mathrm{diam}_E(\mathcal W_N)}{\log N} \ \leq \ \frac 23 +\frac{\log C'}{\log N} \, .
		\end{equation*}
		Taking the limsup proves the result.
	\end{proof}

	The following proposition restates Theorem~\ref{main thm thresholds}(c).
	\begin{proposition}\label{main prop high mu}
		Let $\alpha\in\R\smallsetminus\Q$ satisfy $\rho(\alpha)>2$.
		\begin{enumerate}[label=\textup{(\alph*)}]
			\item The set $\{\beta\in(0,1) : \expolimsup(\alpha,\beta) \leq 2/3\}$ is dense and contains a nonempty perfect set.
			\item The set $\{\beta\in(0,1) : \expoliminf(\alpha,\beta) \leq 1/\rho(\alpha)\}$ contains a dense $G_\delta$ subset of $(0,1)$.
		\end{enumerate}
	\end{proposition}
	
	\begin{proof} Part (a) follows from Propositions~\ref{prop: expolimsup special upper bound}~and~\ref{prop: dense perfect}. Part (b) follows from Proposition~\ref{prop: expoliminf special upper bound} and its remark.
	\end{proof}
	
	\subsection{Proof of Theorem~\ref{main thm thresholds}}
	Theorem~\ref{main thm thresholds} is the union of Theorems~\ref{main thm golden less},~\ref{main thm golden more first part},~and~\ref{main thm golden more second part} and Proposition~\ref{main prop high mu}.
	
	\section{Applications}
	\subsection{Circle maps}\label{sec: circle maps} 
	
	Let $\mathrm{Diff}_{r}^+(\T)$ denote, when $r = 0$, the group of orientation-preserving homeomorphisms of $\T$ and, when $r \in [1,\infty)$, the group of orientation-preserving $C^{\lfloor r\rfloor }$-diffeomorphisms of $\T$ whose $\lfloor r \rfloor$-th derivative satisfies the H\"{o}lder condition with exponent $r - \lfloor r \rfloor$. A circle map $f : \T \to \T$ is \textbf{$C^r$-conjugate} to another circle map $g : \T \to \T$ if there exists $h \in \mathrm{Diff}_{r}^+(\T)$ such that $f = h^{-1}\circ g\circ h$.
	
	We continue to identify $\T$ with $[0,1)$, with topology inherited from $\R/\Z$. Fix a homeomorphism $f \in \mathrm{Diff}_0^+(\T)$. A \textbf{lifting} of $f$ is a homeomorphism $L : \R \to \R$ such that $L(x) \bmod 1 = f(x \bmod 1)$ for all $x \in \R$. It follows from the Poincar\'{e} theorem that the limit $\tau(L) \coloneq \lim_{n \to \infty} \frac{L^n(x_0)-x_0}{n}$ exists and does not depend on $x_0 \in \R$, where the exponent in $L^n$ denotes iterated composition. The \textbf{rotation number} of $f$ is $\omega(f) \coloneq \tau(L) \bmod 1$, where $L$ is any lifting of $f$, which is well defined because any two liftings of $f$ differ by an integer constant. It is well known that $\omega(f) \in [0,1) \smallsetminus \Q$ if and only if $f$ is \textbf{aperiodic} (i.e., has no periodic orbits). When $f$ is aperiodic, $f$ is $C^0$-conjugate to an irrational rotation if and only if $f$ has a dense orbit. For details, see the discussion before Theorem 1.2 in \cite{katznelsonornstein}.
	
	A classical theorem of Denjoy \cite{denjoy} says that every aperiodic $f \in \mathrm{Diff}_1^+(\T)$ whose derivative is of bounded variation is $C^0$-conjugate to $R_{\omega(f)}$.
	
	Let $\leb$ denote Lebesgue measure on $\T$. If $f \in \mathrm{Diff}_0^+(\T)$ is $C^0$-conjugate to an irrational rotation $R_{\alpha}$, i.e., $f = h^{-1}\circ R_{\alpha} \circ h$ with $h \in \mathrm{Diff}_0^+(\T)$, then $\nu \coloneq \leb \circ h$ is the unique $f$-invariant Borel probability measure on $\T$. Moreover, if $h \in \mathrm{Diff}_1^+(\T)$, then $\nu$ and $\leb$ are mutually absolutely continuous.

	Arnol'd~\cite{arnold} gave examples of analytic diffeomorphisms $f$ with Liouville rotation number (i.e., $\rho(\omega(f)) = \infty$) where there is no smooth conjugation to a rotation. Herman~\cite{herman} showed that, for all $r \geq 3$, every $f \in \mathrm{Diff}_r^+(\T)$ with rotation number $\omega(f)$ belonging to a full-measure subset of $\{ \alpha : \rho(\alpha) = 1\}$ is $C^{r-2}$-conjugate to $R_{\omega(f)}$; see also \cite{yoccoz}.
	
	Later results of Katznelson--Ornstein \cite{katznelsonornstein} and Khanin--Sinai (see, e.g., Theorem 4.12 in \cite{katznelsonornstein}) imply that if $f\in\mathrm{Diff}_r^+(\T)$ is aperiodic and $r > 1+\rho(\omega(f))$, then $f$ is $C^1$-conjugate to $R_{\omega(f)}$.
	With the above facts in hand, we show how to extend our results to circle maps.
	
	\begin{proof}[Proof of Corollary~\ref{intro corollary for circle maps}] Suppose $f \in \mathrm{Diff}_{0}^+(\T)$ is aperiodic and has a dense orbit. Then there exists $h \in \mathrm{Diff}_0^+(\T)$ with $f = h^{-1} \circ R_{\omega(f)} \circ h$. Replacing $h$ by $R_\theta \circ h$ for some $\theta \in [0,1)$ if necessary, assume $h(0) = 0$. For fixed $\eta \in (0,1)$, $x \in [0,1)$, and $n \geq 0$, we have
		\begin{equation*}
			\mathbf{1}_{[0,\eta)}(f^nx) \ = \ \mathbf{1}_{[0,h(\eta))}\bigl(R_{\omega(f)}^n (h(x))\bigr) \,. 
		\end{equation*}
		We deduce that $\mathcal{W}_N^f = \mathcal{W}_N$ for all $N \in \N$, where for $\mathcal{W}_N$ we take $\alpha = \omega(f)$ and $\beta = h(\eta)$.
		
		By Theorem~\ref{main thm 1}, let $E(\alpha)$ be a Lebesgue null set such that \eqref{display for thm 1} and \eqref{display for thm 2} hold for all $\beta \in (0,1) \smallsetminus E(\alpha)$.
		
		We conclude that \eqref{result for circle maps} holds for all $\eta \in (0,1) \smallsetminus  h^{-1}E(\alpha)$,  where $\nu(h^{-1}E(\alpha))=0$.
		
		Moreover, if $f \in \mathrm{Diff}_r^+(\T)$ for some  $r > 1+\rho(\omega(f))$, then either the Katznelson--Ornstein or the Khanin--Sinai theorem lets us take $h \in \mathrm{Diff}_1^{+}(\T)$, hence $h^{-1}E(\alpha)$ is Lebesgue null. 
	\end{proof}
	
	\subsection{Sturmian sequences}\label{sec: sturmian}  
	Given an infinite word $\omega\in \mathcal{A}^\N$, let $L_n(\omega)$ be the collection of length $n$ subwords (contiguous subsequences) of $\omega$. 
	A sequence $\omega \in \{0,1\}^\N$ is called \textbf{Sturmian} if $|L_n(\omega)| = n+1$ for all $n \in \N$; for background, see Chapter 2 in \cite{lothaire}.
	These sequences can be  characterized using irrational rotations.
	Let $\alpha \in (0,1)$ and $x \in [0,1)$. Define, for $n \in \{0,1,\ldots\}$,
	\begin{equation*}
		\underline{s}_{\alpha,x}(n) \coloneq  \mathbf{1}_{[0,1-\alpha)}(x + n\alpha \bmod 1) \qquad \text{ and } \qquad \overline{s}_{\alpha,x}(n) \coloneq  \mathbf{1}_{(0,1-\alpha]}(x + n\alpha \bmod 1) \, .
	\end{equation*}
	The sequence $\omega$ is Sturmian if and only if    $\omega=\underline{s}_{\alpha,x}$ or $\omega=\overline{s}_{\alpha,x}$
	for some irrational $\alpha$ and some $x \in [0,1)$ (see \cite[Theorem 2.1.13]{lothaire}).
	
	Observe that $L_N(\omega) = \mathcal{W}_N$ with $\beta = 1 - \alpha$. If $\rho(\alpha) = 1$, Theorem~\ref{main thm golden less} then implies 
	\begin{equation*}
		\lim_{N \to \infty} \frac{\log \max_{w,\tilde{w} \in L_N(\omega)} \editfull(w,\tilde{w})}{\log N} \ = \ \frac{1}{2} \, , 
	\end{equation*}
	and, if $\alpha$ is badly approximable, then Proposition~\ref{prop: badly approximable edit distance} yields more precise bounds. 
	In \cite{bpsubstitution}, these bounds are applied to Sturmian sequences that are fixed points of substitutions. When $\rho(\alpha) > 2$, Propositions~\ref{prop: expoliminf special upper bound}~and~\ref{prop: expolimsup special upper bound} yield upper bounds for liminf and limsup exponents (valid since $\beta = 1-\alpha$ belongs to $B_*(\alpha)$ and $B^*(\alpha)$).

	\section{Edit exponents for toral rotations in dimension $d>1$}\label{sec: multidim}
	
	Theorem~\ref{main thm 4} follows from Theorems~\ref{thm: multidim expolimsup upper bound} and \ref{thm: ddimorbitexpoliminf lower bound}, proved below. Both of these proofs make use of the discrepancy theory of the Kronecker sequence $( \modone{n\alpha})_{n\ge 1}$. We prove the remark following Theorem~\ref{main thm 4} in Section~\ref{subsec: nonsing}.
	
	\subsection{Notation conventions}
	
	Let $x = (x_1,\ldots, x_d) \in \R^d$ or $\T^d$ and $r \geq 0$. Write $\leb^d$ for $d$-dimensional Lebesgue measure and $\modone{x}$ for the unique element of $x + \Z^d$ in $[0,1)^d$. Define $\|x\|_\infty \coloneq  \max_{1\leq i \leq d} \|x_i\|$, where $\|x_i\| = \min_{k\in \Z} |x_i - k|$, and also define the cubes
	\[ B_\infty(x,r) \ \coloneq \ \{z \in \T^d : \|z-x\|_\infty < r\} \; \text{ and } \; \bar{B}_\infty(x,r) \ \coloneq \ \{z \in \T^d : \|z-x\|_\infty \leq r\} \, . \] 
	
	Some of the notation below depends on $\alpha$ and $\beta$, but this dependence will be suppressed:
	
	\begin{notation} \label{notation block for multidim}
		Let $d \in \N$ and $\alpha \in \R^d$. Let $\beta \in (0,1)^d$ and write $\square_\beta \coloneq \prod_{i=1}^d [0,\beta_i)$. 
		
		Given $x \in \T^d$, put, for all $n \in \Z$,
		\begin{equation*}
			\xi_n(x)  \ \coloneq \ \mathbf{1}_{\square_\beta}(\modone{x + n\alpha}) \ = \ \prod_{i=1}^d \mathbf{1}_{[0,\beta_i)}(x_i + n\alpha_i \bmod 1) \, .
		\end{equation*}
		Write $\xi_{[i,j)}(x) \coloneq (\xi_\nu(x))_{\nu=i}^{j-1} \in \{0,1\}^{j-i}$. For each $N\in\N$ and $x,y\in\T^d$, as usual, put $\den Nxy \coloneq \editfull(\xi_{[0,N)}(x),\xi_{[0,N)}(y))$ and $\mathcal W_N\coloneq\{\xi_{[0,N)}(x):x\in\T^d\}$. We want to compute
		\begin{equation*}
			\expolimsup(\alpha,\beta) \ \coloneq \ \limsup_{N\to\infty}  \frac{\log \mathrm{diam}_E(\mathcal{W}_N)}{\log N} \quad \text{ and } \quad  \expoliminf(\alpha,\beta) \ \coloneq \ \liminf_{N\to\infty}  \frac{\log \mathrm{diam}_E(\mathcal{W}_N)}{\log N} \, .
		\end{equation*}			
		Now, assume $(1,\alpha_1,\ldots,\alpha_d)$ is linearly independent over $\Q$, so that $x \mapsto \modone{x + \alpha}$ is ergodic. As in the one-dimensional case, the function $\displaystyle \T^d \times \T^d \ni (x,y) \mapsto \limsup_{N\to\infty} \frac{\log \den N x y}{\log N}$ is invariant under the ergodic $\Z^{2}$-action defined, for each $(n,n') \in \Z^{2}$, by $(x,y) \mapsto (\modone{x+n\alpha} , \modone{y + n'\alpha} )$, and this function is thus almost everywhere equal to a constant $\ddimorbitexpolimsup (\alpha,\beta)$. We similarly define $\ddimorbitexpoliminf(\alpha,\beta)$ using $\liminf$ instead of $\limsup$.
	\end{notation}  
	
	\subsection{Upper bound on limsup exponent via approach-and-follow}
	
	Let $J = \prod_{i=1}^d [a_i,b_i) \subseteq \R^d$ with $0 < b_i - a_i \leq 1$ for $1 \leq i \leq d$. Then $I = \modone{J}$ is called a \textbf{box} of $\T^d$ and $\leb^d(I) = \prod_{i=1}^d (b_i - a_i)$. Note $I$ may wrap around the boundary of $\R^d / \Z^d$. 
	
	\begin{definition} Let $N \in \N$ and let $(\mathbf{x}_n)_{n\geq 1}$ be a sequence in $\R^d$. The \textbf{discrepancy} of $(\mathbf{x}_n)$ is
		\[
		D_N((\mathbf{x}_n)_{n\ge1}) \ \coloneq \ \sup_{I \subseteq \T^d} \left| \frac{1}{N} \sum_{n=1}^N \mathbf{1}_I(\modone{\mathbf{x}_n}) - \leb^d(I) \right|,
		\]
		where the supremum is taken over boxes $I \subseteq \T^d$.
	\end{definition}

	Given $\alpha \in \R^d$, let $\Delta_N=\Delta_N(\alpha)  \coloneq N D_{N}((n\alpha)_{n\geq 1})$. Define
	\begin{equation} \label{eq:dlow}
		\lowdisc \ \coloneq \ \{ \alpha \,: \Delta_N(\alpha) \ = \ N^{o(1)} \} \,.
	\end{equation}
	Schmidt in \cite{schmidt} showed that for almost all $\alpha$, we have $\Delta_N(\alpha) = O(\log^{d+1+\eps}(N))$ for every $\eps > 0$, and Beck improved this result in \cite[Theorem 1]{beck}. Therefore $\lowdisc$ has full measure.
	
	For every closed box $B$ (hence every cube $\bar{B}_\infty(x,r)$), we have
	\begin{equation}\label{disc upper}
		\sum_{n=1}^N \mathbf{1}_B(\modone{n\alpha}) \ \leq \ N\leb^d(B) + \Delta_N \, ,
	\end{equation}
	and, for every box (hence cube) $I$,
	\begin{equation}\label{disc lower}
		N\leb^d(I) \ > \ \Delta_N \quad \Rightarrow \quad \sum_{n=1}^N \mathbf{1}_I(\modone{n\alpha}) \ > \ 0 \, .
	\end{equation}
	
	\begin{lemma}\label{lem: multidim orbit cubes cover}
		Let $\alpha\in\R^d$, and put $\lambda(q)\coloneq\frac12\Delta_q^{1/d}$ for $q\in\N$. Then
		\[
		\bigcup_{\ell=1}^{q}\bar B_\infty\Bigl(\ell\alpha,\frac{\lambda(q)}{q^{1/d}}\Bigr)
		\ =\ \T^d \, .
		\]
	\end{lemma}
	\begin{proof}
		Fix $q \in \N$. Given $x \in \T^d$ and $\delta > 0$, let $B(x,\delta) \coloneq \bar{B}_\infty(x, \frac{\lambda(q)}{q^{1/d}}+\delta)$. The result follows from \eqref{disc lower} applied to $B(x,\delta)$.
	\end{proof}

	Like Lemma~\ref{lem: frequent following}, the next lemma bounds the number of coding disagreements of a close pair by estimating the frequency of visits to certain danger zones.

	\begin{lemma}\label{lem: multidim disagreement count}
		Let $\alpha \in \R^d$ and $\beta \in (0,1)^d$. Let $q,N\in\N$ and let $x,y \in \T^d$ satisfy $\|x-y\|_\infty \leq \eps$ with $0 < \eps \leq 1/2$. Then there exists $C_1 = C_1(d) > 0$ such that
		\begin{equation*}
			\#\bigl\{n\in[0,N)\cap\Z:\xi_n(x)\neq\xi_n(y)\bigr\} \ \leq \ C_1\left(1+\frac Nq\right)\left(2q\eps+\Delta_q\right) \, .
		\end{equation*}
	\end{lemma}
	
	\begin{proof}
		For $1\leq i\leq d$ and $c\in\{0,\beta_i\}$, define $S_{i,c} \coloneq \{z \in \T^d : \|z_i-c\|\leq\eps\}$. If $\xi_n(x)\neq\xi_n(y)$, then $\modone{x+n\alpha}\in \cup_{i=1}^d\bigl(S_{i,0}\cup S_{i,\beta_i}\bigr)$. Now, \eqref{disc upper} gives, for every $a \in \Z$,
		\[
		\sum_{n=a}^{a+q-1} \mathbf{1}_{S_{i,c}} (\modone{x+n\alpha}) \ \leq \ q\leb^d(S_{i,c})+\Delta_q \ \leq \ 2q\eps+\Delta_q \, .
		\]
		Partition $[0,N)\cap\Z$ into $\lceil N/q\rceil$ consecutive blocks of length at most $q$, extending the last block to length $q$ if necessary. Summing the preceding bound over these blocks and the $C_1 \coloneq 2d$ sets $S_{i,c}$, we obtain the result.
	\end{proof}
	
	The following proposition is key for the approach-and-follow argument below.
	
	\begin{proposition}\label{prop: multidim direct diameter bound}
		Let $\alpha \in \R^d$ and $\beta \in (0,1)^d$. There exists $C_2 = C_2(d) > 0$ such that, for every $N,q \in \N$,
		\begin{equation*}
			\mathrm{diam}_E(\mathcal W_N)
			\ \leq \
			q+C_2\left(1+\frac Nq\right)q^{(d-1)/d}\Delta_q^{1/d} \, .
		\end{equation*}
	\end{proposition}
	
	\begin{proof}
		Fix $x,y\in\T^d$ and put
		\[
		\eps_q \ \coloneq \ \frac{\lambda(q)}{q^{1/d}} \ = \ \frac12\left(\frac{\Delta_q}{q}\right)^{1/d} \, .
		\]
		By Lemma~\ref{lem: multidim orbit cubes cover}, there exists $\ell \in [1,q] \cap \Z$ such that $\|x+\ell\alpha-y\|_\infty\leq\eps_q$.
		Using \eqref{hamming bound} and then Lemma~\ref{lem: multidim disagreement count}, we obtain, since $\Delta_q\leq q^{(d-1)/d}\Delta_q^{1/d}$,
		\begin{align*}
			\den N x y \ \leq \ q+C_1\left(1+\frac Nq\right)\left(2q\eps_q+\Delta_q\right) \ \leq \ q+2C_1 \left(1+\frac Nq \right) q^{(d-1)/d}\Delta_q^{1/d} \, .
		\end{align*}
		Taking the maximum over $x,y \in \T^d$ proves the proposition with $C_2 \coloneq 2C_1$.
	\end{proof}
	
	The next theorem gives an upper bound, via approach-and-follow, in terms of the growth rate of the discrepancy of $\alpha$.
	
	\begin{theorem}\label{thm: multidim expolimsup upper bound}
		Let $\alpha \in \R^d$ and $\beta \in (0,1)^d$. Let
		\[
		\delta(\alpha) \ \coloneq \ \limsup_{q\to\infty} \frac{\log \Delta_q}{\log q} \, .
		\]
		Then $\expolimsup(\alpha,\beta) \leq \frac{d}{d+1-\delta(\alpha)}$.
	\end{theorem}
	
	\begin{proof}
		If $\delta(\alpha) = 1$, then the claimed bound is trivial, so assume $\delta(\alpha) \in [0,1)$. Fix $\eps > 0$ such that $\delta(\alpha)+\eps < 1$ and let $q_0\in\N$ such that
		$\Delta_q\leq q^{\delta(\alpha)+\eps}$ for every $q\geq q_0$. Put
		\[
		a \ \coloneq \ \frac{1-\delta(\alpha)-\eps}{d} \ > \ 0
		\qquad\text{and}\qquad
		\theta \ \coloneq \ \frac{1}{1+a} \ = \ \frac{d}{d+1-\delta(\alpha)-\eps} \, .
		\]
		For sufficiently large $N$, let $q\coloneq\lceil N^\theta\rceil$, so that
		$q\geq q_0$. By Proposition~\ref{prop: multidim direct diameter bound} and then $\Delta_q\leq q^{\delta(\alpha)+\eps}$, we see $\mathrm{diam}_E(\mathcal W_N)$ is at most
		\begin{equation*}
			q+C_2\left(1+\frac Nq\right)q^{(d-1)/d}\Delta_q^{1/d} \ \leq \ q+C_2\left(1+\frac Nq\right)q^{1-a} \ \leq \  (1+C_2)q+C_2Nq^{-a} \, .
		\end{equation*}
		Since $N^\theta \leq q \leq 2N^\theta$ and $1-a\theta=\theta$, it follows that
		\[
		\mathrm{diam}_E(\mathcal W_N) \ \leq \ \bigl(2+3C_2\bigr)N^\theta
		\]
		for all sufficiently large $N$, hence
		\[
		\expolimsup(\alpha,\beta) \ \leq \ \theta \ = \ \frac{d}{d+1-\delta(\alpha)-\eps} \, .
		\]
		Letting $\eps \downarrow 0$ proves the result.
	\end{proof}

	\subsection{Lower bound for the liminf typical orbit exponent}
	
	The following lemma bounds the probability that two random points share a coding in terms of the discrepancy of $\alpha$.
	\begin{lemma}\label{lem: multidim same coding bound}
		Let $\alpha \in \R^d$ and $\beta \in (0,1)^d$, and put $C_3 \coloneq \prod_{j=1}^d \beta_j^{-1}$. Suppose $k\in\N$ satisfies
		\begin{equation}\label{eqn: small discrepancy condition}
			\frac{C_3\Delta_k(\alpha)}{k} \ < \ \min\{\|\beta_1\|,\ldots,\|\beta_d\|\} \, .
		\end{equation}
		Then, for all $x,y\in\T^d$,
		\begin{equation}\label{eqn: same code implies close}
			\xi_{[0,k)}(x) = \xi_{[0,k)}(y)
			\quad\Longrightarrow\quad
			\|x-y\|_\infty \ \leq \ \frac{C_3\Delta_k}{k} \, .
		\end{equation}
		Put $C_4\coloneq2^dC_3^d$. Consequently, if $x,y$ are independent and uniform in $\T^d$, then
		\begin{equation*}
			\mathbb P\bigl(\xi_{[0,k)}(x)=\xi_{[0,k)}(y)\bigr)
			\ \leq\ C_4\left(\frac{\Delta_k}{k}\right)^d \, .
		\end{equation*}
	\end{lemma}
	
	\begin{proof}
		Put $\eta \coloneq \frac{C_3\Delta_k}{k}$. We first prove \eqref{eqn: same code implies close}. Suppose that $\|x-y\|_\infty>\eta$. Then there is some $1\leq i\leq d$ such that
		\[
		\eta \ < \ \|x_i-y_i\| \ \leq \ \frac 12 \, .
		\]
		Without loss of generality, assume that $x_i>y_i$.
		
		Define $B\coloneq\prod_{j=1}^dJ_j$, where
		\[
		J_j \ \coloneq \ [0,\beta_j) \quad \text{ for } j \neq i,
		\qquad J_i \ \coloneq \ \min\{x_i-y_i,\beta_i\}+[-\eta,0) \, .
		\]
		By \eqref{eqn: small discrepancy condition}, $\eta < \|\beta_i\|\leq\beta_i$, and since $x_i-y_i\geq\|x_i-y_i\|>\eta$, we have $J_i\subset(0,\beta_i)$. Moreover, $k\leb^d(B) =  \Delta_k/\beta_i > \Delta_k$. Hence, by \eqref{disc lower}, there exists $n \in [0,k) \cap \Z$ such that
		\[
		\modone{x+n\alpha} \in B \subset \square_\beta \, .
		\]
		We claim that $\modone{y+n\alpha}\notin\square_\beta$. Indeed, since $x_i+n\alpha_i\bmod 1 \in J_i$, for some $\gamma\in(0,\eta]$,
		\[
		y_i+n\alpha_i\bmod 1 \ \equiv \ \min\{x_i-y_i,\beta_i\}-\gamma+1+y_i-x_i \bmod 1 \, .
		\]
		If $\min\{x_i-y_i,\beta_i\}=x_i-y_i$, then $\gamma \leq \eta < \|\beta_i\| \leq 1-\beta_i$, and hence
		\[
		y_i+n\alpha_i\bmod1 \ = \ 1-\gamma \ \geq \ \beta_i \, .
		\]
		If instead $\min\{x_i-y_i,\beta_i\} = \beta_i$, then $\beta_i+y_i-x_i \leq 0$, so
		\[
		(\beta_i-\gamma)+(1+y_i-x_i) \ < \ 1 \, .
		\]
		Moreover, $\gamma \leq \eta < \|x_i-y_i\| \leq 1+y_i-x_i$, and therefore
		\[
		y_i+n\alpha_i \bmod1 \ = \ (\beta_i-\gamma)+(1+y_i-x_i) \ \geq \ \beta_i \, .
		\]
		Thus $\xi_n(x)\neq\xi_n(y)$, proving \eqref{eqn: same code implies close}. By \eqref{eqn: small discrepancy condition}, $\eta < 1/2$, so for every $x \in \T^d$, the set of $y\in\T^d$ satisfying $\|x-y\|_\infty\leq\eta$ has Lebesgue measure $(2\eta)^d = C_4(\frac{\Delta_k}{k})^d$. Integrating with respect to $x$ and using \eqref{eqn: same code implies close} gives the final result.
	\end{proof}
	
	We now prove the lower bound.
	
	\begin{theorem}\label{thm: ddimorbitexpoliminf lower bound}
		Let $\alpha \in \R^d$ and $\beta \in (0,1)^d$. Let
		\[
		\delta(\alpha) \ \coloneq \ \limsup_{q\to\infty} \frac{\log \Delta_q}{\log q} \ < \ 1 \, .
		\]
		Then \[ \ddimorbitexpoliminf(\alpha,\beta) \ \geq \ \frac{d(1-\delta(\alpha))}{1+d(1-\delta(\alpha))} \, . \]
	\end{theorem}
	
	\begin{proof}
		First, note that $\delta(\alpha) < 1$ implies $(1,\alpha_1,\ldots,\alpha_d)$ are linearly independent over $\Q$, so $\ddimorbitexpoliminf(\alpha,\beta)$ is defined.
		Put $s\coloneq d(1-\delta(\alpha))>0$. Since $\Delta_k\leq k^{\delta(\alpha)+o(1)}$, condition~\eqref{eqn: small discrepancy condition} holds for all sufficiently large $k$. Thus, for independent uniform $x,y\in\T^d$, Lemma~\ref{lem: multidim same coding bound} gives
		\[
		\phi_k\coloneq\mathbb P\bigl(\xi_{[0,k)}(x)=\xi_{[0,k)}(y)\bigr)	\ \leq \ C_4\left(\frac{\Delta_k}{k}\right)^d
		\ \leq\ k^{-s+o(1)} \, .
		\]
		Applying Corollary~\ref{utility cor for edit exponent lower bound}(a) proves the theorem.
	\end{proof}

	\begin{proof}[Proof of Theorem~\ref{main thm 4}]
		Note $\alpha \in \lowdisc$ implies $\delta(\alpha) = 0$. Apply Theorems~\ref{thm: multidim expolimsup upper bound} and \ref{thm: ddimorbitexpoliminf lower bound}.
	\end{proof}

	\subsection{Bounds for non-singular rotations}\label{subsec: nonsing}

	Given $\alpha \in \R^d$, consider the irrationality measure function
	\begin{equation*}
		\psi_\alpha(q) \ \coloneq \ \min_{n \in [1,q] \cap \Z}  \| n\alpha \|_\infty \, .
	\end{equation*}
	
	First, for each $q \in \N$ and open cube $B \coloneq B_\infty(z,\psi_\alpha(q)/2)$, we have $\sum_{n=0}^q \mathbf{1}_{B}(\modone{ n\alpha}) \leq 1$.

	Second, by Dirichlet's theorem, $\psi_{\alpha}(q) \leq q^{-1/d}$ for all $q \geq 1$.
	We say that $\alpha \in \R^d$ is \textbf{non-singular} if
	\begin{equation*}
		\upsilon(\alpha) \ \coloneq \ \limsup_{q \to \infty} q^{1/d}\psi_\alpha(q) \ > \ 0 \, .
	\end{equation*}
	
	Khintchine \cite{khintchine26} proved that the set of singular $\alpha$ is null, and, when $d > 1$, Cheung and Chevallier \cite{cheuchev} showed it has Hausdorff dimension $\frac{d^2}{d+1}$. Using a transference theorem relating homogeneous forms and inhomogeneous problems, we derive the following lemma.
	
	\begin{lemma}\label{lem: nonsing condition}
		Let $\alpha \in \R^d$ be non-singular, and define the infinite set
		\[
		A\ \coloneq\ \{q \in \N : \psi_\alpha(q) > \tfrac 12 \upsilon(\alpha)q^{-1/d}\} \, .
		\]
		There exist constants $C_5 > 0$ and $C_6 \in \N$ depending only on $\upsilon(\alpha)$ and $d$ such that, for all $Q \in C_6A$,
		\begin{equation}\label{nonsing condition}
			\bigcup_{n = -Q}^Q \bar{B}_\infty\Bigl(n\alpha,\frac{C_5}{Q^{1/d}}\Bigr) \ = \ \T^d \, . 
		\end{equation}
	\end{lemma}
	\begin{proof}
		For $q \in A$, no nonzero $n \in (-q,q) \cap \Z$ satisfies the inequality $\| n \alpha\|_{\infty} \leq \tfrac 12 \upsilon(\alpha) q^{-\frac 1d}$, so by Theorem VI and its corollary in \cite[Chapter V]{cassels}, there exist $\tilde{C}_5,\tilde{C}_6 > 0$ depending only on $\upsilon(\alpha)$ and $d$ such that, for all $z \in \T^d$, 
		\begin{equation*}
			\|n\alpha - z\|_\infty \ \leq \  \tilde{C}_5 q^{-1/d}
		\end{equation*}
		for some integer $n$ with $|n| \leq \tilde{C}_6q$. Then take $C_6 \coloneq \lceil \tilde{C}_6 \rceil$ and $C_5 \coloneq \tilde{C}_5 C_6^{1/d}$.
	\end{proof}
	
	We will also need a lemma that bounds the number of coding disagreements of a close pair of points by estimating the frequency of visits to certain danger zones. First, we need this basic geometric lemma:
	
	\begin{lemma}\label{lem: multidim frequent following}
		Let $\alpha\in\R^d$, $\beta\in(0,1)^d$, and $n\in\Z$. If $0<\eps\leq1/2$ and $x,y\in\T^d$ satisfy $\|x-y\|_\infty \leq \eps$, then $\xi_n(x) \neq \xi_n(y)$ only if $\modone{x+n\alpha}$ belongs to a set that can be covered by at most $2d\eps^{1-d}$ cubes of the form $\bar B_\infty(\zeta,\eps)$.
	\end{lemma}
	
	\begin{proof}
		If $\xi_n(x) \neq \xi_n(y)$, then $\modone{x+n\alpha}$ is within $\|\cdot\|_\infty$-distance $\eps$ of the boundary of $\square_\beta$. The boundary of $\square_\beta$ has $2d$ faces. The $\eps$-neighborhood of each face is a box with one side of length at most $2\eps$ and all other side lengths at most $1$. Since $\eps\leq1/2$, each such neighborhood can be covered by at most $\eps^{1-d}$ cubes of radius $\eps$.
	\end{proof}
	
	\begin{lemma}\label{lem: multidim nonsing frequent following} 
		Let $\alpha\in\R^d$ be non-singular, and fix $\beta\in(0,1)^d$. Suppose $C_7,C_8>0$ are constants. Then there exists a constant $C_{10} > 0$ depending only on $d$, $C_7$, and $C_8$ such that, for all $x,y \in \T^d$ and $q \in \N$ such that $\| x - y\|_\infty \leq C_7q^{-1/d}$ and $q^{-1/d} \leq C_8\psi_\alpha(q)$,
		\begin{equation*}
			\#\bigl\{ n \in [0,q]\cap\Z : \xi_n(x) \neq \xi_n(y) \} \ \leq \ C_{10} q^{(d-1)/d} \, .
		\end{equation*}
	\end{lemma}
	\begin{proof}
		First suppose $q\geq(2C_7)^d$, and put $\eps\coloneq C_7q^{-1/d}\leq1/2$. By Lemma~\ref{lem: multidim frequent following}, every disagreement occurs when $\modone{x+n\alpha}$ belongs to a union of at most $2d\eps^{1-d}$ cubes $\bar B_\infty(\zeta,\eps)$.
		Since $\eps\leq C_7C_8\psi_\alpha(q)$, each such cube can be covered by at most $C_9$ open cubes $B_\infty(z,\psi_\alpha(q)/2)$, where $C_9$ depends only on $d,C_7,C_8$. Each open cube contains at most one of the points $\modone{x+n\alpha}$, $0\leq n\leq q$, because two such points have distance at least $\psi_\alpha(q)$. Thus the number of disagreements is at most
		\[
		2d\eps^{1-d}C_9\ =\ 2dC_7^{1-d}C_9q^{(d-1)/d} \, .
		\]
		For the remaining integers $q<(2C_7)^d$, the trivial bound is
		\[
		q+1\ \leq\ 2q\ <\ 4C_7q^{(d-1)/d} \, .
		\]
		Hence $C_{10}\coloneq\max\{2dC_7^{1-d}C_9,4C_7\}$ proves the lemma for every $q$.
	\end{proof}
	
	We are now ready to derive an upper bound.
	
	\begin{proposition}\label{prop: nonsing expoliminf upper bound}
		Let $\alpha \in \R^d$ be non-singular. For all $\beta \in (0,1)^d$,
		\begin{equation*}
			\expoliminf (\alpha,\beta) \ \leq \ \frac{d}{d+1} \, .
		\end{equation*}  
	\end{proposition}
	\begin{proof}
		Let $A$ be as in Lemma~\ref{lem: nonsing condition}. Then there exist constants $C_5 > 0$ and $C_6 \in \N$ depending only on $\upsilon(\alpha)$ and $d$ such that \eqref{nonsing condition} holds for all $Q \in C_6A$.
		
		Let $x,y\in\T^d$ and $q\in A$, and write $Q\coloneq C_6q$.
		Put $C_7\coloneq C_5C_6^{-1/d}$ and $C_8\coloneq(\tfrac12\upsilon(\alpha))^{-1}$.
		Let $N\coloneq\lfloor q^{(d+1)/d}\rfloor$. By \eqref{nonsing condition}, and switching the role of $x$ and $y$ if necessary, choose $n_0 \in [0,Q] \cap \Z$ so that $\tilde{x} \coloneq x + n_0\alpha$ satisfies
		\begin{equation*}
			\|\tilde{x} - y\|_\infty \ \leq \ \frac{C_5}{Q^{1/d}} \ = \ C_7q^{-1/d}\, ,
		\end{equation*}
		Since $q\in A$, we also have $q^{-1/d}\leq C_8\psi_\alpha(q)$. Let $k \coloneq q + 1$. Hence, by Lemma~\ref{lem: multidim nonsing frequent following} applied to the points $\modone{\tilde{x} + ik\alpha}$ and $\modone{y+ik\alpha}$ for $i \in [0,\lceil \frac{N}{k}\rceil - 1] \cap \Z$, there exists a constant $C_{10} = C_{10}(\upsilon(\alpha),d) > 0$ such that
		\begin{equation*}
			\#\bigl\{ n \in [0,N) : \xi_{n}(\tilde{x}) \neq \xi_n(y) \} \ \leq \ C_{10}\left\lceil \frac{N}{k} \right\rceil q^{(d-1)/d}\, ,
		\end{equation*}
		Put $C_{11}\coloneq2C_{10}$. Since $N\geq q$, we have
		$\lceil N/(q+1)\rceil\leq2N/q$, so the preceding bound is at most
		$C_{11}Nq^{-1/d}\leq C_{11}q$. By \eqref{hamming bound},
		\begin{equation*}
			\den N x y \ \leq \ C_6q + C_{11} q \, .
		\end{equation*}
		
		This bound is independent of $x$ and $y$. Therefore,
		\begin{equation*}
			\frac{\log \mathrm{diam}_E(\mathcal{W}_N)}{\log N} \ \leq \ \frac{\log(C_6 +C_{11})}{\log N} + \frac{\log q}{\log N} \,. 
		\end{equation*}
		The result follows since $A$ is infinite. 
	\end{proof}
	
	Note that non-singularity implies that $(1,\alpha_1,\ldots,\alpha_d)$ is linearly independent over $\Q$, so the typical orbit exponent in the following proposition is well defined.
	
	\begin{proposition}\label{prop: nonsing typical limsup lower bound}
		Let $\alpha \in \R^d$ be non-singular and let $\beta \in (0,1)^d$. Then
		\[
		\ddimorbitexpolimsup(\alpha,\beta) \ \geq \ \frac 12 \, .
		\]
	\end{proposition}
	
	\begin{proof}
		Let $A$, $C_5$, and $C_6$ be as in Lemma~\ref{lem: nonsing condition}. For $q\in A$, put $Q\coloneq C_6q$ and $r\coloneq C_5Q^{-1/d}$. Translating \eqref{nonsing condition} by $Q\alpha$ shows that every closed cube of side length $2r$ contains one of the points $\modone{n\alpha}$ with $0\leq n<k \coloneq 2Q+1$.
		
		Take $q$ sufficiently large so that $3r<\min_i\|\beta_i\|$. If $\|x-y\|_\infty>3r$, then choose $i$ with $\|x_i-y_i\|>3r$. As in the proof of Lemma~\ref{lem: multidim same coding bound}, the interval difference in coordinate $i$ contains a half-open interval of length $3r$ on which the $x$-coordinate belongs to $[0,\beta_i)$ and the $y$-coordinate does not. Its product with the other intervals $[0,\beta_j)$ is a distinguishing box, and every side of this box has length greater than $2r$. Thus it contains a closed cube of   side length $2r$. The covering property forces the $x$-orbit to visit this box before time $k$, showing $\xi_{[0,k)}(x) \neq \xi_{[0,k)}(y)$. Therefore
		\[
		\xi_{[0,k)}(x)=\xi_{[0,k)}(y)\quad\Longrightarrow\quad\|x-y\|_\infty\leq3r \, .
		\]
		For independent uniform $x,y$, put $\phi_k\coloneq\mathbb P\bigl(\xi_{[0,k)}(x)=\xi_{[0,k)}(y)\bigr)$. We deduce that
		\[
		\phi_k\ \leq\ (6r)^d \ = \ \frac{(6C_5)^d}{Q} \ \leq \ \frac{3(6C_5)^d}{k} \, .
		\]
		Along the infinite sequence of $k$ obtained from $q\in A$, this yields $\phi_k = O(k^{-1})$. Corollary~\ref{utility cor for edit exponent lower bound}(b), with $s=1$, proves the proposition.
	\end{proof}

	\section{Open problems} \label{sec: questions}
	
	\noindent\textbf{A variational principle for edit-distance exponents.}
	Let $\mathcal{A}$ be a finite alphabet, let $T$ be a minimal homeomorphism of a compact metric space $X$, and let $\xi : X \to \mathcal{A}$ define a Borel partition. Given $x \in X$ and $N \in \N$, define the $\xi$-name of length $N$ associated to $x$ by
	\[
	\xi_{[0,N)}(x) \ \coloneq \ \bigl(\xi(T^nx) \bigr)_{n=0}^{N-1} \, , \quad \text{ and put } \; \, d_N^{\xi}(x,y) \ \coloneq \ \editfull\bigl(\xi_{[0,N)}(x), \xi_{[0,N)}(y)\bigr) \, .
	\]
	Let $\nu \in \mathcal M_e(T)$, the set of ergodic $T$-invariant probability measures, and define  
	\begin{equation*}
		\overline{\Gamma}(T,\xi)
		\coloneq  \limsup_{N\to\infty} \frac{\log (1+\max_{x,y\in X}d_N^{\xi}(x,y) )}{\log N} \quad \text{and} \quad
		\overline{\Gamma}_{\nu}(T,\xi) \coloneq  \limsup_{N\to\infty} \frac{\log(1+ d_N^{\xi}(x,y))}{\log N} \,,   
	\end{equation*}
	the latter being $(\nu\times \nu)$-a.e.\ constant. Clearly,  
	\begin{equation}\label{easy half of edit exponent variational principle}
		\sup_{\nu\in\mathcal M_e(T)} \overline{\Gamma}_{\nu}(T,\xi) \ \leq \ \overline{\Gamma}(T,\xi) \, .
	\end{equation}
	When does equality hold in \eqref{easy half of edit exponent variational principle}? Theorems~\ref{main thm golden less}~and~\ref{main thm 4} give a positive answer for all irrational rotations on $\T$ with $\rho(\alpha) \leq \varphi$ and for almost every rotation of a torus, provided $\xi$ is the indicator of a box. Some assumptions on the boundaries of the partition determined by $\xi$ are certainly needed.
	In view of Theorem~\ref{main thm 1}, the corresponding question for liminf edit exponents has a negative answer for some irrational rotations on $\T$.
	
	\medskip
	\noindent\textbf{Which values can edit-distance exponents attain?} For limsup exponents, we know by Theorem~\ref{main thm 1} that every value in $[1/2,1]$ is attained. For specific classes of systems, such as diffeomorphisms of tori or nilrotations, which values can be attained? If the edit-distance exponent exists as a limit, which values can be attained? (By Theorem \ref{main thm 4}, this set contains $\{d/(d+1) : d \in \N\}$.) These questions should be restricted to a suitable class of generating Borel partitions, to prevent trivial counterexamples.
	
	\medskip
	\noindent\textbf{Non-singular toral rotations.}
	For a non-singular $\alpha \in \R^d$, Section~\ref{sec: multidim} proves, for every $\beta \in (0,1)^d$,
	\[
	\expoliminf(\alpha,\beta) \ \leq \ \frac{d}{d+1} \qquad \text{and} \qquad \ddimorbitexpolimsup(\alpha,\beta) \ \geq \ \frac 12 \, .
	\]
	When $\alpha$ has low discrepancy, all four diameter and typical-orbit exponents equal $d/(d+1)$. What are the four exponents for an arbitrary non-singular vector $\alpha$?

	\medskip 
	\noindent \textbf{Acknowledgment.} \quad We are grateful to Bohan Yang for helpful discussions. The research of Y. Peres was supported by National Natural Science Foundation of China grant RFIS-W2531011. This project was started in January 2024, when the authors participated in a workshop at the Tsinghua--Sanya International Mathematics Forum. The proof of Theorem~\ref{main thm thresholds} was obtained in part through interaction with LLMs.

\end{document}